\documentclass[reqno, a4paper,11pt]{amsart}

\usepackage{fullpage}

\usepackage{amssymb}

\usepackage{amsfonts,mathrsfs}
\usepackage{amsmath}
\usepackage[enableskew]{youngtab}
\usepackage{mathabx}
\usepackage{graphicx}
\usepackage[all]{xy}

\usepackage{multirow}

\usepackage{amsthm}
\newtheorem{theorem}{Theorem}[section]
\newtheorem{lemma}[theorem]{Lemma}
\newtheorem{proposition}[theorem]{Proposition}
\newtheorem{corollary}[theorem]{Corollary}

\theoremstyle{definition}
\newtheorem{definition}[theorem]{Definition}

\newtheorem{remark}[theorem]{Remark}

\newtheorem{example}[theorem]{Example}

\theoremstyle{remark}
\numberwithin{theorem}{subsection}

\newcommand{\Si}{\mathfrak{S}}
\renewcommand{\hom}{\mathrm{hom}}
\newcommand{\Hom}{{\mathrm{Hom}}}

\newcommand{\End}{{\mathrm{End}}}

\renewcommand{\k}{\Bbbk}

\newcommand{\Fct}{\mathrm{Fct}}

\newcommand{\ev}{\mathrm{ev}}
\newcommand{\I}{\mathsf{I}}
\newcommand{\II}{\mathsf{II}}
\newcommand{\M}{\texttt{M}}
\newcommand{\Q}{\texttt{Q}}

\newcommand{\bs}{\boldsymbol}

\newcommand{\Sym}{\text{Sym}}

\newcommand{\C}{\mathcal{C}}
\newcommand{\D}{\mathcal{D}}

\newcommand{\U}{\mathcal{U}}
\newcommand{\W}{\mathcal{W}}
\newcommand{\Xc}{\mathcal{X}}
\newcommand{\Yc}{\mathcal{Y}}
\newcommand{\Zc}{\mathcal{Z}}
\renewcommand{\P}{\mathcal{P}}

\newcommand{\equi}{\stackrel{\sim}{\longrightarrow}}

\newcommand{\Z}{\mathbb{Z}}

\newcommand{\V}{\mathcal{V}}

\newcommand{\smod}{\mathrm{smod}}
\newcommand{\svec}{\mathrm{svec}}
\newcommand{\Smod}{\mathrm{Smod}}
\newcommand{\Svec}{\mathrm{Svec}}

\newcommand{\vvec}{\mathrm{vec}}

\newcommand{\Pol}{\mathsf{Pol}}

\newcommand{\sdim}{\mathrm{sdim}}

\newcommand{\hp}{\hspace{.02cm}}
\newcommand{\hpp}{\hspace{.05cm}}
\newcommand{\hppp}{\hspace{.2cm}}
\newcommand{\hm}{\hspace{-.02cm}}
\newcommand{\hmm}{\hspace{-.05cm}}

\title[]{On Schur Superfunctors}
\author[]{Jonathan~Axtell }
\address{Department of Mathematics, Seoul National University, 599 Gwanak-ro, Gwanak-gu, Seoul 151-747, Korea}
\email{jdaxtell@snu.ac.kr}
\thanks{This work was supported by the BRL research fund grant \#2013055408 of the National Research Foundation of Korea}
\date{\today}
\begin{document}


\maketitle

\begin{abstract}
We introduce super-analogues of the Schur functors defined by Akin, 
Buchsbaum and Weyman. 
These {\em  Schur superfunctors} 
may be viewed as characteristic-free analogues of the  finite dimensional 
atypical irreducible modules over the Lie superalgebra 
$\mathfrak{gl}(m|n)$  studied by Berele and Regev. 
Our construction realizes Schur superfunctors as objects of 
a certain category of strict polynomial superfunctors.
We show that Schur superfunctors are indecomposable objects of this category.
Another aim is to provide a decomposition of {\em  Schur bisuperfunctors} 
in terms of tensor products of Schur superfunctors.  
\end{abstract}
\medskip

\setcounter{tocdepth}{1}
\tableofcontents


\section*{Introduction}
Let $\k$ be an arbitrary field. 
Then it is well known that irreducible polynomial representations of the general linear 
group scheme, $GL_n$, 
correspond to partitions,  $\lambda$, of length smaller than or equal to 
$n$. Furthermore, each such irreducible may be obtained as the simple head (or cosocle) 
of a 
corresponding {\em Schur} (or {\em costandard}) {\em module}, $S_\lambda(\k^n)$, 
which is essentially an induced module 
from a Borel subgroup of $GL_n$, as defined in \cite{Ja}. (See also \cite{Green} %
for a definition in terms of Schur algebras.) 
%
In \cite{ABW}, Akin, Buchsbaum and Weyman gave an explicit construction of these 
Schur modules which
surprisingly uses only rudimentary multilinear algebra stated in terms of familiar Hopf algebras.
The goal of this paper is to define and study super-analogues of this construction.  

More specifically, let $V$ be a finite dimensional $\k$-vector space, and consider the Hopf algebra $D^{} V$ 
(resp.~$\Lambda^{} V$, $\Sym^{} V$) of divided (resp.~exterior, symmetric) powers of $V$.  Also, given a 
partition $\lambda=(\lambda_1, \dots, \lambda_t)$, let $\lambda'$ denote its transpose.  
Then in \cite{ABW}, the Schur module $S_\lambda(V)$ is defined to be the image of a certain linear map 
$$\theta_{\lambda}(V): \Lambda^{\lambda'}V \to V^{\otimes d} \to \Sym^{\lambda}V,$$
where $\Lambda^{\lambda'}(V)$, $\Sym^{\lambda}(V)$ are tensor products of exterior and symmetric 
powers, respectively.  Dually, the Weyl (or co-Schur) module $W_\lambda(V)$ is defined as the
image of a composition
$$\theta'_\lambda(V):  D^{\lambda'}V \to V^{\otimes d} \to  \Lambda^{\lambda}V.$$
In both cases, the linear maps are obtained by composing a permutation of the multiplication with 
iterations of the comultiplication in the respective Hopf algebras.
%

Now let $\P_d= \P_{d,\k}$ denote the category of (homogeneous) strict polynomial functors defined by Friedlander and Suslin 
in \cite{FS}.  Let $\vvec$ denote the category of finite dimensional $\k$-vector spaces, and let $\mathfrak{sch}$ denote the
category of all schemes over $\k$.   Then objects of $\P_d$ may be considered as (homogeneous) 
{\em $\mathfrak{sch}$-enriched}
endofunctors of $\vvec$.  In particular, the action on hom-spaces may be 
described in terms of polynomial equations which 
are homogeneous of degree $d$.  We recall in Section \ref{ss:poly} the equivalent but somewhat
simpler ``linearized" definition of the category $\P_d$.
Typical examples of strict polynomial functors are $D^d$, $\Lambda^d$, and $\Sym^d$ which
send a vector space to its $d$-th divided, exterior, or symmetric power, respectively.  Since there is a
tensor product bifunctor, $-\otimes-: \P_{d_1}\times \P_{d_2} \to \P_{d_1+d_2}$, we may also consider the
objects $D^{\lambda},\Lambda^{\lambda},\Sym^\lambda \in \P_d$, where $\lambda$ is a partition such
that $|\lambda|:= \sum\lambda_i = d$.  Evaluation of $T\in \P_d$ at $V\in \vvec$ gives rise to a polynomial $GL(V)$-module,
$T(V)$.  Furthermore, if $n\geq d$ then $\P_d$ is equivalent to the category of all finite dimensional polynomial
$GL_n$-modules which are homogeneous of degree $d$.
%

The naturality of the construction in \cite{ABW} allows a straightforward generalization to the functorial setting.
I.e., the linear maps $\theta_\lambda(V)$ and $\theta\hp'_\lambda(V)$ yield natural transformations which are composites
of morphisms
$$ \Lambda^{\lambda'} \hookrightarrow I^{\otimes d} \twoheadrightarrow \Sym^{\lambda}
\quad \text{ and }\quad
D^{\lambda'}\hookrightarrow I^{\otimes d} \twoheadrightarrow \Lambda^{\lambda}$$
in the category $\P_d$, respectively, where $I^{\otimes d}$ is the $d$-th tensor product functor.  Since $\P_d$ 
is abelian, one may define the {\em Schur functor}, $S_\lambda$, to be 
the image of $\theta_\lambda$, and the {\em Weyl functor}, $W_\lambda$, is the image of $\theta'_\lambda$ in $\P_d$.
%

In \cite{Axtell}, the author defined categories $\Pol^\I_d$, $\Pol^\II_d$ whose
objects are {\em polynomial superfunctors} of types $\I$ and $\II$.  
If $m,n\ge d$, then $\Pol^\I_d$, $\Pol^\II_d$ are equivalent to the category 
of finite dimensional (degree $d$ homogeneous) polynomial  representations over the algebraic 
supergroups $GL(m|n)$,  $Q(n)$, respectively (see  \cite{BrKu}, \cite{BrKl2} for the definitions of
these supergroups).  In this paper, we will mostly be concerned with the category
$\Pol^\I_d$ of type $\I$ polynomial superfunctors.  Some results, however, can be transferred
to the category $\Pol^\II_d$ via a certain restriction functor, $\Pol^\I_d\to \Pol^\II_d$, described 
in Section \ref{ss:superpoly} below.
 %

Typical examples of objects in $\Pol^\I_d$ are given by $I^{\otimes d}, \Gamma^d, S^d$ 
which are super analogues of the tensor, divided, and symmetric powers considered above.  
Since tensor products of polynomial superfunctors are well-defined, we may also consider
$\Gamma^\lambda, S^\lambda \in \Pol^\I_d$, for any partition $\lambda$ such that $|\lambda| =d$.
We will also need to consider the (right) parity change functor $\Pi \in \Pol^\I_1$.  Given any
$T\in \Pol^\I_d$, we may then consider both left and right compositions, 
$\Pi\circ T$ and $T\circ \Pi$, as objects in $\Pol^\I_d$.
%

Now suppose $\lambda$ is a partition with $|\lambda|=d$.
We define the Schur superfunctor $\widehat{S}_\lambda$ in Section \ref{ss:superSchur} as the
object of $\Pol^\I_d$ which is the image of a certain natural transformation 
$\hat{\theta}_{\lambda}$. This morphism $\hat{\theta}_{\lambda}$ is given as a composite of natural 
transformations of the form
$$ \Pi \circ \Gamma^{\lambda'} \circ \Pi \hookrightarrow\ I^{\otimes d} \twoheadrightarrow\ S^{\lambda} \qquad \quad
(\Gamma^{\lambda'} \circ \Pi \hookrightarrow\ I^{\otimes d} \twoheadrightarrow\  S^{\lambda}),$$
if $d$ is odd (even).  
Similarly, we define the Weyl superfunctor $\widehat{W}_\lambda$ to be the image of a 
natural transformation $\check{\theta}_{\lambda}$ given by a composite of morphisms of the form
$$  \Gamma^{\lambda'} \hookrightarrow\ I^{\otimes d} \twoheadrightarrow\ \Pi \circ  S^{\lambda}\circ \Pi \qquad\quad 
(\Gamma^{\lambda'}  \hookrightarrow\ I^{\otimes d} \twoheadrightarrow\  S^{\lambda}\circ \Pi),$$
if $d$ is odd (even).  
More generally, we will define skew Schur and Weyl superfunctors, $\widehat{S}_{\lambda/\mu}$ 
and $\widehat{W}_{\lambda/\mu}$,  associated to any skew partition, $\lambda/\mu$. 
%

In \cite{ABW}, the authors also studied {\em Schur complexes}, which they used 
to obtain resolutions of certain polynomial $GL_n$-modules.  
We consider these Schur complexes as {\em strict polynomial bifunctors} $SC_{\lambda/\mu}(\ ,\ )$
associated to skew partitions $\lambda/\mu$ (cf.   Definition \ref{def:SchurFunct}).  
Evaluating $SC_{\lambda/\mu}$ at a pair of vector spaces $V,W$ gives 
(the underlying vector space of) a corresponding Schur complex.  We will show
in Proposition \ref{prop:SchurComplex} that 
 $SC_{\lambda/\mu}(V,W)$ is also isomorphic to (the underlying ordinary vector space of)
the evaluation of the Schur superfunctor $\widehat{S}_{\lambda/\mu}$ at a certain vector superspace.

In this  paper, we also give analogues of several results obtained in 
\cite{ABW} for Schur functors and Schur complexes.  
%
For example, we provide a Standard Basis Theorem (Theorem \ref{thm:standard}) 
for Schur superfunctors using a certain ``Straightening Law" (Lemma \ref{lem:straighten}).  
We also consider categories $\mathsf{biPol}^\I_d$, $\mathsf{biPol}^\II_d$ of 
{\em strict polynomial bisuperfunctors}  in Section \ref{Sec:bisuper}.
After defining the {\em Schur bisuperfunctors}, 
$\widehat{S}^\mathsf{bi}_{\lambda/\mu}\in \mathsf{biPol}^\I_d$, for any skew 
partition $\lambda/\mu$ such that $|\lambda|-|\mu| = d$,  we prove a filtration
of these superfunctors in Theorem  \ref{thm:filtration}.
This result is a super-analogue of Theorem II.4.11 in \cite{ABW}, 
which was recently used to obtain a
categorification of Fock space representations in \cite{HTY}.

We conclude the paper by showing that
the Schur superfunctors $\widehat{S}_\lambda$ are indecomposable
objects in the category $\Pol^\I_d$.  In order to prove this, we
first show in Theorem \ref{thm:indecomp} that the Schur supermodule 
$\widehat{S}_\lambda(M)$, given by
evaluation at a superspace $M$, is an indecomposable module 
for the algebraic supergroup 
$GL(M)$ (or rather of the corresponding Schur superalgebra).  
On the other hand, 
the Schur superfunctor, $\widehat{S}^{\hp \II}_{\lambda}$, is 
usually not an indecomposable object of the category $\Pol^\II_d$.
It is thus an interesting problem to define an analogue
of Schur superfunctors of type $\II$ whose corresponding Schur supermodules
are indecomposable $Q(n)$-supermodules.  Such a construction seems to be lacking
even in characteristic zero.

In \cite{LZ}, La Scala and Zubkov also constructed {\em costandard modules}, 
$\nabla(\lambda)$, for the supergroup $GL(M)$. 
Provided that the size of $\lambda$ is small enough, these supermodules 
appear to be isomorphic to our Schur supermodules, 
$\widehat{S}_\lambda(M)$. 
In particular, the standard basis of $\nabla(\lambda)$ given in 
\cite[Theorem 6.6]{LZ} is in bijection with the standard basis given in Theorem 
\ref{thm:standard} below.
%

\subsection*{Acknowledgments}  
The author wishes to thank Seok-Jin Kang and Myungho Kim  
for many helpful discussions and  Jerzy Weyman for helpful correspondence.
\smallskip
%


\section{Superalgebra Preliminaries}
We will assume throughout that $\k$ is an arbitrary field of characteristic not equal 2.  
Given $V,W$ a pair of $\k$-vector spaces, the set of all linear maps from $V$ to $W$ 
is denoted $\Hom(V,W)$.

\subsection{Categories enriched in $\k$-superspaces}\label{ss:svec}
A  {\em vector superspace} is a $\k$-vector space $M$ with $\Z_2$-grading, 
$M= M_0\oplus M_1$.  If $\dim (M_0) =m$ and $\dim (M_1)=n$, we
write $\text{sdim}(M) = m|n$. 
An element of $M_0$ (resp.~$M_1$) is called 
{\em even} (resp.~{\em odd}).  We write $\underline{M}$ to denote the
underlying ordinary vector space without a $\Z_2$-grading. 
A {\em subsuperspace} of 
$M$ is a subspace $N$ of $M$ such that $N = (N\cap M_{0})\oplus 
(N\cap M_{1})$.
If $M$ and $N$ are superspaces, we 
consider $\Hom(M,N)$ as a superspace by setting $\Hom(M,N)_\varepsilon$
equal to the set of linear maps $f:M\to N$ which are $\Z_2$-homogeneous 
of degree $\varepsilon$, i.e. $f(M_\delta) \subset N_{\varepsilon+\delta}$ 
$(\varepsilon,\delta \in \Z_2)$.
The $\k$-linear dual $M^\asterisk = \Hom(M,\k)$ is a superspace by viewing 
$\k$ as vector superspace concentrated in degree 0.

The tensor product $M\otimes N$ of superspaces is also a superspace with 
$(M\otimes N)_0 = M_0\otimes N_0 \oplus M_1\otimes N_1$ and
$(M\otimes N)_1 = M_0\otimes N_1 \oplus M_1\otimes N_0.$
There is an even linear isomorphism 
$\tau_{_{M,N}}: M\otimes N \to N\otimes M$ of superspaces, given by 
\begin{equation*}
v\otimes w \mapsto (-1)^{|v| |w|} w\otimes v \quad (v\in M, w\in N),
\end{equation*} 
called the {\em supertwist} map.

A {\em superalgebra} is an associative algebra $A$ which is also a superspace
such that $A_\varepsilon A_\delta \subset A_{\varepsilon+\delta}$. 
As above, $\underline{A}$ denotes the ordinary associative algebra 
obtained by forgetting the $\Z_2$-grading.  
Given two superalgebras $A$ and $B$, 
a {\em superalgebra homomorphism} 
$\varphi: A \to B$ is an even linear map which is an algebra homomorphism; 
its kernel is a {\em superideal}, i.e., an ordinary two-sided ideal 
which is also a subsuperspace. 
We view the tensor product of superspaces 
$A\otimes B$ as a superalgebra with multiplication defined by the usual {\em rule of signs}
convention:
\begin{equation}\label{eq:signs}
(a\otimes b)(a' \otimes b') = (-1)^{|b| |a'|} (a a') \otimes (b b') \qquad (a,a' \in A,\ b,b' \in B).
\end{equation}
The supertwist map gives an isomorphism
$\tau_{_{A,B}}: A\otimes B \cong B \otimes A$ of 
superalgebras.

\begin{example}\label{ex:cliff}
The {\em Clifford superalgebra}, $\C_n$, is the superalgebra with odd generators
$c_1, \dots, c_n$ satisfying: $c_i^2=1$ and $c_ic_j=-c_jc_i$, for $1\leq i\neq j \leq n$.
For any $n\geq 1$, there is an isomorphism 
$\C_n \cong \C_1\otimes \cdots \otimes \C_1$ ($n$ copies).
\end{example}

%
Let $A$ be a superalgebra.  
A {\em right $A$-supermodule} is a superspace $V$ 
which is a right $A$-module in the usual sense, such that 
$A_\varepsilon V_\delta \subset V_{\varepsilon+\delta}$ 
for $\varepsilon,\delta \in \Z_2$.   One may similarly define 
{\em left $A$-supermodules}.  
  A {\em homomorphism} $\varphi: V \rightarrow W$ of  right (resp.~left) 
  $A$-supermodules $V$ and $W$ is a 
  linear map $\varphi$ such that 
\[
 \varphi(v a) = \varphi(v) a \qquad (\mbox{resp.~} \varphi(av) =(-1)^{|\varphi| |a|} a \varphi(v)), 
\qquad ^\forall a \in A, v \in V.
\]

Let $\Smod\hpp A$ (resp.~$A\hpp \Smod$)  denote the category 
of all right (resp.~left) $A$-supermodules
and $A$-supermodule homomorphisms.
Let $\smod\hpp A$ (resp.~$A\hpp \smod$)  denote the full subcategory of all right 
(resp.~left) $A$-supermodules which are finite dimensional (over $\k$).
Similarly, write $\Svec$ (resp.~$\svec$) for the category of all 
(resp.~finite dimensional) $\k$-vector superspaces.
If $V,W \in A\hpp \smod$ (resp. $\smod\hpp A$), we let $\Hom_{A}(V,W)$ 
denote the set of $A$-homomorphisms from $V$ to $W$.  Also let $\End_A(V)$ 
denote the superalgebra of all $A$-supermodule endomorphisms of $V$.  

Given any superspace $M$ we denote by $\Pi M$ the same vector space, but with opposite 
$\Z_2$-grading.  For example, we write $\k^{m|n} = \k^m \oplus (\Pi \k)^n$.  
%
Given a superalgebra $A$, we   define the 
{\em (right) parity change} functor 
$\Pi : \Smod\hpp A\to \Smod\hpp A$ which sends $V\mapsto \Pi V$, 
such that $\Pi V$ has the same right action: $v\cdot a = va$ ($a\in A, v\in V$). On a right
$A$-supermodule homomorphism $\varphi:V\to W$, 
we set $\Pi(\varphi)= (-1)^{|\varphi|}\varphi$ as a linear map.

\smallskip

We say a category $\V$ is an {\em $\Svec$-enriched category} 
if the hom-sets $$\hom_\V(V,W)\qquad  (V,W\in \V)$$ are 
$\k$-superspaces, while composition is bilinear and 
{\em even}.  I.e., if $U,V,W \in \V$, then composition induces an even 
linear map:
$${\rm hom}_{\V}(V,W) \otimes {\rm hom}_{\V}(U,V) \rightarrow {\rm hom}_{\V}(U,W).$$
As usual if  $V,W$ are isomorphic in $\V$, we write $V \cong W$.  If there 
is an even isomorphism $\varphi:V \cong W$ (i.e., 
$\varphi\in {\rm hom}_\V(V,W)_0$), then we use the notation 
$$V \simeq W.$$
Let $\V_\ev$ denote the subcategory of $\V$ consisting of the same 
objects but only even homomorphisms.

For a superalgebra $A$, the category $A\hpp \Smod$ ($\Smod\hpp A$) 
is naturally an $\Svec$-enriched category, as is the full subcategory
$A\hpp \smod$  ($\smod\hpp A$).  
Further, the even subcategories 
$(_A\Smod)_\ev$, etc.{\hpp\hpp }are abelian categories in the usual sense.  
All functors between the $\Svec$-enriched categories which we consider 
will send even morphisms to even morphisms.  They 
will thus give rise to corresponding functors between the underlying even 
subcategories.

\begin{definition}
Suppose $\V,\V'$ are $\Svec$-enriched categories.  An {\em even} functor 
$F:\V\to \V'$  is called {\em $\k$-linear} (or {\em $\Svec$-enriched}) if 
the function
$$F_{V,W}: \hom_\V(V,W) \to \hom_{\V'}(FV,FW)$$
 is a linear map for any
pair of objects $V,W \in \V$.  
Suppose that $F,G: \V \to \V'$ are a pair of {\em $\k$-linear functors}.  Then
an {\em even} (resp.~{\em odd}) {\em $\Svec$-enriched natural transformation} 
$\eta: F \to G$ consists of a collection of even (resp.~odd) linear maps 
 $$\eta(V) \in \hom_{\V'}(FV,GV) \qquad (^\forall V\in \V)$$
such that for a given $\varphi \in \hom_\V(V,W)$ we have 
 $$G(\varphi) \circ \eta(V) = (-1)^{|\varphi| |\eta(V)|} \eta(W) \circ F(\varphi).$$
In general, an {\em $\Svec$-enriched natural transformation}, $\eta:F\to G$, is 
defined to be a collection of linear maps, $\eta(V)= \eta_0(V) \oplus \eta_1(V)$, 
for all $V\in \V$, such that 
$$ \eta_\varepsilon(V) \in \hom_{\V'}(FV,GV)_\varepsilon \qquad (\varepsilon \in \Z_2)$$
and $\eta_0$ (resp.~$\eta_1$)$: F\to G$ is an even (resp.~odd)  $\Svec$-enriched
natural transformation.
\end{definition}

\begin{definition}[Cartesian product of $\Svec$-enriched categories]
Given two $\Svec$-enriched categories $\V$, $\W$ we may form the direct product
$\V\times \W$, which is the $\Svec$-enriched category with pairs $(V,W)$ 
$(V\in \V, W\in\W)$ as objects and morphisms 
$$\hom_{\V\times \W}\big((V,W), (V',W')\big)\ =\ \hom_\V(V,V') \otimes \hom_\W(W,W').$$
The composition of morphisms follows the rule of signs convention (\ref{eq:signs}) stated above.
\end{definition}

In the above definition, notice that we have chosen the hom-spaces to be tensor products
instead of Cartesian products.  In this way, we may consider only even, $\k$-linear functors: 
$\V\times \W \to \U$, instead of bilinear functors. 
%

\subsection{Hopf superalgebras}\label{ss:Hopf}
A {\em cosuperalgebra} is a superspace $C$ with the structure of $\k$-coalgebra 
such that the comultiplication $\Delta:C\otimes C \to C$ and the counit 
$\epsilon: C\to \k$ are even linear maps.  If $C$ and $D$ are cosuperalgebras, 
the tensor product $C\otimes D$ is a cosuperalgebra with comultiplication
$$ \Delta_{C\otimes D}=\ 1\otimes \tau_{_{C,D}} \otimes 1 \circ (\Delta_C\otimes \Delta_D).$$
A homomorphism $\varphi:C\to D$ of cosuperalgebras is an even linear map
which is a homomorphism of coalgebras in the usual sense.

By a {\em bisuperalgebra}, we mean a superalgebra which is also a 
cosuperalgebra such that the comultiplication and counit are superalgebra
homomorphisms.
E.g., if $A$ is a bisuperalgebra, then the following diagrams are commutative:
\begin{equation}\label{eq:diagram1}
\xymatrix{
A \otimes A\  \ar[d]_{m_{_A}} \ar[rr]^-{\Delta\otimes \Delta}  & &\ A \otimes A \otimes A \otimes A \ar[d]^-{m_{\hmm_{A\otimes A}}}\\
A\ \ar[rr]_-{\Delta} & &\ A\otimes A,}
\qquad
\xymatrix{
A \otimes A  \ar[dr]_{\epsilon\otimes\epsilon} \ar[r]^-{m_A}   &A \ar[d]^-{\epsilon}\\
 & \k.}
\end{equation}
(Recall that $m_{\hmm_{A\otimes A}}$ is defined to be $m_{\hmm_A}\otimes m_{\hmm_A} \circ (1\otimes \tau_{\hmm_{A, A}} \otimes 1)$.)
A morphism $\varphi:A \to B$ of bisuperalgebras is a superalgebra map
which is also a map of cosuperalgebras.
\smallskip

A {\em Hopf superalgebra} is a bisuperalgebra $A$ which is also a Hopf algebra, 
such that the {\em antipode}, $\iota: A\to A$, is an even linear map. 

\begin{definition}
We defined a {\em $\Z_{\geq 0}$-graded superspace} to be a superspace 
$M = \bigoplus_{d=0}^\infty M^d$ such that each $M^d$ is a finite dimensional 
subsuperspace.
A $\Z_{\geq 0}$-graded superspace $M$ is {\em connected} if $M^0= \k$.
\end{definition}

Note that the tensor product of $\Z_{\geq 0}$-graded superspaces, $M\otimes N$, is
naturally $\Z_{\geq 0}$-graded, with $(M\otimes N)^d = \bigoplus_{i+j=d}M^i\otimes N^j$. 

\begin{definition}
Suppose $A$ is a (co)superalgebra whose underlying superspace is $\Z_{\geq 0}$-graded. 
Then $A$ is a {\em graded (co)superalgebra} if the multiplication $m:A\otimes A\to A$
(resp.~comultiplication $\Delta: A\to A\otimes A$) is a $\Z_{\geq 0}$-homogeneous map.
A {\em graded bisuperalgebra} is a bisuperalgebra which is graded
both as a superalgebra and as a cosuperalgebra simultaneously.

Define the {\em graded dual}, $A^{\mathsf{gr},*}$, of a graded bisuperalgebra,
$A = \bigoplus_{d=0}^\infty A^d$, to be the bisuperalgebra  
$$A^{\mathsf{gr},*}= \bigoplus_{d=0}^\infty (A^d)^*,$$
with multiplication $m^\mathsf{gr}$ and comultiplication $\Delta^\mathsf{gr}$ defined to be
the maps which are dual to compultiplication $\Delta$ and multiplication $m$, respectively. 
\end{definition}

\begin{remark} \label{rmk:Takeuchi}
By Takeuchi's antipode formula (cf. \cite[Lem.~14]{Takeuchi} or, e.g., \cite[Prop.~3.8.8]{HGK}), 
any connected $(\Z_{\geq 0})$-graded bialgebra is automatically a Hopf algebra.
Using this formula, it may be checked that if $A$ is a connected bisuperalgebra, then both $A$ and 
$A^{\mathsf{gr},*}$ are also Hopf superalgebras.
\end{remark}

\subsection{The symmetric superalgebra}
Let $M$ be a vector superspace.  For $d\in \Z_{\geq 1}$, we write 
$M^{\otimes d}= M\otimes \cdots \otimes M$
($d$ copies) and $M^{\otimes 0}=\k$. The {\em tensor superalgebra}, $\text T^{\udot}\hm M$, is 
the tensor algebra $$\text T^{\udot}\hm M = \bigoplus_{d\ge 0} M^{\otimes d}$$ regarded 
as a vector superspace.  It is the free associative 
superalgebra generated by $M$.
Hence, there is a unique superalgebra map 
$\Delta:\text T^{\udot}\hm M \to \text T^{\udot}\hm M \otimes \text T^{\udot}\hm M$ such that: $v \mapsto v\otimes 1 + 1 \otimes v.$
Also, let $\epsilon:\text T^{\udot}\hm M \to \k$ be such that $\epsilon(1)=1$ and
$\epsilon(x) = 0$, if $x\in M^{\otimes d}$ with $d>0$.
Then these maps makes $\text T^{\udot}\hm M$ into a bisuperalgebra.
The fact that $\text T^{\udot}\hm M$ is a Hopf superalgebra follows by
Remark \ref{rmk:Takeuchi}.
\smallskip

A superalgebra $A$ is called {\em commutative} if 
$ab = (-1)^{|a| |b|} ba$ for all $a,b \in A$.
The {\em symmetric superalgebra}, 
$S^{\udot}\hmm M$, is the free commutative superalgebra generated by $M$.
Explicitly, $S^{\udot}\hmm M = \bigoplus_{d \geq 0} S^d M$
is the quotient of $\text T^{\udot}\hm M$ by the super ideal 
$$\mathcal{I}= \langle x\otimes y - (-1)^{|x| |y|}y\otimes x ;\ x,y \in M \rangle.$$ 
Since the tensor product of commutative superalgebras is also
commutative, there is a unique homomorphism
$\Delta: S^{\udot}\hmm M \to S^{\udot}\hmm M\hp \otimes\hp S^{\udot}\hmm M$ of superalgebras such that: 
$v \mapsto v\otimes 1 + 1 \otimes v.$  Define the counit 
$\epsilon: S^{\udot}\hmm M \to \k$ in the same way as above.  Together these maps give
$S^{\udot}\hmm M$ the structure of Hopf superalgebra (again using Remark \ref{rmk:Takeuchi}).

Notice that $S^\udot (\ )$ satisfies the {\em exponential property}, since there is an
isomorphism:
$$
S^\udot(M)\otimes S^\udot(N) \cong S^\udot(M\oplus N),$$
which is given by: $x\otimes y \mapsto xy$, for 
any $x\in S^\udot M, y\in S^\udot N$.
In particular, we have  
${S^{\udot}\hmm (M)} \cong\ S^\udot (M_0) \otimes S^{\udot}(M_1)$.  
By forgetting the $\Z_2$-grading, we then have an isomorphism
\begin{equation}\label{eq:alg1}
 \underline{S^{\udot}\hmm (M)}\ \cong\  \Sym^{}\hp\hp \underline{M_0} \otimes \Lambda^{}\hp\hp \underline{M_1}
\end{equation} 
of ordinary (graded) associative algebras, where $\Sym^{} (\ )$ and $\Lambda^{} (\ )$ denote 
the symmetric and exterior algebras of a vector space, respectively.

Suppose now that $M\in \svec$.  Let $\Xc= (X_1, \dots, X_m)$ be a basis of $M_0$ and
$\Yc=(Y_1, \dots, Y_n)$ a basis of $M_1$.  
If we write the multiplication in $S^\udot\hmm M$ simply by juxtaposition, then
it follows from ({\ref{eq:alg1}) that the set of elements
\begin{equation} \label{eq:sym}
\{X_1^{d_1} \cdots X_m^{d_m}\hpp Y_1^{e_1} \cdots Y_n^{e_n}\hpp ;\hpp 
 d_1, \dots, d_m \in \Z_{\geq 0},\hpp e_1, \dots e_n\in \{0,1\}, \text{ and  
 $\sum d_i + \sum e_j = d$} \}
\end{equation}
forms a basis of $S^d M$.

\smallskip

\subsection{The divided powers superalgebra}\label{ss:divided}
Let $M\in \svec$ and suppose $d\in \Z_{\ge 0}$.  Then we define the
{\em $d$-th divided power} of $M$ to be the superspace: $ \Gamma^d M = S^d(M^\ast)^\ast.$
\smallskip

Suppose as above that $\Xc=(X_1, \dots, X_m)$ and $\Yc=(Y_1, \dots, Y_n)$ are bases
of $M_0$ and $M_1$, which are ordered by indices.  Fix some total order on their union 
$\Zc = \Xc \sqcup \Yc$ which preserves the relative orders in $\Xc$ and $\Yc$ respectively.
We write the elements of $\Zc$ as $Z_1, \dots, Z_{m+n}$ with the total order again indicated
by indices.  
Now $\Zc$ is a $\Z_2$-graded set, with $\Zc_0 = \Xc$ and $\Zc_1 = \Yc$.  
We write $|Z_i| = \varepsilon$ if $Z_i\in \Zc_\varepsilon$, for $i= 1, \dots, m+n$ and $\varepsilon = 0,1$.
\smallskip

Let  $\check{\Zc}= (\check{Z}_1, \dots, \check{Z}_{m+n})$ denote the
 basis of $M^\ast$ which is dual to $\Zc$.  Suppose $\alpha \in (\Z_{\geq 0})^{m+n}$
 is such that $|\alpha| = \sum \alpha_i =d$.
Then we denote by
\begin{equation*}
Z^{(\alpha)}\hp =\hp \hpp Z_1^{(\alpha_1)} \dots Z_{m+n}^{(\alpha_{m+n})} 
\end{equation*}
the element of $\Gamma^d M$ which is dual to the monomial 
$\check{Z}_1^{\alpha_1}  \dots \check{Z}_{m+n}^{\alpha_{m+n}}$ of  $S^d(M^\ast)$.
Since $S^\udot(M^\ast)$ is commutative and (\ref{eq:sym}) is a basis of $S^d(M^\ast)$, it follows that 
$$\{Z^{(\alpha)} : |\alpha| =d \text{ and } 0\leq \alpha_i\leq 1 \text{ if }  |Z_i| = 1\}$$ 
gives a basis of $\Gamma^d M$.

Let us  consider an alternative way to denote this basis for $\Gamma^dM$.
We write $I(m|n,d)$ to denote the set of all functions 
$\bs i: \{1,\dots, d\} \to \{1,\dots, m+n\}$.  
We may identify any 
$\bs i\in I(m|n,d)$ as a sequence $\bs i =(i_1, \dots, i_d)$ of
elements $i_k = \bs i(k) \in \{1, \dots, m+n\}$. 
For each  $\bs i\in I(m|n,d)$ there is
a corresponding element 
$Z^{(\bs i)} \in \Gamma^dM$ which is dual to the element
$$\check{Z}_{i_1} \check{Z}_{i_2} \cdots \check{Z}_{i_d} \in S^d M.$$
If $\bs i\in I(m|n,d)$, we let $wt(\bs i) \in (\Z_{\geq 0})^{\hp m+n}$ denote the sequence
obtained by setting $$wt(\bs i)_s = \#\{1\leq k\leq d\hp : \hpp i_k=s\},$$
for $s=1, \dots, m+n$.

We say that $\bs i \in I(m|n,d)$ is {\em $\Zc$-restricted} if: 
$wt(\bs i)_s\leq 1 $ whenever $|Z_s| = 1$.
Note that the symmetric group $\Si_d$ acts on $I(m|n,d)$ by composition,
i.e.~$(\bs i.\sigma)(k) = \bs i(\sigma k)$ for any $\sigma\in\Si_d$.
We define the {\em standardization} of $\bs i$ to be the
sequence $\rm{st}(\bs i) = \rm{st}_\Zc(\bs i)$  obtained by rearranging the entries of $\bs i$ into
non-decreasing order.  I.e., $\rm{st}(\bs i) = \bs i.\sigma$ for 
some $\sigma \in \Si_d$ such that
$\bs i_{\sigma 1} \leq \cdots \leq \bs i_{\sigma d}.$
We say that $\bs i$ is {\em standardized} if $\bs i = \rm{st}(\bs i)$. 
If $\bs i$ is standardized, then  $Z^{(\bs i)} = Z^{(wt(\bs i))}$, and  $Z^{(\bs i)} \neq 0$ if and only if
$\bs i \in I(m|n,d)$ is $\Zc$-restricted.
It follows that 
$$\{Z^{(\bs i)} : \bs i \in I(m|n,d) \text{ is } \Zc\text{-standardized and } \Zc\text{-restricted} \}$$
gives another way to denote our basis for $\Gamma^dM$.

For any $\bs i = (i_1,\dots, i_d) \in I(m|n,d)$, let us write $|\bs i|_\Zc$ to denote the  
sequence in $(\Z_2)^{m+n}$ given by: 
$|\bs i|_\Zc = (|Z_{i_1}|, \dots, |Z_{i_{m+n}}|)$.
Given $\varepsilon \in (\Z_2)^d$ and $\sigma \in \Si_d$, we let
$$
 \text{sgn}( \varepsilon, \sigma) = \prod_{\substack{1\leq s<t \leq d\\ \sigma^{-1}s>\sigma^{-1}t}} (-1)^{\varepsilon_s \varepsilon_t}$$
We then set  
$\text{sgn}_{\Zc}(\bs i; \sigma) =  \text{sgn}( |\bs i|_\Zc, \sigma)$.  
If $\bs i = \bs j.\sigma$, then we also write 
$\text{sgn}_{\Zc}(\bs i; \bs j) = \text{sgn}_{\Zc}(\bs i; \sigma)$.

There is also a unique right action of the symmetric group $\Si_d$ 
on the tensor power $M^{\otimes d}$ of a superspace 
such that each transposition $(i\ i+1)$ for $1\le i \le d-1$ acts by: 
\[
(v_1 \otimes \cdots \otimes v_d). (i\ i+1)\hpp =\hpp\hpp (-1)^{|v_i| |v_{i+1}|}\hpp v_1\otimes \cdots \otimes v_{i+1} \otimes v_i \otimes \cdots \otimes v_d,
\]
for any $v_1, \dots, v_d \in M$ with $v_i, v_{i+1}$ being $\Z_2$-homogeneous.  
We denote by $Z^{\bs i}$ the element of $M^{\otimes d}$ given by 
$Z^{\bs i}\hp =\hpp Z_{i_1} \otimes \dots \otimes Z_{i_d}$. Then $M^{\otimes d}$
has as a basis  
$\{Z^{\bs i}\hp| \hpp \bs i\in I(m|n,d)\}$, and the above action 
may be given explicitly by 
\begin{equation}\label{eq:sgn}
Z^{\hp \bs i}.\sigma =  \text{sgn}_\Zc(\bs i; \sigma) Z^{\hp \bs i.\sigma},
\end{equation}
for any $\sigma \in \Si_d$.
\smallskip

Notice that the symmetric power is the coinvariant superspace, 
$S^d M = (M^{\otimes d})_{\Si_d}$, with respect to the above action.
This means that there is a canonical even isomorphism 
\begin{equation}\label{eq:coinv}
 \Hom_{\k \Si_d}(M^{\otimes d}, N) \simeq \Hom_{\k \Si_d}(S^d M, N)
 \end{equation}
for any $N\in \svec$ considered as a trivial $\Si_d$-module.
Now there is also a right action of $\Si_d$ on the dual space $(M^{\otimes d})^\asterisk$ 
given by $(f.\sigma)(v) = f(v.\sigma^{-1})$, for $f\in (M^{\otimes d})^\ast, v\in M^{\otimes d}$ 
and $\sigma \in \Si_d$.  Further $(M^\ast)^{\otimes d} \simeq (M^{\otimes d})^\ast$
as $\Si_d$ module. It then follows from the definition of $\Gamma^dM$ and (\ref{eq:coinv}) 
that we have a canonical even isomorphism
$$ \Hom_{\k \Si_d}(N, M^{\otimes d}) \simeq \Hom_{\k \Si_d}(N, \Gamma^d M)$$
for any $M\in \svec$.
Hence the $d$-th divided power is isomorphic to the subsuperspace of invariants: 
$\Gamma^d M \simeq (M^{\otimes d})^{\mathfrak{S}_d}$.

We define the {\em divided powers supergalebra} to be the dual Hopf superalgebra:  
$$\Gamma^{\udot}\hmm M= (S^{\udot} M^\ast)^{\mathsf{gr},*} = \bigoplus_{d\ge 0}\Gamma^dM.$$  
Let us denote the multiplication of elements $f,g \in \Gamma^\udot M$ by $f\coasterisk g$. 
Then $\Gamma^\udot(\ )$ also satisfies the exponential property.  This follows by duality 
from the exponential property for $S^\udot (\ )$, since there is an isomorphism
\begin{equation}\label{eq:exponential}
\Gamma^\udot(M)\otimes \Gamma^\udot(N)\ \cong\ \Gamma^\udot(M\oplus N),
\end{equation}
given by $x\otimes y \mapsto x\coasterisk y$, for all $x\in \Gamma^\udot(M)$,
$y\in \Gamma^\udot(N)$.
%
%
We also have an isomorphism of ordinary (graded) algebras
\begin{equation}\label{eq:alg2}
 \underline{\Gamma^{\udot}\hmm M}\ \cong\ D^{}\hp \underline{M_0}\otimes \Lambda^{}\hp \underline{M_1},
\end{equation}
where $D^{} (\ )$ denotes the usual divided power algebra of a vector space.

Suppose $\bs i \in I(m|n,d)$ and $\bs j \in I(m|n,d')$. 
Then we define $\bs i \vee \bs j \in I(m|n,d+d')$ to be the element 
$$\bs i \vee \bs j = (i_1, \dots, i_{d}, j_1, \dots, j_{d'}).$$
Let us write $l(\bs i)=d$ if $\bs i \in I(m|n,d)$ for some $d>0$, and
we set
 $$I(m|n) = \bigcup_{d\geq1} I(m|n,d).$$ Then $(\bs i, \bs j) \mapsto \bs i\vee\bs j$ 
 defines a binary operation $I(m|n)\times I(m|n) \to I(m|n)$.
\smallskip

\begin{lemma}\label{lem:multiply} Let  
$\Zc = \Zc_0\sqcup \Zc_1=  (Z_1, \dots, Z_{m+n})$ be a $\Z_2$-homogeneous basis 
of $M = M_0\oplus M_1$ as above. 
Suppose $\bs i\in I(m|n,d)$, $\bs j \in I(m|n,d')$ are both standardized. 
\smallskip
\begin{itemize} 
\item[(i)] 
Let $\Delta
: \Gamma^{r+s}M\to \Gamma^rM \otimes \Gamma^sM$
denote the components of the comultiplication 
$\Delta: \Gamma^\udot M \to \Gamma^{\udot}\hmm M \otimes \Gamma^{\udot}\hmm M$.
These components are given explicilty by:
%
$$\displaystyle \Delta( Z^{(\bs i)} ) = \sum_{\mathrm{st}(\bs i^1\vee \bs i^2)= \bs i} 
{\text{sgn}_\Zc(\bs i^1\vee \bs i^2; \bs i)}\hpp
Z^{(\bs i^1)} \otimes  Z^{(\bs i^2)},$$
summing over ordered pairs of standardized sequences 
$\bs i^1\in I(m|n, r)$, $\bs i^2\in I(m|n, s)$. 
\medskip

\item[(ii)] Multiplication on basis elements in $\Gamma^{\udot}\hmm M$ satisfies:
$$Z^{(\bs i)} \coasterisk Z^{(\bs j)}  = 
\begin{cases}
C\hp Z^{(\bs i\vee \bs j)}  & \text{if }\bs i\vee \bs j \text{ is } \Zc\text{-strict}\\
0 & \text{if }\bs i\vee \bs j \text{ is not } \Zc\text{-strict},
\end{cases}
$$
where  $C$ is some positive integer such that $C=1$ if and only if 
$\text{Im}(\bs i) \cap\text{Im}(\bs j) = \emptyset$.
\end{itemize}
\end{lemma}

\begin{proof}
The lemma follows from (\ref{eq:alg2}) and the usual properties of exterior algebras and divided powers
of ordinary vectors spaces, as described in \cite[Sec.~1]{Pirashvili}, \cite[Chap.\hpp 1]{Weyman}.
\end{proof}

Note that there is a surjective map, $m: M^{\otimes d} \twoheadrightarrow S^d M$, given
by $d$-fold multiplication in the symmetric superalgebra.  By duality and some abuse of
notation, we have a 
corresponding embedding $\Delta:\Gamma^d M \hookrightarrow  M^{\otimes d}$ given
by applying successive iterations of comultiplication.  This embedding is given 
explicitly as follows. 

If $\bs i \in I(m|n,d)$, then 
let us denote the stabilizer of $\bs i$ in $\Si_d$ by 
$\Si_{\bs i} = \{ \sigma \in \Si_d: \bs i.\sigma = \bs i\}$.
Then we have:
\begin{equation}\label{eq:Delta}
\Delta(Z^{(\bs i)}) 
= \sum_{\sigma\in \Si_d/ \Si_{\bs i}} Z^{\bs i}\hmm\hm.\sigma 
= \sum_{\sigma\in \Si_d/ \Si_{\bs i}} \text{sgn}_\Zc(\bs i; \sigma) Z^{\hp \bs i.\sigma},
\end{equation}
where the last equality comes from (\ref{eq:sgn}).  
By the injectivity of $\Delta:\Gamma^d M \hookrightarrow  M^{\otimes d}$, we then have 
\begin{equation*}\label{eq:compare}
Z^{(\bs i)} = \text{sgn}_\Zc(\bs j; \sigma) Z^{(\bs j)}
\end{equation*}
 in $\Gamma^dM$, 
for any $\bs j\in I(m|n,d)$ such that there exists some $\sigma \in \Si_d$, 
with $\bs i = \bs j.\sigma$.
\smallskip

Notice also from the commutativity of diagram (\ref{eq:diagram1}) that we have 
$\Delta(x\coasterisk y) = \Delta(x) \coasterisk \Delta(y)$ for any $x,y \in\Gamma^\udot M$.
In particular, 
\begin{equation}\label{eq:div_tensor}
\Delta(x\coasterisk y) = \sum_{\sigma \in \Si_{d+e}/\Si_d\times \Si_e} (\Delta(x) \otimes \Delta(y)). \sigma,
\end{equation}
if $x\in \Gamma^dM$ and $y\in \Gamma^eM$.

%
%

\section{Recollection of Schur Functors and Schur Complexes}\label{Sec:SchurFunct}
In this section, we recall the definition of Schur complexes defined in \cite{ABW}.
As vector spaces, these are canonically isomorphic to images of the underlying
 vector spaces of the 
Schur superfunctors to be defined in the next section.
We also recall the definitions of Schur and Weyl functors, which are fundamental 
examples of objects in categories of {\em strict polynomial functors} 
(cf. \cite{TouzeRingel}, \cite{Krause}).

\subsection{Strict polynomial functors and bifunctors} \label{ss:poly}
Let $\vvec=\vvec_\k$ 
denote the category of finite dimensional $\k$-vector spaces.
Then define the category $\D^d_\k = \D^d(\vvec)$ with the same objects as $\vvec$
and with morphisms
$$\hom_{\D^d_\k}(V,W):= D^d\Hom(V,W),$$
for $V,W\in \vvec$, where $D^d(\ )$ denotes the ordinary $d$-th divided power of a vector
space.  The composition of morphisms in $\D^d_\k$ is given by the isomorphisms
\begin{align*}
  D^d(\Hom(V,W)) &= (\Hom(V,W)^{\otimes d})^{\Si_d} \cong  
  \Hom(V^{\otimes d}, W^{\otimes d})^{\Si_d}  \cong\ \Hom_{\k \Si_d}(V^{\otimes d},W^{\otimes d}),
 \end{align*}
where $\Si_d$ acts on $V^{\otimes d}, W^{\otimes d}$ by place permutations, and on
$\Hom(V^{\otimes d}, W^{\otimes d})$ by conjugation.
\smallskip

Notice that $\D^d_\k$ is a $\k$-linear category; i.e. hom-sets are vector spaces,
and the composition is bilinear. 
The category of {\em (homogeneous) strict polynomial functors} is defined to be
$$ \P_d\ =\ \Fct_\k(\D^d_\k, \vvec),$$
the category consisting of all $\k$-linear functors $T: \D^d_\k \to \vvec$.  
Hence if $T\in \P_d$, then for any $V,W \in \D^d_\k$, the induced map 
$$T_{_{V,W}}: \hom_{\D^d_\k}(V,W)\to  \Hom(TV,TW)$$
is linear.  Morphisms in $\P_d$ are natural transformations between functors.  
\smallskip

The category $\P_d$ is abelian, with the existence of kernels, cokernels, products and
coproducts being induced from the target category $\vvec$.   
Typical examples of objects in $\P_d$ are respectively given by the exterior power,
$\Lambda^d$, the symmetric power, $\Sym^d$, and the divided power,
$D^d$, of a vector space.
\smallskip

We now define tensor products of strict polynomial functors following \cite[Sec.~2]{Krause}. 
Let us define a category $\D^d_\k \otimes \D^e_\k$ with the same objects as $\vvec$ and
with morphisms
$$\hom_{\D^d_\k \otimes \D^e_\k}(V,W) = \hom_{\D^d_\k}(V,W) \otimes \hom_{\D^e_\k}(V,W).$$  
Given nonnegative integers $d$ and $e$, we have an embedding 
$\Si_d \times \Si_e \hookrightarrow \Si_{d+e}$.  This induces an embedding 
\begin{equation}\label{eq:embed}
D^{d+e}V \hookrightarrow D^{d}V\otimes D^{e}V,
\end{equation}
 for any $V\in \vvec$, given by the composition of the following maps
\begin{align*}
 D^{d+e}V = (V^{\otimes\hp d+e})^{\Si_{d+e}}\ \subset\ (V^{\otimes\hp d+e})^{\Si_d\times \Si_e}
 \simeq\ (V^{\otimes d})^{\Si_d}\otimes (V^{\otimes e})^{\Si_e} = \Gamma^dV \otimes \Gamma^eV.
\end{align*}
We then have a $\k$-linear functor 
$$i_{d,e}:\D^{d+e}_\k \hookrightarrow \D^d_\k \otimes \D^e_\k,$$ 
which is the identity on objects and which acts by (\ref{eq:embed}) on morphisms.  
Composition with this functor 
gives a tensor product $-\otimes-: \P_d \times \P_e \to \P_{d+e}$.  I.e., we define
$(S\otimes T)(V) = S(V)\otimes T(V)$ for any $V\in \D^{d+e}_\k$ and $S\otimes T$ acts by 
$i_{d,e}$ on morphisms.

\begin{example}
Suppose $\alpha \in (\Z_{\ge 0})^n$ and $|\alpha| = \sum \alpha_i = d$.  Then we have
polynomial functors, $\Sym^\alpha, \Lambda^\alpha, D^\alpha \in \P_d$, defined as the
tensor products
$$
X^\alpha \ =\ X^{\alpha_1} \otimes\ \cdots\ \otimes X^{\alpha_n},
$$
for $X= \Sym, \Lambda, D$ respectively.
\end{example}

We also  define {strict polynomial bifunctors}.
For this, we first define the category $\D^d_{\k\times \k}= \D^d(\vvec\times\vvec)$ with
the same objects as $\vvec\times \vvec$ and with morphisms
$$\hom_{\D^d_{\k\times\k}}((V,W),(V',W')) 
= \bigoplus_{e+f=d}\hom_{\D^e_\k}(V,W)\otimes \hom_{\D^f_\k}(V',W')$$
Then define the category of {\em (homogeneous) strict polynomial bifunctors}, 
$\P^{[2]}\hpp =\hpp \Fct_\k(\D^d_{\k \times \k}, \vvec).$ 

Next we define an external tensor product 
$$-\boxtimes-: \P_d\times\P_e \to \P^{[2]}_{d+e}.$$ 
Let 
$\D_\k^e \boxtimes \D_\k^f \subset \D^d_{\k\times\k}$ denote the subcategory with the
same objects as $\D^d_{\k\times\k}$ but with morphisms:
$$\hom_{\D_\k^e\boxtimes\D_\k^f}((V,W),(V',W'))\hpp 
=\hpp \hom_{\D^e_\k}(V,W)\otimes \hom_{\D^f_\k}(V',W').$$
Now suppose $S\in \P_d$,  $T\in \P_e$. Let $S\boxtimes T \in \P^{[2]}_{d+e}$ denote the 
functor defined by setting, for $V,W \in \vvec$,
$$(S\boxtimes T)(V,W) := S(V)\otimes T(W),$$ 
and $S\otimes T$ acts on morphisms via the embedding 
$\D_\k^e \boxtimes \D_\k^f \hookrightarrow \D^d_{\k\times\k}$

\subsection{Schur and Weyl functors}

Now let $\Z_+^\infty$ denote the set of all infinite sequences $(\alpha_1, \alpha_2, \dots)$ 
of nonnegative integers such that $\alpha_i = 0$ for all but finitely many indices.  We identify any 
$\alpha\in (\Z_{\geq 0})^n$ as an element of $\Z_+^\infty$ by setting $\alpha_i = 0$ for $i\geq n+1$. 
For $\alpha\in \Z_+^\infty$, we write $|\alpha|= \sum \alpha_i$.  Also, let $l(\alpha)$ denote the cardinality 
of $\{i: \alpha_i\neq 0\}$.

A {\em partition} is a sequence $\lambda = (\lambda_1, \lambda_2, \dots) \in \Z_+^\infty$ with non-increasing
entries: $\lambda_1\geq \lambda_2 \geq \cdots$.   Let $\mathscr{P}$ denote the set of all partitions. 
We identify each $\lambda \in \mathscr{P}$ with its corresponding diagram 
$$\triangle_\lambda:=
\{ (i,j)\in (\Z_{>0})^2:\ 1\leq j\leq \lambda_i\}.
$$
A {\em skew partition} $\lambda/\mu$ is a pair of partitions such that $\mu \subset \lambda$; i.e.,
 $\mu_i \leq \lambda_i$ for $i\geq 1$.  The corresponding {\em skew diagram}, 
 $\triangle_{\lambda/\mu}$, is the complement of $\triangle_\mu$ in the set $\triangle_\lambda$.  
For $\lambda \in \mathscr{P}$, let $\lambda'=(\lambda'_1, \lambda'_2,\dots)$ denote the 
conjugate partition, where $\lambda'_i$ equals the cardinality of $\{j: \lambda_j \geq i\}$.

Let us first define a certain permutation associated to any skew partition.
Suppose $\mu \subset \lambda$ are partitions such that $|\lambda| - |\mu| = d$.  Then
each integer $r\in \{1, \dots, d\}$ can be written uniquely as a sum
  $$ r\ =\ (\lambda_1-\mu_1)\ +\ \cdots\ +\ (\lambda_{i-1}-\mu_{i-1})\ +\ j$$
 with $1\leq j \leq \lambda_i-\mu_i$.  The pair $(i, \hat{\jmath})= (i, \mu_i+j)$
 describes the position ($i$-th row, $\hat{\jmath}$-th column) of $r$ in the skew Young diagram
 $\triangle_{\lambda/\mu}$, and $\lambda/\mu$ determines a permutation
 $\sigma_{\lambda/\mu} \in \Si_d$ by 
  $$\sigma_{\lambda/\mu}(r)\ =\ (\lambda_1' - \mu_1')\ +\ \cdots\ +\ 
 (\lambda_{\hat{\jmath}-1}' - \mu_{\hat{\jmath}-1}')\ +\ (i-\mu_{\hat{\jmath}}'),$$
where $1\leq i\leq \lambda_j' - \mu_j'$.  Note that $\sigma_{\lambda'/\mu'} = \sigma_{\lambda/\mu}^{-1}$.
\smallskip

\begin{example} Let $\lambda=(4,3,1)$ and $\mu=(2,1)$.  Then we may number the corresponding
skew Young diagram
$\triangle_{(4,3,1)/(2,1)}$ in two different ways as follows:
$$\Yvcentermath1 \qquad \young(::12,:34,5), \qquad \young(::35,:24,1).$$  The first numbering
is from left to right in rows and from top to bottom, while the second numbering proceeds
down columns and then from left to right.
One may check that sending a number in a square of the left-hand diagram to the corresponding
number in the same square of the right-hand diagram gives rise to the permutation 
$\sigma_{\lambda/\mu}$.  In this case, we have $\sigma_{(4,3,1)/(2,1)} = (1, 3, 2, 5),$ written in cycle
notation. 
\end{example}
 
Next we define certain natural transformations in the category $\P_d$, for $d= |\lambda/\mu|$.
First, let $\theta_{\lambda/\mu}(V)$ be the composite of the linear maps
$$ \Lambda^{\lambda'/\mu'}\hp V\ \lhook\joinrel\rightarrow\ 
V^{\otimes d}\ 
 \stackrel{\rho(\sigma_{\lambda/\mu})}{\relbar \joinrel \relbar \joinrel \relbar \joinrel \longrightarrow}\
V^{\otimes d}\ 
  \twoheadrightarrow\ \Sym^{\lambda/\mu}\hp V,$$
where the first map is the canonical inclusion, $\rho$ is the representation of the symmetric
group acting by ordinary place permations, and the last map is given by $d$-fold multiplication 
in the algebra $\Sym V$.

Next let $\theta'_{\lambda/\mu}$ be defined as the composite
$$ D^{\lambda'/\mu'}\hp W\ \lhook\joinrel\rightarrow\ 
W^{\otimes d}\ 
 \stackrel{\rho^-(\sigma_{\lambda'/\mu'})}{\relbar \joinrel \relbar\joinrel \relbar \joinrel \longrightarrow}\
W^{\otimes d}\ 
  \twoheadrightarrow\ \Lambda^{\lambda/\mu}\hp W,$$

 \begin{definition}[Schur and Weyl functors]
Define the {\em Schur functor}, $S_{\lambda/\mu}$, of shape $\lambda/\mu$, to 
be the image of $\theta_{\lambda/\mu}$ in the category $\P_d$.  The {\em Weyl functor}, 
$W_{\lambda/\mu}$, is the image of the natural transformation $\theta'_{\lambda/\mu}$.
\end{definition} 

\begin{remark}
A similar construction of the Schur and Weyl functors, $S_{\lambda}$ and $W_{\lambda}$, 
corresponding to a partition $\lambda$ may be found in \cite{TouzeRingel} or \cite{Krause}. 
Note that when comparing our definition with the original definition in \cite{ABW} one must 
replace $\lambda/\mu$ by the conjugate shape $\lambda'/ \mu'$. 
For a useful dictionary of the various notations for Schur and Weyl functors (or modules) appearing 
in the literature, see \cite[Section 6.1.1]{TouzeRingel}.  
\end{remark}

\subsection{Schur complexes}
We define the following strict polynomial bifunctors
$$(\Lambda\boxtimes D)^d = \bigoplus_{i=0}^d \Lambda^i\boxtimes D^{d-i}\quad
 \text{and}\quad (\Sym\boxtimes\Lambda)^d =\bigoplus_{i=0}^d \Sym^i\boxtimes\Lambda^{d-i}.$$
Notice there is a canonical embedding $(\Lambda\boxtimes D)^d(V,W) \hookrightarrow (V\oplus W)^d$
given by tensoring the canonical embeddings for $\Lambda^i$ and $D^j$: 
$$\bigoplus_{i+j=d} \Lambda^i V \otimes D^{j} W \hookrightarrow \bigoplus_{i+j=d} V^{\otimes i} 
\otimes W^{\otimes j}\hpp \large{\subset}\hpp (V\oplus W)^d.$$ 
Hence there is a canonical embedding $(\Lambda\boxtimes D)^d \hookrightarrow (I^{\otimes d})^\text{bi}$.

Also, $\Sym V \otimes \Lambda W$ is a graded algebra such that 
$$(\Sym V \otimes \Lambda W)_d = \bigoplus_{i= 0}^d \Sym^i V \otimes D^{d-i} W = (\Sym \boxtimes D)^d(V,W).$$
Hence, $d$-fold multiplication in the algebra $\Sym V \otimes D W$ yields a canonical map
$(V\oplus W)^{\otimes d} \twoheadrightarrow (\Sym \boxtimes D)^d(V,W)$.
This yields a natural transformation $(I^{\otimes d})^\text{bi} \twoheadrightarrow (\Sym \boxtimes \Lambda)^d$.

Next, we define an action of the symmetric group on $(V\oplus W)^{\otimes d}$ which is a
combination of the two action, $\rho$ and $\rho^-$, considered above.  
Let $\rho^\text{sc}: \Si_d \to \End((V\oplus W)^{\otimes d})$ be the unique representation
such that the transposition $(i\ i+1)$ maps:
$$u_1 \otimes \cdots \otimes u_d \mapsto \begin{cases} 
-(u_1 \otimes \cdots \otimes u_{i+1} \otimes u_i\otimes \cdots \otimes u_d)& \text{if } u_i\in W \text{ and } u_{i+1} \in W,\\
\ u_1 \otimes \cdots \otimes u_{i+1} \otimes u_i\otimes \cdots \otimes u_d & \text{if $u_i\in V$ or $u_{i+1} \in V$}.
\end{cases}$$
Note that if we consider $V,W$ as purely even superspaces, then this agrees with the action
of $\Si_d$ on the superspace $(V\oplus \Pi W)^{\otimes d}$ described in Section \ref{ss:divided} above.
This then defines a natural transformation $\rho^\text{sc}(\sigma): (I^{\otimes d})^\text{bi} \to (I^{\otimes d})^\text{bi}$,
for any $\sigma \in \Si_d$.

Now let $\theta^\text{sc}_{\lambda/\mu}$ be the natural transformation defined as the composite
$$ (D\boxtimes \Lambda)^{\lambda'/\mu'} \lhook\joinrel\rightarrow\ 
(I^{\otimes d})^\text{bi} 
 \stackrel{\rho^\text{sc}(\sigma_{\lambda/\mu})}{\relbar \joinrel \relbar\joinrel \relbar \joinrel \longrightarrow}\
(I^{\otimes d})^\text{bi} 
  \twoheadrightarrow\ (\Sym \boxtimes\Lambda)^{\lambda/\mu},$$
where the first map is given by tensoring the canonical inclusions, 
 and the third map is given by tensoring the respective multiplication maps.

\begin{definition}\label{def:SchurFunct}
The image of the natural transformation ${\displaystyle \theta^\text{sc}_{\lambda/\mu}}$ 
in $(\P_d)^{bi}$, denoted by $SC_{\lambda/\mu}$, is called
the {\em Schur complex bifunctor} of shape $\lambda/\mu$.
\end{definition}
\smallskip

We notice that, for $V,W\in \vvec$, we have 
${\displaystyle \theta^{sc}_{\lambda/\mu}(V,0) = \theta_{\lambda/\mu}(V)}$
and ${\displaystyle \theta^{sc}_{\lambda/\mu}}(0,W) = \theta'_{\lambda/\mu}(W)$.  
Hence, we have that
\begin{equation}\label{eq:SchurWeylComplex}
S_{\lambda/\mu}= SC_{\lambda/\mu}(-, 0)\quad
\text{ and }\quad W_{\lambda/\mu}= SC_{\lambda/\mu}(0,-).
\end{equation}

\begin{remark}
To any linear map $\phi:W\to V$, there was defined in \cite{ABW} an associated complex, ``$L_{\lambda'/\mu'}\phi$,"
called the {\em Schur complex} of shape $\lambda/\mu$.  It is not difficult to check that $SC_{\lambda/\mu}(V,W)$
is (linearly) isomorphic to the Schur complex $L_{\lambda'/\mu'}\phi$ of the conjugate shape, regardless of which 
map $\phi$ is chosen.
\end{remark}

\section{The Schur and Weyl Superfunctors}\label{Sec:superSchur}
In this section, we construct the Schur and Weyl superfunctors
via certain $\Svec$-enriched natural transformations in the 
categories of polynomial superfunctors.  We also provide a 
standard basis for the evaluation of a Schur superfunctor on
a vector superspace in terms of tableaux.

\subsection{Categories 
of strict polynomial superfunctors}\label{ss:superpoly}
Suppose $A$ is an associative superalgebra, and let 
$d\in \Z_{\ge 1}$.  Notice that 
the right action of $\sigma \in \Si_d$ on the tensor power 
$A^{\otimes d}$ is in fact a superalgebra automorphism.    
Denote by $A\wr\Si_d$ the vector superspace
 \begin{equation*}
 A\wr\Si_d = \k \Si_d \otimes A^{\otimes d}
\end{equation*}
(where the group algebra $\k \Si_d$ is viewed as superspace 
concentrated in degree zero).  
We then consider $A\wr\Si_d$ as a superalgebra with 
multiplication defined by the rule
\[
(\sigma \otimes a)(\sigma' \otimes b) = \sigma \sigma' \otimes (a\cdot \sigma') b
\]
for $\sigma, \sigma' \in \Si_d$, $a,b \in A$.  
In what follows, we will 
identify $A^{\otimes d}$ (resp.\ $\k\Si_d$) with the subsuperalgebra 
$ 1\otimes A^{\otimes d}$ (resp.~$\k \Si_d\otimes 1$) of $A\wr\Si_d$.

\begin{example}[Sergeev superalgebra]
If $A = \k$, then $\k \wr \Si_d = \k \Si_d$, the group 
algebra of $\Si_d$.  On the other hand, if we identity $\C_d$
with $\C_1^{\otimes d}$ via the isomorphism in Example \ref{ex:cliff} 
then $\C_1 \wr \Si_d = \W_d,$ the {\em Sergeev superalgebra}
(cf. \cite{BKprojective}). 
\end{example}

Recall that a superalgebra is called {\em simple} if it has no
nonzero proper superideals. E.g., $\k$ and $\C_1$ are both
simple superalgebras.
Suppose now that $B$ is a simple finite dimensional superalgebra. 
For each $d \in \Z_{\ge 0}$, we define 
a new category $\mathsf\Gamma^d_{\hmm B}= \mathsf\Gamma^d(\smod\hpp B)$. 
 The objects of $\mathsf\Gamma^d_{\hmm B}$ 
are the same as those of $\smod\hpp B$, i.e. finite dimensional right 
$B$-supermodules.  Given $M,N \in \smod\hpp B$, set
$$\hom_{\mathsf\Gamma^d_{\hmm B}}(M,N) := \Gamma^d \Hom_B(V,W).$$  
In order to define a composition law, we need the 
following:
\begin{lemma}[\cite{Axtell}, Lemma 3.1] \label{lem:Ax}
Let $V\in \smod\hpp B$.  Then $V^{\otimes d} \in \smod\hpp {B\wr\Si_d}$. 
Furthermore if $V,W \in \smod\hpp B$, then we have a natural isomorphism
\begin{equation}\label{eq:isom2}
\Hom_{B\wr\Si_d}(V^{\otimes d},W^{\otimes d})\ \simeq\ \Gamma^d\Hom_B(V,W).
\end{equation}
\end{lemma}
Using the isomorphism (\ref{eq:isom2}) for any $V,W \in \smod\hpp B$, 
composition in $\smod\hpp {B\wr\Si_d}$ then induces a composition law 
in $\mathsf\Gamma^d_{\hmm B}$.  In particular, we are concerned with the categories
$$\mathsf\Gamma^d_\M = \mathsf\Gamma^d_{\k}\quad \mbox{ and }
\quad \mathsf\Gamma^d_\Q= \mathsf\Gamma^d_{\C_1}.$$
This notation comes from the fact that $\k \cong \mathcal{M}_{1|0}$ and $\C_1\cong \mathcal{Q}_1$ are
the smallest examples, respectively, in the series $\mathcal M_{m|n}$, $\mathcal Q_n$ of all 
finite dimensional simple superalgebras.  (See Example 2.4 and Remark 2.8 of 
\cite{Axtell} for more details.)

Clearly, $\mathsf\Gamma^d_\M$ and $\mathsf\Gamma^d_\Q$ are both 
$\Svec$-enriched categories.    
\begin{definition}
Let $\Pol^\I_d = \mathrm{Fct}_\k(\mathsf\Gamma^d_\M, \svec)$, the category 
of all even $\k$-linear functors from $\mathsf\Gamma^d_\M$ to $\svec$.  
Similarly, let $\Pol^\II_d= \Fct_\k(\mathsf\Gamma^d_\Q,\svec)$.  In both cases, 
morphisms are $\Svec$-enriched natural transformations (see Section \ref{ss:svec}). 
The objects of $\Pol^\I_d, \Pol^\II_d$ are called {\em (homogeneous, 
strict) polynomial superfunctors of type $\I$, $\II$,} respectively.   
\end{definition}
Notice that $\Pol_d^\I$ and $\Pol_d^\II$ are both $\Svec$-enriched categories, 
and the underlying even subcategories $(\Pol_d^\I)_\ev$ and $(\Pol_d^\II)_\ev$
are abelian.  In \cite[Section 4.2]{Axtell}, the objects of $\Pol^\II_d$ are also 
referred to as {\em spin polynomial functors}.

\begin{example}
In the category $\Pol^\I_1$, we have the identity $I: \svec\to \svec$.  The (right) parity 
change functor, $\Pi: \svec\to \svec$, is also seen to be an object of
$\Pol^\I_1$.  Notice that we have an {\em odd} $\Svec$-enriched natural 
isomorphism $id: I \to \Pi$, where $id_M: M \to \Pi M$ is just the identity
map on $M$.
\end{example}

The functor, $F_\C:\smod\hpp {\C_1}\to \svec$, which forgets the $\C_1$-action, 
induces a forgetful functor $F_{\C,d}:\mathsf\Gamma^d_\Q \to \mathsf\Gamma^d_\M$, 
for any $d\geq 0$.  Composition with $F_{\C,d}$ then induces a functor 
$$\text{Res}_{\C,d} = -\circ F_{\C,d}: \Pol^\I_d \to \Pol^\II_d.$$
We again let $\Pi$ denote the functor $\text{Res}_{\C,1}(\Pi)\in \Pol^\II_1$.

\begin{example}\label{ex:object}
Let $d\in \Z_{\ge 0}$, and suppose $V\in \Gamma^d_\mathtt{X}$ ($\mathtt{X} =\M,\Q$).  
Then we define the following objects of $\Pol^\I_d$, $\Pol^\II_d$: 
$$ \Gamma^{d,V}:= \hom_{\mathsf{\Gamma}^d_\mathtt{X}}(V, \cdot),\qquad \qquad 
S^d_V:= \hom_{\mathsf{\Gamma}^d_\mathtt{X}}(\cdot, V^\ast)^\ast.$$
Let us set 
$\Gamma^d= \Gamma^{d,\k}$ and $\Gamma_\Pi^d = \Gamma^{d, \k^{0|1}}$
in $\Pol^\I_d$, while in the category $\Pol^\II_d$ we write
$\Gamma^d =  \Gamma^{d,\C_1}$ and $\Gamma^d_\Pi = \Gamma^{d, \Pi \C_1}$.
There is clearly an even isomorphism 
$\Gamma^d \simeq \Gamma_\Pi^d$ of objects in 
$\Pol^\II_d$,  since $\Pi \C_1 \simeq \C_1$ in $\Gamma^d_\Q$. 
Notice that $\Gamma^1 = I$, $\Gamma^1_\Pi = \Pi$ in both $\Pol^\I_d$ and $\Pol^\II_d$. 
The objects $S^d$ and $S_\Pi^d$ in $\Pol^\I_d$, $\Pol^\II_d$ are defined in a
similar way. 
\end{example}

We recall that the {\em Kuhn dual} of $T\in \Pol^{\dagger}_d$ ($\dagger= \I,\II$) 
is the object $T^\#\in \Pol^{\dagger}_d$  which is defined so that
$T^\#(V)= T(V^{\asterisk})^\asterisk$ for all $V\in \mathsf\Gamma^d_\texttt{X}$ ($\texttt{X}=\M,\Q$). Cleary, we have
a canonical isomorphism
\begin{equation}\label{eq:dual}
S^d \simeq (\Gamma^d)^\#
\end{equation}
of objects in $\Pol^{\dagger}_d$.

We may also form the tensor product of polynomial superfunctors.  
Suppose given nonnegative integers $d$ and $e$.  Then, as in (\ref{eq:embed}), the embedding 
$\Si_d \times \Si_e \hookrightarrow \Si_{d+e}$ induces an embedding 
\begin{equation}\label{eq:embed2}
\Gamma^{d+e}M \hookrightarrow \Gamma^{d}M\otimes \Gamma^{e}M
\end{equation}
 for any $M\in \svec$.  
We may also consider the categories $\mathsf\Gamma^d_\M \otimes \mathsf\Gamma^e_\M$, 
$\mathsf\Gamma^d_\Q\otimes \mathsf\Gamma^e_\Q$ whose objects are the same as $\svec_\k$, $\smod\hpp {\C_1}$ 
and whose morphisms are of the form
$$\hom_{\mathsf\Gamma^d_\M}(M,N)\otimes \hom_{\mathsf\Gamma^e_\M}(M,N), 
\quad \hom_{\mathsf\Gamma^d_\Q}(V,W)\otimes \hom_{\mathsf\Gamma^e_\Q}(V,W)$$  
respectively for $M,N \in \svec_\k$ and $V,W\in \smod\hpp {\C_1}$.  Here, we assume that 
the composition of morphisms obeys the usual rule of signs conventions as in (\ref{eq:signs}).  

One may show that (\ref{eq:embed2}) yields embeddings of categories
\begin{equation}\label{eq:embed3}
 \mathsf\Gamma^{d+e}_\M \hookrightarrow \mathsf\Gamma^d_\M \otimes \mathsf\Gamma^e_\M, \qquad  
 \mathsf\Gamma^{d+e}_\Q \hookrightarrow \mathsf\Gamma^d_\Q \otimes \mathsf\Gamma^e_\Q.
\end{equation}
Then, as in Section \ref{ss:poly}, (\ref{eq:embed3}) yields induced bifunctors:
$$-\otimes-: \Pol^{\dagger}_d \times \Pol^{\dagger}_e \rightarrow \Pol^{\dagger}_{d+e}.$$
For example, if we let $\alpha\in (\Z_{\geq 0})^n$, then we have objects 
$$S^\alpha =\ S^{\alpha_1} \otimes\ \cdots\ \otimes S^{\alpha_n},  \qquad 
\Gamma^\alpha =\ \Gamma^{\alpha_1} \otimes\ \cdots\ \otimes \Gamma^{\alpha_n}$$  
 which belong to $ \Pol^{\dagger}_d$ $(\dagger= {\I,\II})$, where
$d= |\alpha|$.  
Also, we have the objects $I^{\otimes d} = I \otimes \cdots \otimes I$ and 
$\Pi^{\otimes d} = \Pi \otimes \cdots \otimes \Pi$ which belong to $\Pol^\I_d$.  
Since $1: \Pi \to I$ is an odd $\Svec$-enriched natural transformation, we have an isomorphism 
\begin{equation}\label{eq:tensor}
1^{\otimes d}: \Pi^{\otimes d}\ \cong\ I^{\otimes d}
\end{equation}
which is an odd (resp.~even) morphism if $d$ is odd (resp.~even).

Notice that the canonical embedding $\Delta:\Gamma^dM \hookrightarrow M^{\otimes d}$ and
surjection $m:M^{\otimes d}\twoheadrightarrow S^dM$ of 
superspaces induce even morphisms 
\begin{equation}\label{eq:canonical}
\Delta:\Gamma^d \hookrightarrow I^{\otimes d}\quad  \text{ and }\quad m:I^{\otimes d} \twoheadrightarrow S^d,
\end{equation}
respectively.  
It follows from (\ref{eq:tensor}) and (\ref{eq:canonical})  
that we have even morphisms:
\begin{equation}\label{eq:tensor2}
 \Delta_\Pi:\Gamma^d_\Pi \hookrightarrow I^{\otimes d}\quad \text{and} \quad
 m_\Pi:I^{\otimes d} \twoheadrightarrow S^d_\Pi,
\end{equation}
where
$$ \Delta_\Pi = 
\begin{cases}
\Delta &\text{if }d\text{ is even},\\
\Pi\circ \Delta &\text{if }d\text{ is odd},
\end{cases} \quad \text{and}   \quad m_\Pi = 
\begin{cases}
m &\text{if }d\text{ is even},\\
\Pi\circ m &\text{if }d\text{ is odd}.
\end{cases}
$$
Also, for any any $\alpha\in (\Z_{\geq 0})^n$, set 
$$\Gamma^\alpha_\Pi := \Gamma^{\alpha_1}_\Pi \otimes \cdots \otimes \Gamma^{\alpha_n}_\Pi, \qquad 
S^\alpha_\Pi := S^{\alpha_1}_\Pi\otimes \cdots \otimes S^{\alpha_n}_\Pi.$$ 
Then, by some more abuse  of notation, we also have canonical even morphisms: 
$\Delta_\Pi: \Gamma^\alpha_\Pi \hookrightarrow I^{\otimes |\alpha|}$ and 
$m_\Pi:I^{\otimes |\alpha|} \twoheadrightarrow S^\alpha_\Pi,$
given by taking the tensor products of the canonical maps (\ref{eq:tensor2}).
\smallskip

\subsection{Schur and Weyl superfunctors}\label{ss:superSchur}
We now proceed to construct super analogues of the Schur and Weyl functors.  
As we will see in 
Proposition \ref{prop:SchurComplex} below, the image of a Schur superfunctor is canonically 
isomorphic to a corresponding Schur complex.
\smallskip

Suppose now that $\mu \subset \lambda$ are partitions and $|\lambda| - |\mu|= d$.
Let us write $X^{\lambda/\mu} = X^{\lambda-\mu}$, for $X=\Gamma, S$, respectively. 
We then define an $\Svec$-enriched natural transformation 
$\hat{\theta}_{\lambda/\mu}: \Gamma^{\lambda'/\mu'}_\Pi\to S^{\lambda/\mu}$
as the composition of morphisms
$$ \Gamma_\Pi^{\lambda'/\mu'}\ 
 \stackrel{\Delta_\Pi}{\relbar\joinrel\relbar \joinrel \longrightarrow}\ 
I^{\otimes d}\ 
 \stackrel{\sigma_{\lambda/\mu}}{\relbar \joinrel \relbar \joinrel \longrightarrow}\
I^{\otimes d}\ 
 \stackrel{m}{\relbar \joinrel\relbar \joinrel\to}\ S^{\lambda/\mu},$$
where
the action of  $\sigma_{\lambda/\mu}$ on $I^{\otimes d}$ is induced by the action of $\Si_d$ on 
the superspace $M^{\otimes d}$, for any $M\in \Gamma^d_\texttt{X}$ ($\texttt{X}= \M,\Q$), respectively.

Recall the restriction functor, $\text{Res}_{\C,d}: \Pol^\I_d \to \Pol^\II_d$, induced by the forgetful
functor, $F_{\C,d}: \mathsf\Gamma^d_\Q \to \mathsf\Gamma^d_\M$, described above.  

\begin{definition}
The {\em Schur superfunctor, $\widehat{S}_{\lambda/\mu} = \widehat{S}_{\lambda/\mu}^\I$, 
of type $\I$} is the polynomial superfunctor defined by setting $\widehat{S}_{\lambda/\mu}(M)$ 
equal to the image of the map $\hat{\theta}_{\lambda/\mu}(M)$ for any $M\in \mathsf\Gamma^d_\M$.
The {\em Schur superfunctor of type $\II$} is defined by restriction: 
$\widehat{S}_{\lambda/\mu}^\II = \text{Res}_{\C,d}(\widehat{S}_{\lambda/\mu})$.
\end{definition}

We may also define an $\Svec$-enriched natural transformation 
$\check{\theta}_{\lambda/\mu}: \Gamma^{\lambda'/\mu'}\to S^{\lambda/\mu}_\Pi$
as the composition
$$ \Gamma^{\lambda'/\mu'}\ \stackrel{\Delta}{\relbar \joinrel \longrightarrow}\ I^{\otimes d}\  
 \stackrel{\sigma_{\lambda'/\mu'}}{\relbar \joinrel\relbar \joinrel \relbar \joinrel \longrightarrow}\
I^{\otimes d}\ 
 \stackrel{m_\Pi}{\relbar \joinrel\relbar \joinrel\relbar \joinrel \to}\ S^{\lambda/\mu}_\Pi,$$
where each map is defined in a similar way to the above.

\begin{definition}
The {\em Weyl} (or {\em co-Schur}) {\em superfunctor, 
$\widehat{W}_{\lambda/\mu}=\widehat{W}_{\lambda/\mu}^\I$, 
of type $\mathsf{I}$} is defined by setting $\widehat{W}^\I_{\lambda/\mu}(M)$ equal to the image of 
$\check{\theta}_{\lambda/\mu}(M)$, and the {\em Weyl superfunctor of type $\II$} is given by: 
$\widehat{W}_{\lambda/\mu}^\mathsf{II} = \text{Res}_{\C,d}(\widehat{W}_{\lambda/\mu})$.
\end{definition}

\begin{remark}
Notice that the polynomial superfunctors $\Pi$ and ${I}$ are both canonically isomorphic
to their own Kuhn duals.  It is also not difficult to check that $(S\otimes T)^\#$ is canonically
isomorphic to $S^\#\otimes T^\#$ for all $S,T \in \Pol^\mathsf{I}_\k$. 
It then follows from the definition of the morphisms $\hat{\theta}_{\lambda/\mu}$ and 
$\check{\theta}_{\lambda/\mu}$ and from (\ref{eq:dual}) that there is a canonical
isomorphism 
\begin{equation*}
\widehat{W}_{\lambda'/\mu'} \simeq (\widehat{S}_{\lambda/\mu})^\#
\end{equation*}
which is analogous to the usual duality between Schur and Weyl functors
(cf. \cite[Prop.~II.4.1]{ABW}).
For this reason, we consider only Schur superfunctors in the remainder. 
\end{remark}

\begin{proposition}\label{prop:SchurComplex}
Suppose $M\in \mathsf\Gamma^d_\M$.  Then there is a canonical isomorphism
$$SC_{\lambda/\mu}(\hpp\underline{M_0},\hp \underline{M_1}\hpp) \ \cong\ \underline{\widehat{S}_{\lambda/\mu}(M)}$$
of vector spaces.  
Furthermore, there are natural isomorphisms
$$ \widehat{S}_{\lambda/\mu}(M_0)\ \cong\ S_{\lambda/\mu}(\underline{M_0}) \quad
\text{and} \quad \widehat{S}_{\lambda/\mu}(M_1) \cong W_{\lambda/\mu}(\underline{M_1}),$$
where $S_{\lambda/\mu}$, $W_{\lambda/\mu}$ are the ordinary Schur and 
Weyl functors respectively.
\end{proposition}

\begin{proof}
The first statement follows by comparing the definitions of the natural transformations 
$\theta^\mathsf{sc}_{\lambda/\mu}$ and $\hat{\theta}_{\lambda/\mu}$.  The 
second part then follows from the isomorphisms given 
in (\ref{eq:SchurWeylComplex}) above.
\end{proof}
\smallskip

\subsection{The Standard Basis Theorem}\label{ss:standard}

Let us introduce the relevant tableaux with which we shall need to work.  
Let $\Zc$ be a $\Z_2$-graded set $\Zc= \Zc_0 \sqcup \Zc_1$.  Assume $\Zc$ has a total order, $<$.
A {\em tableaux}, $\mathfrak t$, of shape $\lambda/\mu$ with values in $\Zc$ is a function 
$\mathfrak t: \triangle_{\lambda/\mu} \to \Zc$. Let $\text{Tab}_\Zc({\lambda/\mu})$ denote the set of all
$\Zc$-valued  tableaux.

\begin{definition}
Suppose $\mathfrak t\in \text{Tab}_{\Zc}({\lambda/\mu})$. 
\begin{itemize}
\item[1.] The tableau $\mathfrak{t}$ is called {\em row standard} (resp.~{\em row costandard}) 
if the entries in each row are 
nondecreasing and repeated entries, if any, occur only among values in $\Zc_0$ (resp.~$\Zc_1$).  

\item[2.] It is called {\em column standard} (resp.~{\em column costandard}) 
if the columns are nondecreasing and any repeated entries occur among the $\Zc_1$
(resp.~$\Zc_0$).

\item[3.] The tableau $\mathfrak{t}$ is called {\em standard} if it is both row standard and 
column standard,
and it is called {\em costandard} if it is both row costandard and column costandard.  

\end{itemize}
\end{definition}
\smallskip

Suppose $M\in \svec$, and let $\Xc=(X_1,\dots, X_m)$ (resp. $\Yc=(Y_1, \dots, Y_n)$)
be an ordered basis of $M_0$ (resp. $M_1$).  Then we fix some total order on the
union $\Zc= \Xc\sqcup \Yc =(Z_1, \dots, Z_{m+n})$ 
which preserves the respective orderings in $\Xc$ and $\Yc$.  
This gives a $\Z_2$-graded set with $\Zc_0=\Xc$, $\Zc_1=\Yc$. 
Now, since we are dealing the Schur superfunctors, we will also need
to consider the vector superspace $\Pi M$ with basis $\Zc_\pi= \{\Pi Z_1, \dots, \Pi Z_{m+n}\}$,
which is again totally ordered by the indices.  
Then $\Zc_\pi$ is  a  $\Z_2$-graded set with $(\Zc_\pi)_0 = \{\Pi Y_1, \dots, \Pi Y_n\}$ and
$(\Zc_\pi)_1 = \{\Pi X_1, \dots, \Pi X_m\}$. 
\smallskip

In working with tableaux, it will be more convenient to deal with an 
indexing set in place of the basis $\Zc$. 
Let us consider the set $[m+n]= [m+n]_\Zc = \{1, \dots, m+n\}$ as a $\Z_2$-graded
set with 
$[m+n]_\varepsilon=  \{1\leq k \leq m: |Z_k| =\varepsilon\}$,  for $\varepsilon \in \Z_2$, 
and with the usual total order on integers.  
It is clear that  $\text{Tab}_{[m+n]}(\lambda/\mu)$ is in bijection with $\text{Tab}_{\Zc}(\lambda/\mu)$,
and this bijection preserves row (co)standardness and  column (co)standardness of tableaux.

For any $\mathfrak t \in \text{Tab}_{[m+n]}(\lambda/\mu)$, let $w(\mathfrak t) \in I(m|n,d)$ denote
the {\em reading word} of $\mathfrak t$ obtained by reading entries 
from left to right and top to bottom.   I.e., $w(\mathfrak t) = (w_1,w_2, \dots, w_d)$ with 
$w_1 = \mathfrak t(1,\mu_1+1)$, $w_2 = \mathfrak t(1,\mu_1+2)$, \dots, 
$w_d = \mathfrak t(q, \lambda_q)$, where $q = l(\lambda)$. 
Note that this gives a bijection 
$$w(\ ): \text{Tab}_{[m+n]}(\lambda/\mu) \stackrel{\sim}{\rightarrow} I(m|n,d)$$
if $d = |\lambda/\mu|$.

Now fix a tableau $\mathfrak t\in \text{Tab}_{[m+n]}(\lambda/\mu)$.  
Recalling the notation of Section ~\ref{ss:divided}, we introduce the following notation for 
 elements
$$  Z^{(\mathfrak t^i)}= Z^{(w(\mathfrak t^i))}  \in \Gamma^{\lambda_i-\mu_i}M\quad \text{ and } \quad
Z^{\mathfrak t^i}  = Z^{w(\mathfrak t^i)} \in M^{\otimes \lambda_i-\mu_i}.$$
If $q= l(\lambda)$ and $d= |\lambda/\mu|$, then we further write 
$$  Z^{(\mathfrak t)} =  Z^{(\mathfrak t^1)}\otimes \dots \otimes 
 Z^{(\mathfrak t^q)}\hp \in\hpp \Gamma^{\lambda/\mu} \quad \text{ and }
 Z^{\mathfrak t} = Z^{\mathfrak t^1}\otimes \dots \otimes Z^{\mathfrak t^q}\hp \in\hpp M^{\otimes d}.
$$
Now we may also use tableaux in $\text{Tab}_{[m+n]}(\lambda/\mu)$ to parametrize elements of
$\Gamma^{\lambda/\mu}_\Pi M$ and $(\Pi M)^{\otimes d}$.  I.e., we set
$$ Z_{\Pi}^{(\mathfrak t^i)}= (\Pi Z)^{(w(\mathfrak t^i))}  \in \Gamma_\Pi^{\lambda_i-\mu_i}M \quad \text{ and }
\quad Z_\Pi^{\mathfrak t^i}  = (\Pi Z)^{w(\mathfrak t^i)} \in (\Pi M)^{\otimes \lambda_i-\mu_i}, 
$$
and we further write 
$$
 Z_\Pi^{\mathfrak t} = Z_\Pi^{\mathfrak t^1}\otimes \dots \otimes Z_\Pi^{\mathfrak t^q}\hp \in\hpp (\Pi M)^{\otimes d}, \qquad
 {} Z_\Pi^{(\mathfrak t)} =  Z_\Pi^{(\mathfrak t^1)}\otimes \dots \otimes 
 Z_\Pi^{(\mathfrak t^q)}\hp \in\hpp \Gamma_\Pi^{\lambda/\mu}.
$$

It follows from Section \ref{ss:divided} that the set of $Z^{(\mathfrak t)}$ such that 
$\mathfrak t\in \text{Tab}_{[m+n]}(\lambda/\mu)$ is row standard (resp.~row costandard)
forms a basis of $\Gamma^{\lambda/\mu}M$ (resp.~$\Gamma_\Pi^{\lambda/\mu}M$).
Also, it is clear that the sets
$\{Z^{\mathfrak t}\}$ and $\{Z_\Pi^{\mathfrak t}\}$, such that $\mathfrak t\in \text{Tab}_{[m+n]}(\lambda/\mu)$,  
give  bases for $M^{\otimes d}$  and $(\Pi M)^{\otimes d}$, respectively.
\smallskip

\begin{definition}
Let $\Zc$ be a $\Z_2$-graded set with a total order, $z_1< \cdots < z_r$, of its elements.  
Suppose $\mathfrak{t}\in \text{Tab}_\Zc({\lambda/\mu})$, and let $p,q$ be positive integers. 
Define $\mathfrak{t}_{p,q}$ to be the number of times the elements $z_1, \dots, z_q$ appear as 
entries in the first $p$ rows of $\mathfrak{t}$. More precisely, $\mathfrak{t}_{p,q}$ is equal to the 
cardinality of  $\{ (i,j)\in \triangle_{\lambda/\mu}:\ 1\leq i \leq p \text{ and } \mathfrak{t}(i,j)\in 
\{z_1, \dots, z_q\}\}$. If $\mathfrak{s}$ is another 
tableau, we say $\mathfrak{s}\trianglelefteq \mathfrak{t}$ if $\mathfrak{s}_{p,q}\geq \mathfrak{t}_{p,q}$ 
for all $p,q$.  We say $\mathfrak{s}\triangleleft\mathfrak{t}$ if $\mathfrak{s}\trianglelefteq \mathfrak{t}$ and 
$\mathfrak{s}_{p,q} > \mathfrak{t}_{p,q}$ for at least one pair $p,q$.  
\end{definition}

\begin{remark}\label{rmk:quasi-order}
The above puts a quasi-order on the set $\text{Tab}_{\Zc}(\lambda/\mu)$:   
$\trianglelefteq$ is reflexive and transitive, 
but we may have $\mathfrak{t}\trianglelefteq \mathfrak{s}$ and $\mathfrak{s}\trianglelefteq \mathfrak{t}$ with 
$\mathfrak{s}\neq \mathfrak{t}$.  We do however obtain a 
genuine partial order by restricting either to the set of all row (co)standard tableaux, 
or to the set of all column (co)standard tableaux, of shape $\lambda/\mu$.
\end{remark}

Notice that the definition of the quasi-order above depends only on the total order, $<$, of $\Zc$
and not on its $\Z_2$-grading.  
The statement and proof of the following lemma are thus essentially the 
same as given in Lemma II.2.14 of \cite{ABW}.

\begin{lemma}\label{lem:exchange}
Let $\mathfrak{t}\in \text{Tab}_{\Zc}(\lambda/\mu)$, and let $\mathfrak{s}$ 
be the tableau, also of shape $\lambda/\mu$, formed by exchanging certain entries from the 
$k$-th row of $\mathfrak{t}$, say $\mathfrak{t}(k,h_1),$ $\dots,$ $\mathfrak{t}(k,h_\alpha)$, with
certain entries from the $(k+1)$-st row of $\mathfrak{t}$, say $\mathfrak{t}(k+1,l_1),$ $\dots,$
$\mathfrak{t}(k+1,l_\alpha)$, such that
$\mathfrak{t}(k+1,l_\nu) < \mathfrak{t}(k,h_\nu)$, for $\nu=1, \dots, \alpha$.
More precisely, for 
$(i,j)\notin \{(k,h_1), \dots, (k,h_\alpha), (k+1,l_1), \dots (k+1, l_\alpha)\},$ 
$\mathfrak{s}(i,j)= \mathfrak{t}(i,j)$, but  $\mathfrak{s}(k,h_\nu) = \mathfrak{t}(k+1, l_\nu)$ and 
$\mathfrak{s}(k+1, l_\nu)= \mathfrak{t}(k,h_\nu)$, for $\nu = 1, \dots, \alpha$.  Then 
$\mathfrak{s}\triangleleft\mathfrak{t}$.
\end{lemma}

Now suppose $\alpha^1, \dots, \alpha^r \in \Z_+^\infty$, and let $\alpha = \sum \alpha^i$,
so that $\alpha_j= \alpha^1_j + \cdots + \alpha^r_j$ for $j=1, \dots, r$.  Then there is a
canonical even morphism
\begin{equation}\label{eq: } 
S^{\alpha^1} \otimes S^{\alpha^2}\otimes \cdots \otimes S^{\alpha^r} \twoheadrightarrow S^\alpha,
\end{equation}
which is given by first applying iterations of the super twists followed by iterations of the multiplication.

\begin{lemma}\label{lem:factor}
The natural transformation $\hat{\theta}_{\lambda'/\mu'}$ can be factored as follows:
\begin{align*}
&\Gamma^{\lambda_1-\mu_1}_\Pi \otimes \cdots \otimes \Gamma^{\lambda_i-\mu_i}_\Pi \otimes 
\Gamma^{\lambda_{i+1}-\mu_{i+1}}_\Pi \otimes \cdots \otimes \Gamma^{\lambda_q-\mu_q}_\Pi\ 
 \\[.18cm]%
&\qquad\ \stackrel{\alpha}{\longrightarrow}\ \Gamma^{\lambda_1-\mu_1}_\Pi \otimes \cdots \otimes\ S^{(\lambda_i,\lambda_{i+1})'/(\mu_i,\mu_{i+1})'} \otimes \cdots \otimes \Gamma^{\lambda_q-\mu_q}_\Pi\\[.18cm]%
&\qquad\ \stackrel{\beta}{\longrightarrow}\ S^{(\lambda_1,\dots,\lambda_{i-1})'/(\mu_1,\dots,\mu_{i-1})'} \otimes\ 
S^{(\lambda_i,\lambda_{i+1})'/(\mu_i,\mu_{i+1})'}
\otimes\ S^{(\lambda_{i+2},\dots,\lambda_q)'/(\mu_{i+2},\dots,\mu_q)'}\\[.16cm]%
&\qquad\ \stackrel{\gamma}{\longrightarrow}\ S^{\lambda'/\mu'},
\end{align*}
where $\alpha$ is the map 
$1\otimes \cdots \otimes \hat{\theta}_{(\lambda_i,\lambda_{i+1})/(\mu_i,\mu_{i+1})}\otimes\cdots \otimes 1$, 
$\beta$ is the map 
$\hat{\theta}_{(\lambda_1, \dots, \lambda_{i-1})/(\mu_1,\dots,\mu_{i-1})} \otimes 1\otimes \hat{\theta}_{(\lambda_{i+2}, \dots, \lambda_n)/(\mu_{i+2},\dots,\mu_n)}$, 
and $\gamma$ is the map 
$$S^{\alpha^1}\otimes S^{\alpha^2}\otimes S^{\alpha^3}\twoheadrightarrow S^{\lambda'/\mu'}$$ 
given by (\ref{eq: }) with 
$\alpha^1= (\lambda_1, \dots, \lambda_{i-1})'- (\mu_1, \dots, \mu_{i-1})'$, $\alpha^2= (\lambda_i,\lambda_{i+1})'-(\mu_i,\mu_{i+1})'$ 
and 
$\alpha^3= (\lambda_{i+2},\dots, \lambda_n)'-(\mu_{i+2},\dots, \mu_n)'$.
\end{lemma}

We wish to define a natural transformation, $\diamondsuit_{\lambda/\mu}$, whose image lies in $\Gamma^{\lambda/\mu}_\Pi$ 
and such that 
$$\widehat{S}_{\lambda'/\mu'}\ \simeq\ \Gamma^{\lambda/\mu}_\Pi\hmm \left/\hpp \text{Im} \diamondsuit_{\lambda/\mu}. \right.$$
By the above lemma, we may reduce to the case where $\lambda=(\lambda_1,\lambda_2)$ and 
$\mu=(\mu_1,\mu_2)$.  Now if $\mu_1\geq \lambda_2$, then $S^{\lambda'/\mu'} \cong M^{\otimes d}$,
where $d=\lambda_1+\lambda_2-\mu_1-\mu_2$.
In this case, $\hat{\theta}_{\lambda/\mu}: \Gamma^{\lambda/\mu}_\Pi \to M^{\otimes d}$ is just the canonical 
embedding, 
and we have an isomorphism 
$\widehat{S}_{\lambda/\mu} \cong \Gamma^{\lambda/\mu}_\Pi$.  
Hence, we may further reduce to the case where $\mu=(\mu_1,\mu_2)$ 
$\subset (\lambda_1,\lambda_2)=\lambda$, 
with $\mu_1<\lambda_2$.

\begin{definition}\label{lem:image2}
Let $M\in \svec$, and suppose $\mu \subset \lambda$
are arbitrary partitions with $l(\lambda)=q$.  Suppose also
that $1\leq i \leq q-1$, and let $u,v$ be nonnegative integers such that 
$u+ v < \lambda_{i+1}-\mu_i$. 
Then we define a map $\diamondsuit_i(\lambda/\mu, u,v; M)$ as the composition
\begin{align*}
\big( \Gamma_\Pi^u \otimes \Gamma_\Pi^{p_1-u+p_2-v}\otimes\Gamma_\Pi^{v}\big)( M) 
 \stackrel{1\otimes \Delta \otimes 1}{\relbar\joinrel\relbar\joinrel\relbar\joinrel\relbar\joinrel\to}&\ 
 \big(\Gamma_\Pi^u \otimes \Gamma_\Pi^{p_1-u} \otimes \Gamma_\Pi^{p_2-v} 
 \otimes \Gamma_\Pi^{v}\big)(M)\\[.1cm]
\stackrel{m\otimes m}{\relbar\joinrel\relbar\joinrel\relbar\joinrel\relbar\joinrel\to}&\ 
\big(\Gamma_\Pi^{p_1}\otimes\Gamma_\Pi^{p_2}\big)( M),
\end{align*}
where $p_1 = \lambda_i-\mu_i$, $p_2= \lambda_{i+1}-\mu_{i+1}$ and  
$\Delta$ is the appropriate component of the comultiplication 
in $\Gamma^{\udot} (\Pi M)$, while $m$ denotes multiplication.
We then define the map ${\diamondsuit}_{\lambda/\mu}(M)$ to
be the sum of all these maps, i.e., 
$${\diamondsuit}_{\lambda/\mu}(M)\ =\ 
\sum_{i=1}^{q-1}\hpp \sum_{u,v}\hpp 
1_1 \otimes \cdots \otimes 1_{i-1}\otimes {\diamondsuit}_i(\lambda/\mu, u,v; M) 
\otimes 1_{i+2}\cdots \otimes 1_n,$$
where the inner sum is over pairs $u,v$ such that $0 \leq u+v < \lambda_{i+1}-\mu_i$. \vspace{.03cm}
This defines an $\Svec$- enriched natural transformation, $\diamondsuit_{\lambda/\mu}$, whose image lies
in $\Gamma^{\lambda/\mu}_\Pi$.
\end{definition}

Suppose now that $M\in \svec$ with ordered basis $\Zc = \Zc_0 \sqcup \Zc_1 = ( Z_1, \dots, Z_{m+n})$ as above.
We have the following.

\begin{lemma}[Straightening Law]\label{lem:straighten}
Suppose that $\mathfrak t \in \text{Tab}_{\Zc}(\lambda/\mu)$ is row costandard, but not costandard.
Then there are row costandard tableaux $\mathfrak t_l \in \text{Tab}_{\Zc}(\lambda/\mu)$ such
that $\mathfrak t_l \neq \mathfrak t$ with $\mathfrak t_l \triangleleft \mathfrak t$ and 
$Z_\Pi^{(\mathfrak t)} + \sum_l C_l Z_\Pi^{(\mathfrak t_l)}\in \text{Im}\hpp\hpp\diamondsuit_{\lambda/\mu}(M)$
for some integers $C_l$.
\end{lemma} 

\begin{proof}
Suppose first that $l(\lambda)=2$. Then there is
a column $j$ such that either: 
$$\mathfrak t(1,j) > \mathfrak t(2,j),\quad \text{ or }\quad
  \mathfrak t(1,j), \mathfrak t(2,j) \in [m+n]_1 
\text{ and } \mathfrak t(1,j)= \mathfrak t(2,j).$$  
Let $j_0$ denote the minimal such $j$ with this property.

Now set $u= j_0-\mu_1-1$ and $v= \lambda_2-j_0$. Let us also write \vspace{.03cm}%
$p_i= \lambda_i-\mu_i$ for $i=1,2$. \vspace{0.03cm}
If we let $g_1, \dots, g_{p_1}$ (resp.~$h_1, \dots, h_{p_2}$) denote the entries in 
the first (resp.~second) row, then the tableau $\mathfrak t$ has the form
$$\xy <.55cm,0cm>:
(2.25,2)*=0{}; (24,2)*=0{} **@{-} ,
(.5,1)*=0{}; (24,1)*=0{} **@{-} ,
(.5,0)*=0{}; (22.5,0)*=0{} **@{-} ,
(.5,1)*=0{}; (.5,0)*={} **@{-},
(2.25,2)*=0{}; (2.25,1)*={} **@{-},
(3.5,2)*=0{}; (3.5,1)*={} **@{-}, 
(6.25,2)*=0{}; (6.25,1)*={} **@{-}, 
(7.5,2)*=0{}; (7.5,0)*={} **@{-}, 
(12.5,2)*=0{}; (12.5,0)*={} **@{-}, 
(15.25,1)*=0{}; (15.25,0)*={} **@{-}, 
(17.5,1)*=0{}; (17.5,0)*={} **@{-}, 
(20.5,1)*=0{}; (20.5,0)*={} **@{-}, 
(22.5,1)*=0{}; (22.55,0)*={} **@{-}, 
(22,2)*=0{}; (22,1)*={} **@{-},
(24,2)*=0{}; (24,1)*={} **@{-},
(1.2,0.5)*=0{h_1},
(4.5,0.5)*=0{\cdots},
(1.75,1)*=0{}; (1.75,0)*={} **@{-},
(2.9,1.5)*=0{g_1},
(5,1.5)*=0{\cdots},
(6.9,1.5)*=0{g_{u}},
(10,1.5)*=0{g_{u+1}=g_{j_0-\mu_1}},
(17.5,1.5)*=0{\cdots},
(10,0.5)*=0{h_{j_0-\mu_2}\hmm = h_{p_2-v}},
(14,0.5)*=0{h_{p_2-v+1}},
(16.25,0.5)*=0{\cdots},
(19,0.5)*=0{h_{p_2-v+r}},
(21.5,0.5)*=0{\cdots},
(23,1.5)*=0{g_{p_1}}
\endxy$$
where $r$ is the largest nonnegative integer such that: 
$h_{p_2-v} = \cdots = h_{p_2-v+r}$. 
Recall from Section \ref{ss:divided} that 
$\Gamma_\Pi^u(M) \otimes \Gamma_\Pi^{p_1-u+p_2-v+r}(M)\otimes\Gamma_\Pi^{v-r}( M)$
has a basis of elements of the form
$$Z_\Pi^{(\bs a,\bs b,\bs c)}\hpp =\hpp\hpp Z_\Pi^{(\bs a)}\otimes Z_\Pi^{(\bs b)}\otimes Z_\Pi^{(\bs c)},$$ 
for standardized sequences
$\bs a, \bs b, \bs c \in I(m|n)$ \vspace{.05cm}%
such that 
$l(\bs a) = u$, $l(\bs b) = p_1 +p_2-u-v+r$, and $l(\bs c) = v-r$.  
Let us fix 
the sequences $\bs a, \bs b, \bs c$ determined by setting
$$a_j =  \mathfrak t(1,j),\quad  \mu_1+1\leq j < j_0;\qquad  c_j = \mathfrak t(2,j),\quad 
j_0+r+1\leq \lambda_2 < j_0;\quad \text{and}$$
$$b_j = \begin{cases}
  \mathfrak t(1,j),&  j_0\leq j \leq \lambda_1\\ 
\mathfrak t(2,j), & \mu_2+1\leq j < j_0+r.
\end{cases}
$$ 
The fact that $\bs a, \bs c$ are ${\Zc_\pi}$-restricted 
follows 
immediately from the fact that $\mathfrak t$ is row costandard.  On the other hand, 
one may check that $\bs b$ is $\Zc_\pi$-restricted by
using the row-constandardness of $\mathfrak t$ together with the fact that 
$h_{p_2-v-1}\leq g_u$, with $h_{p_2-v}= g_{u+1}$ only if both 
$g_u$ and $h_{p_2-v-1}$ belong to $[m+ n]_1$.

We now wish to show that
$\diamondsuit_1(\lambda/\mu, u,v ;M)(Z_\Pi^{(\bs a,\bs b, \bs c)})$ contains the 
basis element $Z_\Pi^{(\mathfrak t)}$ as a summand with coefficient $1$, and all other
nonzero summands correspond to row costandard tableaux occuring before 
$\mathfrak t$ in the order $\trianglelefteq$.
From Lemma \ref{lem:multiply}.(i), we have
$$
\diamondsuit_1(\lambda/\mu, u,v ;M)(Z_\Pi^{(\bs a,\bs b, \bs c)}) =\hpp
\sum_{\text{st}_{\Zc}(\bs b^1\vee\bs b^2)=\bs b} Z_\Pi^{(\bs a)} \coasterisk Z_\Pi^{(\bs b^1)} 
\otimes Z_\Pi^{(\bs b^2)} \coasterisk Z_\Pi^{(\bs c)}
$$
where the sum ranges over pairs of standardized sequences $\bs b^1$, $\bs b^2$.
To each  pair $\bs b^1, \bs b^2$, with $\text{st}_\Zc(\bs b^1\vee\bs b^2)= \bs b$, 
there corresponds a unique tableau $\mathfrak t({\bs b^1,\bs b^2})$ of shape $\lambda/\mu$ 
such that $w(\mathfrak t({\bs b^1,\bs b^2})) = \text{st}(\bs a \vee \bs b^1)\vee \text{st}(\bs b^2 \vee \bs c)$. 
Hence, from Lemma \ref{lem:multiply}.(ii) we have 
$$ Z_\Pi^{(\bs a)} \coasterisk Z_\Pi^{(\bs b^1)}  \otimes Z_\Pi^{(\bs b^2)} \coasterisk Z_\Pi^{(\bs c)} =\ 
C({\bs b^1,\bs b^2}) \  Z_\Pi^{(\bs a \vee \bs b^1)} \otimes Z_\Pi^{(\bs b^2\vee \bs c)}\hpp =\ 
\pm C({\bs b^1,\bs b^2}) \  Z_\Pi^{(\mathfrak t({\bs b^1, \bs b^2}))}
,$$ 
for some nonnegative integer $C({\bs b^1,\bs b^2})$.
If $\text{st}(\bs a \vee \bs b^1)$ and $\text{st}(\bs b^2 \vee \bs c)$ are both
$\Zc_\pi$-restricted, then $\mathfrak t({\bs b^1, \bs b^2})$ is row costandard and $C({\bs b^1,\bs b^2}) >0$.  
Otherwise, we have $Z_\Pi^{(\mathfrak t({\bs b^1, \bs b^2}))}\hmm =0$.
Now there is exactly one pair $\dot{\bs b}^1, \dot{\bs b}^2$ such that
$\mathfrak t = \mathfrak t({\dot{\bs b}^1,\dot{\bs b}^2})$.
Furthermore, all of the other tableaux $\mathfrak t({\bs b^1,\bs b^2}) \neq \mathfrak t$
are obtained from $\mathfrak t$ 
\smallskip 
by the type of exchange described in Lemma \ref{lem:exchange}, so that
$\mathfrak t({\bs b^1,\bs b^2}) \triangleleft \mathfrak t$.

It thus remains to show that $Z_\Pi^{(\mathfrak t)}$ occurs as a summand with coefficient $C({\dot{\bs b}^1, \dot{\bs b}^2})=1$. 
We claim that $g_u< g_{u+1}$.  Suppose not. Then $g_u = g_{u+1}\in [m+ n]_1$, since 
$\mathfrak t$ is row costandard. 
The fact that $j_0$ was chosen to be minimal implies that 
$g_u \leq h_{p_2-v-1} \leq h_{p_2-v} \leq g_{u+1}$, 
with $g_u= h_{p_2-v-1}$ only if $g_u, h_{p_2-v-1}\in [m+ n]_0$, a contradiction.
On the other hand, it follows by the definition of $r$ that $h_{p_2-v+r}<h_{p_2-v+r+1}$.
Hence, we have that $\text{Im}(\bs a)\cap\text{Im}(\dot{\bs b}^1)=\emptyset$ and
$\text{Im}(\dot{\bs b}^2)\cap \text{Im}(\bs c)=\emptyset$.  It then follows from
Lemma \ref{lem:multiply}.(ii) that the coefficient of $Z_\Pi^{(\mathfrak t)}$ is $1$.
This completes the proof when $l(\lambda)=2$.

Now suppose $l(\lambda)>2$.  Since $\mathfrak t$ is not costandard, there
must be some pair $(i,j), (i+1,j)\in \triangle_{\lambda/\mu}$ such that 
$\mathfrak t(i,j) > \mathfrak t(i+1,j)$, or such that 
$\mathfrak t(i,j) = \mathfrak t(i+1,j)$ and $\mathfrak t(i,j), \mathfrak t(i+1,j)\in [m+ n]_1$.
We may then apply the above argument to the tableau 
$\bar{\mathfrak t}$ which consists of  $i$-th and $(i+1)$-st rows of $\mathfrak t$.
Since the maps $\diamondsuit_i(\lambda/\mu, u,v ;M)(Z_\Pi^{(\bs a,\bs b, \bs c)})$
do not affect the other rows of $\mathfrak t$, we may again obtain row costandard tableaux 
$\mathfrak t_l\triangleleft \mathfrak t$ 
such that $\mathfrak t + \sum  C_l \mathfrak t_l \in 
\text{Im }\diamondsuit_i(\lambda/\mu, u,v ;M)
\subset \text{Im } \diamondsuit_{\lambda/\mu}(M)$.
\end{proof}
%

\begin{corollary}\label{cor:cosets}
The cosets of the elements $Z_\Pi^{(\mathfrak{t})}$, indexed by costandard tableaux $\mathfrak{t} \in 
\text{Tab}_{\Zc}(\lambda/\mu)$,  span
the superspace 
$\Gamma^{\lambda/\mu}_\Pi(M)\hpp \big{/}\hpp \text{Im}\hpp\hpp\diamondsuit_{\lambda/\mu}(M)$.
\end{corollary}

\begin{proof}
It is clear that the set of cosets, 
$Z_\Pi^{(\mathfrak t)} + \text{Im}\hpp\hpp\diamondsuit_{\lambda/\mu}(M)$, 
corresponding to row costandard 
 $\mathfrak t \in \text{Tab}_{\Zc}(\lambda/\mu)$,
generates 
$\Gamma^{\lambda/\mu}_\Pi(M)\hpp \big{/}\hpp \text{Im}\hpp\hpp\diamondsuit_{\lambda/\mu}(M)$.
The Corollary then follows from Lemma \ref{lem:straighten} by induction.
\end{proof}

\begin{proposition}\label{prop:kernel}
There is an embedding $\text{Im}\hpp\hpp\diamondsuit_{\lambda/\mu} \hookrightarrow 
\text{Ker} \hpp \hpp \hat{\theta}_{\lambda'/\mu'}$ in the category $(\Pol^\I_d)_\ev$.
\end{proposition}

\begin{proof}
Let $M\in \svec$. It suffices to show that 
$\text{Im}\hpp\hpp\diamondsuit_{\lambda/\mu}(M) \subset \text{Ker} \hpp \hpp \hat{\theta}_{\lambda/\mu}(M)$.
By Lemma \ref{lem:factor} and the definition of $\diamondsuit_{\lambda/\mu}(M)$, we may assume that $l(\lambda)=2$. 
Then suppose we are given standardized $\bs a, \bs b, \bs c \in I(m|n)$ which are  $\Zc_\pi$-restricted with  $l(\bs a) = u$, $l(\bs b)=p_1+p_2-u-v$, and $l(\bs c)= v$, where
$p_1= \lambda_1-\mu_1$, $p_2= \lambda_2-\mu_2$ and $u+v<\lambda_2-\mu_1$. 
Then there is a corresponding basis element
$$Z_\Pi^{(\bs a,\bs b,\bs c)}\hpp =\hpp\hpp Z_\Pi^{(\bs a)}\otimes Z_\Pi^{(\bs b)}\otimes Z_\Pi^{(\bs c)}\hpp \in \hppp
\Gamma_\Pi^u(M) \otimes \Gamma_\Pi^{p_1-u+p_2-v}(M)\otimes\Gamma_\Pi^{v}( M).$$ 
%
%
We must  show that the composition
$
\hat{\theta}_{\lambda/\mu}\circ\diamondsuit_{\lambda/\mu}(M)(Z_\Pi^{(\bs a,\bs b,\bs c)}) = 0.
$
\smallskip

Using Lemma \ref{lem:multiply}, (\ref{eq:div_tensor}) and (\ref{eq:Delta}), 
respectively, one may check that 
\begin{align}
\Delta_\Pi\circ\diamondsuit_{\lambda/\mu}(M)&(Z_\Pi^{(\bs a, \bs b,\bs c)})
=
\sum_{\bs b^1, \bs b^2} \text{sgn}_{\Zc_\pi}(\bs b^1; \bs b^2) \hppp
\Delta_\Pi(Z_\Pi^{(\bs a)} \coasterisk Z_\Pi^{(\bs b^1)}) \otimes
\Delta_\Pi(Z_\Pi^{(\bs b^2)} \coasterisk Z_\Pi^{(\bs c)}) \nonumber \\
=&
\sum_{\bs b^1, \bs b^2}\sum_{\sigma} \text{sgn}_{\Zc_\pi}(\bs b^1; \bs b^2) 
\big[ \Delta_\Pi(Z_\Pi^{(\bs a)}) \otimes \Delta_\Pi(Z_\Pi^{(\bs b^1)})\big].\sigma \otimes
\big[\Delta_\Pi(Z_\Pi^{(\bs b^2)}) \otimes \Delta_\Pi(Z_\Pi^{(\bs c)}) \big].\sigma' \nonumber \\ 
=&
 \sum_{\bs b^1,\bs b^2} \sum_{\rho,\rho'} \text{sgn}_{\Zc_\pi}(\bs b^1; \bs b^2)\hpp 
\hpp Z^{\bs a \vee \bs b^1}\hmm\hmm.\rho\hpp \otimes\hpp
 Z^{\bs b^2\vee \bs c}.\rho',   \label{eq:sum3}
\end{align}
summing over all pairs of standardized sequences $\bs b^1, \bs b^2$ such that 
$l(\bs b^1) = p_1-u$, $l(\bs b^2) = p_2-v$ and 
$\text{st}(\bs b^1\vee \bs b^2) = \bs b$, over all pairs of permutations
$\sigma \in \Si_{p_1}  /\Si_u\times \Si_{p_1-u}$ and $\sigma' \in \Si_{p_2}/ \Si_{p_2-v}\times\Si_v$,
and over all pairs
$\rho \in \Si_{p_1}/\Si_{\bs a} \times \Si_{\bs b^1}$ and $\rho'\in  \Si_{p_2}/ \Si_{\bs b^2} \times \Si_{\bs c}$, respectively.
\smallskip

Now let $\mathfrak t(\bs b^1, \bs b^2; \rho, \rho')$ denote the tableau of shape 
$\lambda/\mu$ whose first row equals $(\bs a \vee \bs b^1).\rho$ and with second row 
$(\bs b^2 \vee \bs c).\rho'$.
Let us also consider the map $\bar{\theta}_{\lambda/\mu}: \Gamma_\Pi^{\lambda/\mu} M
\to M^{\otimes d}$ defined as the composition
$$ \Gamma_\Pi^{\lambda/\mu}\ 
 \stackrel{\Delta_\Pi}{\relbar \joinrel \longrightarrow}\ 
I^{\otimes d}\ 
 \stackrel{\sigma_{\lambda/\mu}}{ \relbar \joinrel \longrightarrow}\
I^{\otimes d},$$
which is the first part of the composition forming the natural transformation 
$\hat{\theta}_{\lambda/\mu}:\Gamma_\Pi^{\lambda/\mu}\to S^{\lambda'/\mu'}$. 
Note that  $\hat{\theta}_{\lambda/\mu}(M) = m\circ \bar{\theta}_{\lambda/\mu}(M)$
where $m$ denotes tensor product of multiplications in $S^\udot M$.

Since $u+v < \lambda_2 - \mu_1$, it follows that, for each 
$\mathfrak t = \mathfrak t(\bs b^1, \bs b^2; \rho, \rho')$, 
there exists a pair $(1,j), (2,j)\in \triangle_{\lambda/\mu}$ such
that $\mathfrak t(1,j)$, $\mathfrak t(2,j)$ both come from entries of $\mathbf b$. 
Let $\check{\jmath}$ denote the minimal such $j$. 
We then define an involution on the set of all $\mathfrak t(\bs b^1, \bs b^2; \rho, \rho')$ 
given by $\mathfrak t \mapsto \check{\mathfrak t} = 
\check{\mathfrak t}(\check{\bs b}^1, \check{\bs b}^2; \check{\rho}, \check{\rho}') $, where $\check{\mathfrak t}$ 
is the tableau obtained from $\mathfrak t$ by interchanging the entries
in the $(1, \check{\jmath})$ and $(2, \check{\jmath})$ positions.
\smallskip

For each tableau $\mathfrak t = \mathfrak t(\bs b^1, \bs b^2; \rho, \rho')$ as above, 
let us denote
$$
\text{sgn}_{\pi}(\bs b^1, \bs b^2; \rho, \rho')\hpp :=\hpp \text{sgn}_{\Zc_\pi}(\bs b^1, \bs b^2)\times
\text{sgn}_{\Zc_\pi}(\bs a\vee \bs b^1; \rho )\times \text{sgn}_{\Zc_\pi}( \bs b^2\vee \bs c; \rho').
$$
From (\ref{eq:sgn}) and  (\ref{eq:sum3}), it then remains to show that the identity
$$ \text{sgn}_{\pi}(\bs b^1, \bs b^2; \rho, \rho')\hpp 
m(Z^{\mathfrak t}.\hpp\sigma_{\lambda/\mu})\ +\   
 \text{sgn}_{\pi}(\check{\bs b}^1, \check{\bs b}^2; \check{\rho}, \check{\rho}')\hpp 
 m(Z^{\check{\mathfrak t}}.\hpp\sigma_{\lambda/\mu})\hpp 
 =\hpp 0$$
holds for each $\mathfrak t = \mathfrak t(\bs b^1, \bs b^2; \rho, \rho')$, and this may be checked
 directly by computation.
\end{proof}

Let $M\in \svec$ with $\sdim(M)=m|n$, and fix a totally ordered $\Z_2$-homogenous basis $\Zc$ of $M$.
Also write $[m+n]= [m+n]_\Zc$ as above. 

\begin{theorem}[Standard Basis Theorem for Schur Superfunctors]\label{thm:standard}
Let $\mu\subset \lambda$ be partitions such that $|\lambda/\mu|=d$. 
The images $\hat{\theta}_{\lambda'/\mu'}(Z_\Pi^{(\mathfrak t)})$ of elements
indexed by costandard tableaux, $\mathfrak t\in \text{Tab}_{[m+n]}(\lambda/\mu)$, are linearly
independent and form a basis of $\widehat{S}_{\lambda'/\mu'}M$.  Moreover, we have an
even isomorphism
$$\widehat{S}_{\lambda'/\mu'} \simeq \Gamma^{\lambda/\mu}_\Pi\big{/} \text{Im}\hpp\hpp\diamondsuit_{\lambda/\mu}$$
in the category $\Pol^\I_d$.
\end{theorem}

\begin{proof}
It follows from the definition of $\widehat{S}_{\lambda/\mu}$, Corollary \ref{cor:cosets}, 
and Proposition \ref{prop:kernel} that there is a surjective map 
$$\Gamma^{\lambda/\mu}_\Pi \big{/} \text{Im}\hpp\hpp\diamondsuit_{\lambda/\mu}
\twoheadrightarrow \widehat{S}_{\lambda'/\mu'}$$
in the category $\Pol^\I_d$. 
Hence, the isomorphism will follow once we prove that $\hat{\theta}_{\lambda'/\mu'}(Z_\Pi^{(\mathfrak t)})$,
with $\mathfrak t$ costandard, are linearly independent in $\widehat{S}_{\lambda'/\mu'}(M)$.

To show linear independence, recall the surjection 
$m:M^{\otimes d} \twoheadrightarrow S^{\lambda'/\mu'}M$. Then the set of all
$$Z_{(\mathfrak s)}\hp :=\hpp m(Z^{\mathfrak s})\hpp \in\hp S^d M,$$ 
such that  
$\mathfrak s \in \text{Tab}_{[m+n]}(\lambda'/\mu')$ 
is row standard, forms a basis of $S^{\lambda'/\mu'}M$. 
It follows from Remark \ref{rmk:quasi-order} that $\trianglelefteq$ gives a total order
on these row-standard $\mathfrak s$.

Now given any tableau $\mathfrak t \in \text{Tab}_{[m+n]}(\lambda/\mu)$, we let
$\mathfrak t' \in \text{Tab}_{[m+n]}(\lambda'/\mu')$ denote the tableau of  
conjugate shape $\lambda'/\mu'$ obtained
by setting $\mathfrak t'(i,j) = \mathfrak t(j,i)$ for all $(i,j) \in \triangle_{\lambda'/\mu'}$. 
Clearly,  $\mathfrak t \in \text{Tab}_{[m+n]}(\lambda/\mu)$ is
costandard if and only if  $\mathfrak t'$ is standard. 
Furthermore, notice that
$$Z^{\mathfrak s}.\sigma_{\lambda'/\mu'} = \text{sgn}(\mathfrak t, \sigma_{\lambda'/\mu'}) Z^{\mathfrak s'}
\in M^{\otimes d},$$
for any $\mathfrak s \in \text{Tab}_{[m+n]}(\lambda/\mu)$, 
and the smallest element occurring in 
$\hat{\theta}_{\lambda'/\mu'}(Z^{(\mathfrak t)})$ with 
respect to the above order, $\trianglelefteq$, is 
\begin{equation*}
Z_{(\mathfrak t')} \in S^{\lambda'/\mu'}M.
\end{equation*} 
This follows, since the action of the map $\Delta_\Pi(M)$ is given
explicitly by 
$$\Delta_\Pi(M). Z_\Pi^{(\mathfrak t)} = \sum_{\sigma \in \Si_\lambda/\Si_\mathfrak{t}}
\text{sgn}_\pi(\mathfrak t; \sigma) Z^{\mathfrak t.\sigma}$$
where $\Si_{\mathfrak t} = \Si_{\mathfrak t^1}\times \dots \times \Si_{\mathfrak t^q}\subset$
$\Si_{\lambda_1} \times \dots \times \Si_{\lambda_q} = \Si_\lambda$.
Now there is exactly one coset in $\Si_\lambda/\Si_\mathfrak{t}$ which fixes $Z_\Pi^{\mathfrak t}$.
Every other representative $\sigma \in \Si_\lambda/\Si_\mathfrak{t}$ must put bigger elements
of $\mathfrak t$ to earlier columns.  
Therefore, for such $\sigma$ we obtain later elements $Z_{((\mathfrak t.\sigma)')}$, 
with respect to the order $\trianglelefteq$ in $S^{\lambda'/\mu'}M$. 

Moreover, one sees immediately that $Z_{(\mathfrak t')}$ occurs
in $\hat{\theta}_{\lambda'/\mu'}(Z^{(\mathfrak t)})$ with coefficient
$\pm 1$.  
It is also obvious that the initial elements $Z_{(\mathfrak t')}$
are different for different costandard tableaux $\mathfrak t$.
This proves that the images $\hat{\theta}_{\lambda'/\mu'}(Z^{(\mathfrak t)})$
corresponding to costandard $\mathfrak t$ are linearly independent.
\end{proof}


We next wish to provide a 
standard basis for the evaluation of the Schur 
superfunctor $\widehat{S}_{\lambda/\mu}^{\hp \II}$
at a supermodule $M\in \smod\hpp \C_1$.
\smallskip

Let $\U_r(1) \in \smod\hpp \C_1$ denote the Clifford superalgebra,
 $\C_1$, considered as a
supermodule over itself via right multiplication.  Since $\C_1$ is
a simple superalgebra, the supermodule $\U_r(1)$ is irreducible.
Recall from \cite[Ex.~2.6]{Axtell} that for every $V\in \smod\hpp \C_1$ there
is a $\C_1$-supermodule isomorphism
$$ V \simeq \U_r(1)^n = \U_r(1)^{\oplus n}$$
for some $n\geq 0$, so that $\sdim(V) = n|n$.

Now fix $V\in \smod\hpp \C_1$.  Then, as a vector superspace, we have 
$$\widehat{S}^{\hp \II}_{\lambda/\mu}(V)
= \text{Res}_{\C,d}(\widehat{S}_{\lambda/\mu})(V)
= \widehat{S}_{\lambda/\mu}(F_{\C,d}(V)).$$
As above, we choose a totally ordered $\Z_2$-homogeneous 
basis $\Zc$ of the superspace $F_{\C,d}(V)$, and let $[n+n] = [n+n]_\Zc$.
 We again denote basis elements of $\Gamma_\Pi^{\lambda/\mu}(F_{\C}(V))$ 
by $Z_\Pi^{(\mathfrak t)}$, for row costandard $\mathfrak t\in \text{Tab}_{[n+n]}(\lambda/\mu)$.
\smallskip

As a direct consequence of Theorem \ref{thm:standard}, we now have the following.

\begin{corollary}[Standard Basis for Schur Superfunctors of Type $\II$]\label{cor:standard}
The images $\hat{\theta}^{\hp \II}_{\lambda'/\mu'}(Z_\Pi^{(\mathfrak t)})$ of elements
indexed by costandard tableaux, $\mathfrak t\in \text{Tab}_{[n+n]}(\lambda/\mu)$, form 
a basis of $\widehat{S}^{\hp \II}_{\lambda'/\mu'}M$, where 
$\hat{\theta}^{\hp \II}_{\lambda'/\mu'}: 
\Gamma_\Pi^{\lambda'/\mu'} \to S^{\lambda/\mu}$ denotes the even 
morphism in $\Pol^\II_d$ formed by composition of 
$\hat{\theta}_{\lambda'/\mu'}$ with the functor 
$F_{\C,d}:\mathsf{\Gamma_\Q^d} \to \mathsf{\Gamma_\M^d}$.
\end{corollary}

\section{Filtrations of Bisuperfunctors}\label{Sec:bisuper}
In this section, we consider filtrations of the superspaces, $\widehat{S}_{\lambda/\mu}(M\oplus N)$,
for any given pair $M,N\in \svec$.  
This leads naturally to a filtration of {\em Schur  bisuperfunctors}.

\subsection{Strict polynomial bisuperfunctors}
Let $\mathsf\Gamma^d_{\M\times \M}= \mathsf\Gamma^d(\svec\times\svec)$ be the category with the same objects
as $\svec \times \svec$, i.e. pairs $(M,N)$ of finite dimensional superspaces,
and morphisms
$$ \hom_{\mathsf\Gamma^d_{\M\times \M}}\left((M,N),(M',N')\right) = \bigoplus_{e+f=d} 
\hom_{\mathsf\Gamma^e_\M}(M,M') \otimes \hom_{\mathsf\Gamma^f_\M}(N,N').$$
Here, the composition of morphisms in $\mathsf\Gamma^d_{\M\times\M}$ is given using the
rule of signs convention (as in (\ref{eq:signs})).
We then define the category $\mathsf{bi} \Pol^\I_d$ of {\em strict polynomial bisuperfunctors}
to be the category of all even $\k$-linear functors $T:\mathsf\Gamma^d_{\M\times \M} \to \svec$.  
One may also define categories $\mathsf\Gamma^d_{\Q\times \Q}$ and 
$\mathsf{biPol}^\II_d$ in a similar way.

Now suppose $S\in \Pol^\I_e$, $T\in \Pol^\I_f$. 
Then we define the {\em external tensor product} 
$S\boxtimes T:\mathsf\Gamma^{e+f}_{\M\times \M} \to \svec$ to be the bifunctor which acts on objects 
$(M,N) \in \mathsf\Gamma^{e+f}_{\M\times\M}$ by  
$$(S\boxtimes T)(M,N) = S(M) \otimes T(N).$$  Suppose $e',f'$ is another pair of nonnegative
integers such that $e+f=e'+f'$, and let $\varphi_1\in \Gamma^{e'}\Hom(M,N)$, $\varphi_2\in \Gamma^{f'}\Hom(M,N)$.  
Then, we set
$$(S\boxtimes T)(\varphi_1\otimes \varphi_2 ) = 
\begin{cases} S(\varphi_1)\otimes T(\varphi_2)& \text{ if } e'=e, f'=f \\
0 & \text{ otherwise}.
\end{cases}$$
Clearly, we have $S\boxtimes T \in \mathsf{biPol}^\I_{e+f}$.

Notice also that there is a functor $\mathsf\Gamma^d_{\M\times \M} \to \mathsf\Gamma^d_\M$ given by sending
$(M,N) \mapsto M\oplus N$.  The action on morphisms is given by the embedding
\[
\hom_{\mathsf\Gamma^d_{\M\times \M}}\left((M,N),(M',N')\right) \hookrightarrow\ 
\hom_{\mathsf\Gamma^d_\M}(M\oplus N,M'\oplus N'),
\]
which comes from the exponential property (\ref{eq:exponential}) for $\Gamma^d(\ )$.  This 
induces a functor $\ ^\mathsf{bi}: \Pol^\I_d \to \mathsf{biPol}^\I_d$, where 
$T^{\mathsf{bi}}(M,N)= T(M\oplus N)$.  
In a completely analogous way, we may also define a functor
$\ ^\mathsf{bi}: \Pol^\II_d \to \mathsf{biPol}^\II_d$.
%

\subsection{Schur bisuperfunctors}
In the remainder of this section, we will study decompositions of the strict polynomial bisuperfunctors 
$(\widehat{S}_{\lambda/\mu})^\mathsf{bi}$ and $(\widehat{S}^{\hp \II}_{\lambda/\mu})^\mathsf{bi}$.
\smallskip

We denote the lexicographic order on sequences in $\Z_+^\infty$ by $\preceq$. 
Thus, if $\alpha, \beta \in \Z_+^\infty$, then $\alpha \preceq \beta$ if $\alpha = \beta$
or if there exists some $1\leq i< \infty$ such that $\alpha_1=\beta_1, \dots,
\alpha_i= \beta_i$ and $\alpha_{i+1} <\beta_{i+1}$. 
This total order restricts to an order on partitions which is clearly consistent with the partial order of
inclusion; i.e. if $\mu\subset \lambda$, then $\mu\preceq \lambda$.
\medskip

If we consider $\Pi M$, $\Pi N$ in place of $M,N\in \svec$, then
recall from (\ref{eq:exponential}) that we have an embedding
\begin{equation}\label{eq:embedPi}
\Gamma_\Pi^{d_1}M \otimes \Gamma_\Pi^{d_2}N \hookrightarrow \Gamma_\Pi^{d_1+d_2}(M\oplus N)
\end{equation}
given by mapping $x\otimes y \mapsto x\coasterisk y$.
It follows for any partitions $\mu \subset \xi \subset \lambda$ that we have a
corresponding embedding 
\begin{equation}\label{eq:embedPi2}
\Gamma_\Pi^{\xi/\mu}(M) \otimes \Gamma_\Pi^{\lambda/\xi}(N) \hookrightarrow \Gamma_\Pi^{\lambda/\mu}(M\oplus N)
\end{equation}
given by tensoring embeddings of the form (\ref{eq:embedPi}).
Notice that this further gives an $\Svec$-enriched natural transformation:    
$\Gamma_\Pi^{\xi/\mu}\boxtimes \Gamma_\Pi^{\lambda/\xi} 
\to (\Gamma_\Pi^{\lambda/\mu})^\mathsf{bi}.$

\begin{definition}
Suppose $\mu \subset \xi \subset \lambda$ are partitions 
such that $|\lambda/\mu|=d$.
Let us define polynomial bisuperfunctors $K^\xi = K^\xi(\Gamma_\Pi^{\lambda/\mu})$
and $\dot{K}^\xi = \dot{K}^\xi(\Gamma_\Pi^{\lambda/\mu})$
as the images of the morphisms
$$ \bigoplus_{\substack{ \mu \subset \nu \subset \lambda,\\ \nu \preceq \xi}}
\Gamma_\Pi^{\nu/\mu}\boxtimes \Gamma_\Pi^{\lambda/\nu} \to (\Gamma_\Pi^{\lambda/\mu})^\mathsf{bi}, \qquad\quad
 \bigoplus_{\substack{ \mu \subset \nu \subset \lambda,\\ \nu \prec \xi}}
\Gamma_\Pi^{\nu/\mu}\boxtimes \Gamma_\Pi^{\lambda/\nu} \to (\Gamma_\Pi^{\lambda/\mu})^\mathsf{bi}
,$$
respectively.
Then  $K^\xi$ and $\dot{K}^\xi$ define objects in the category 
 $\mathsf{biPol}^\I_d$ of strict polynomial bisuperfunctors.
\end{definition}

Note for example that we have  
$ \dot{K}^\xi(M,N) = \sum_{\nu \prec \xi} K^{\nu}(M,N)$
for any $M,N \in \svec$.

\begin{definition}
With the same assumptions as in the previous definition, we define polynomial 
bisuperfunctors, $L_\xi=L_\xi((\widehat{S}_{\lambda/\mu})^\mathsf{bi})$ 
and $\dot{L}_\xi=\dot{L}_\xi((\widehat{S}_{\lambda/\mu})^\mathsf{bi})$, by setting
$$
L_\xi(M, N) = 
\hat{\theta}_{\lambda/\mu} \big( K^\xi (M, N) \big) \quad\text{and}\quad
\dot{L}_\xi(M, N) = 
\hat{\theta}_{\lambda/\mu} \big( \dot{K}^\xi(M, N) \big),
$$
for all $M,N \in \svec$.  
Then $L_\xi$ and $\dot{L}_\xi$ are also objects in $\mathsf{biPol}^\I_d$.
\end{definition}

Now let $M,N \in \svec$. Suppose $\sdim(M)= m|n$ and $\sdim(N)= m'|n'$.  Let us write
$r=m+n$, $s= m'+n'$.  Let $\Zc=\{Z_1, \dots, Z_r\}$ and $\Zc'=\{Z'_{1}, \dots, Z'_{s}\}$
be $\Z_2$-homogeneous bases of $M$ and $N$, respectively, which are ordered as
indicated by their indices.  If we write $W_j= Z_j$ for $j=1, \dots, r$ and $W_{r+k} = Z_k'$ 
for $k=1, \dots, s$, then 
$\W= \Zc \sqcup \Zc'= \{W_1, \dots, W_{r+s}\}$ is an ordered homogeneous  basis of $M\oplus N$.
Let us introduce an indexing set for $\W$.  Denote by 
$[r+s]= [r+s]_\W$ the   set  $\{1, \dots, r+s\}$, with $\Z_2$-grading 
$$[r+s]_\varepsilon = \{1\leq i \leq r+s: |W_i| = \varepsilon\}, \quad \text{for }\varepsilon \in \Z_2,$$
and with the usual ordering of integers. 
Let $\mathfrak{t}\in \text{Tab}_{[r+s]}(\lambda/\mu)$, and define the {\em $M$-part} 
of $\mathfrak t$ to be the  sequence 
$\kappa(\mathfrak{t})$  in $(\Z_{\geq 0})^\infty$ 
such that $\kappa(\mathfrak t)_i$= $\mu_i\ +$ the number of elements of $\{1, \dots, r\}$ 
in the $i$-th row of $\mathfrak{t}$.
If  $\mathfrak t$ is costandard, then $\kappa(\mathfrak t)$ is a partition such that 
$\mu \subset \kappa(\mathfrak t) \subset \lambda$.

%
\begin{lemma}[\cite{ABW}, Lemma II.4.4]\label{lem:abw} 
Let $\mathfrak{s},\mathfrak{t}\in \text{Tab}_{[r+s]}(\lambda/\mu)$ with $
\mathfrak{s}\trianglelefteq \mathfrak{t}$. 
Then $\kappa(\mathfrak{s}) \succeq \kappa(\mathfrak{t})$.
\end{lemma}

We consider the subsets $[r] = \{1, \dots, r\}$ and $[s'] = \{r+1, \dots, r+s\}$ 
as ordered $\Z_2$-graded subsets of $[r+s]$.
Suppose $\xi$ is a partition such that $\mu \subset \xi \subset \lambda$.
If $\mathfrak s \in \text{Tab}_{[r]}(\xi/\mu)$ and 
$\mathfrak t\in \text{Tab}_{[s']}(\lambda/\xi)$, then define the {\em double tableau}, 
$\mathfrak s| \mathfrak t \in \text{Tab}_{[r+s]}(\lambda/\mu)$, 
by setting
$$ (\mathfrak s| \mathfrak t )(i,j) =
\begin{cases}
\mathfrak s(i,j)& \text{if }\mu_i< j \leq \xi_i\\
\mathfrak t(i,j)& \text{if } \xi_i< j \leq \lambda_i
\end{cases}$$
for all $(i,j) \in \triangle_{\lambda/\mu}$.
Clearly if $\mathfrak{s,t}$ are row (resp. column) costandard, then
so is $\mathfrak{s|t}$.  

We will use notation similar to the previous section for the elements 
of $\Gamma_\Pi^{\lambda/\mu}(M\oplus N)$.  I.e., if $\mathfrak t \in \text{Tab}_{[r+s]}(\lambda/\mu)$,
then we write 
$$W_\Pi^{(\mathfrak t^i)} = (\Pi W)^{(w(\mathfrak t^i))},\ \text{ for }
1\leq i \leq q, \quad 
\text{ and }\quad
W_\Pi^{(\mathfrak t)} = W_\Pi^{(\mathfrak t^1)} \otimes \cdots \otimes W_\Pi^{(\mathfrak t^q)}$$
where $q=l(\lambda)$.
  If $\mathfrak{s} \in \text{Tab}_{[r]}(\xi/\mu)$ and 
$\mathfrak t\in \text{Tab}_{[s']}(\lambda/\xi)$, then let us write 
$  W_\Pi^{(\mathfrak s)} \coasterisk W_\Pi^{(\mathfrak t)} = 
\pm W_\Pi^{(\mathfrak s^1)} \coasterisk W_\Pi^{(\mathfrak t^1)}  \otimes \cdots 
\otimes W_\Pi^{(\mathfrak s^q)}\coasterisk W_\Pi^{(\mathfrak t^q)}$, 
where the coefficient $\pm 1$ is determined as usual by  the rule of signs. 
Then notice that the embedding
(\ref{eq:embedPi2}) sends 
$$ W_\Pi^{(\mathfrak s)} \otimes W_\Pi^{(\mathfrak t)} \mapsto 
  W_\Pi^{(\mathfrak s)} \coasterisk W_\Pi^{(\mathfrak t)}  = \pm \hpp  W_\Pi^{(\mathfrak{s|t})},$$
since $\text{Im}(\mathfrak s) \cap \text{Im}(\mathfrak t) = \emptyset$. 
It follows from our definitions that $W_\Pi^{(\mathfrak{s|t})} \in K^\xi(M,N)$.  
Furthermore, 
the set of all $W_\Pi^{(\mathfrak{s|t})}$ corresponding to row costandard 
$\mathfrak{s} \in \text{Tab}_{[r]}(\nu/\mu)$, $\mathfrak t\in \text{Tab}_{[s']}(\lambda/\nu)$, 
such that $\nu \preceq \xi$ (resp.~$\nu \prec \xi$), 
gives a basis for $K^\xi(M,N)$ (resp.~$\dot{K}^\xi(M,N)$).
\smallskip

\begin{proposition}
The canonical morphism 
$ \Gamma^{\xi/\mu}_\Pi\boxtimes \Gamma^{\lambda/\xi}_\Pi 
\to 
K^\xi$ induces a morphism
$$
\widehat{S}_{\xi/\mu} \boxtimes \widehat{S}_{\lambda/\xi} \hpp\to\hpp 
L_\xi\hp/\hp \dot{L}_\xi 
$$
in the category $(\Pol^\I_d)_\ev$.
\end{proposition}

\begin{proof}
For each pair $M,N \in \svec$, we must show that there exists a map, $\psi=\psi_\xi(M,N)$, which makes the following diagram commute:
$$\xymatrix{
\Gamma_\Pi^{\xi/\mu}M \otimes \Gamma_\Pi^{\lambda/\xi}N\ 
\ar[d]_-{\hat{\theta}_{\xi/\mu}\boxtimes\hpp \hat{\theta}_{\lambda/\xi}} \ar[r]^-{-\coasterisk-}
&\ K^\xi(M,N) \ar[d]^-{\hat{\theta}_{\lambda/\mu}(\cdot\hpp \oplus\hpp \cdot)}\\
\widehat{S}_{\xi/\mu}M \otimes \widehat{S}_{\lambda/\xi}N\  \ar@{.>}[r]_-{\psi}  &\ 
L_\xi(M,N)/ \dot{L}_\xi(M,N)}
$$
We may use Theorem \ref{thm:standard} to identify $\widehat{S}_{\xi/\mu}M$ with the 
cokernel of the map $\diamondsuit_{\xi/\mu}(M)$,
and $\widehat{S}_{\lambda/\xi}N$ with the cokernel of $\diamondsuit_{\lambda/\xi}(N)$.   
In order to prove the existence of $\psi$ it thus suffices to  show that 
$$ \text{Im}(\diamondsuit_{\lambda/\xi}(M))\coasterisk \Gamma_\Pi^{\xi/\mu}(N)\ +\ 
\Gamma_\Pi^{\lambda/\xi}(M)\coasterisk \text{Im}(\diamondsuit_{\xi/\mu}(N))$$
is contained in 
$$J = \dot{K}^\xi(M, N) + \text{Im}(\diamondsuit_{\lambda/\mu}(M\oplus N)).$$

We first show that 
$\text{Im}(\diamondsuit_{\xi/\mu}(M))\coasterisk \Gamma_\Pi^{\lambda/\xi}(N)$ is contained in $J$.
Recall that the map 
$$\diamondsuit_{\xi/\mu}(M)= \sum_{i=1}^{q-1}\sum_{u,v} 
1 \otimes \cdots \otimes \diamondsuit_i(\xi/\mu, u, v;M) \otimes\cdots \otimes 1,$$
where $q=l(\xi)$ and $u,v\geq 0$ are such that $u+v < \xi_{i+1}-\mu_i$.  
Let us fix such a triple $i,u,v$.
Then choose any tableaux  $\mathfrak t \in \text{Tab}_{[s']}(\lambda/\xi)$, 
$$\mathfrak r \in \text{Tab}_{[r]}((\xi_1,\dots, \xi_{i-1})/(\mu_1, \dots, \mu_{i-1}))\ \text{ and }\ 
\mathfrak s \in \text{Tab}_{[r]}((\xi_{i+2},\dots, \xi_q)/(\mu_{i+2}, \dots, \mu_q)).$$
We also consider a basis element $W_\Pi^{(\bs a, \bs b, \bs c)}$ of $\Gamma_\Pi^{u, p_1+p_2-u-v, v}M$,
where $\bs a, \bs b, \bs c \in I(m|n)$ are such that 
$l(\bs a) = u$, $l(\bs b) = p_1+p_2-u-v$ and $l(\bs c) =v$, with $p_1 = \xi_i-\mu_i$ and $p_2= \xi_{i+1}-\mu_{i+1}$.
Now let us write $W =  W_\Pi^{(\mathfrak r)} \otimes W_\Pi^{(\bs a, \bs b, \bs c)} \otimes W_\Pi^{(\mathfrak s)}$.  
Then  we have
$$ 
\diamondsuit_{\xi/\mu}(W) = W_\Pi^{(\mathfrak r)} 
\otimes \diamondsuit_1(\xi/\mu,u,v;M)(W_\Pi^{(\bs a, \bs b, \bs c)}) 
\otimes W_\Pi^{(\mathfrak s)},$$
and it suffices to show that  $\diamondsuit_{\xi/\mu}(W)\coasterisk W_\Pi^{(\mathfrak t)}$ belongs to $J$.

Next, let us consider  the action of 
$\diamondsuit_{\lambda/\mu}(M\oplus N)$. 
Let us write $\mathfrak t_1 = (\mathfrak t^1, \dots, \mathfrak t^{i-1})$, 
$\bar{\mathfrak t} = (\mathfrak t^i, \mathfrak t^{i+1})$, and 
 $\mathfrak t_2 = (\mathfrak t^{i+2}, \dots, \mathfrak t^q)$. Also consider the element
\[
W' = 
W_\Pi^{(\bs a)} \otimes (W_\Pi^{(\bs b)}\coasterisk 
W_\Pi^{({\mathfrak t}^i)}) \otimes (W_\Pi^{(\bs c)} \coasterisk W_\Pi^{({\mathfrak t}^{i+1})}).
\]
Then we have 
\begin{align*}
\diamondsuit_{\lambda/\mu}(M\oplus N) &\left( ( W_\Pi^{(\mathfrak r)} \coasterisk W_\Pi^{(\mathfrak t_1)})  \otimes W' \otimes  (W_\Pi^{(\mathfrak s)} \coasterisk W_\Pi^{(\mathfrak t_2)})\right) = \\[.18cm]
& ( W_\Pi^{(\mathfrak r)} \coasterisk W_\Pi^{(\mathfrak t_1)})\hpp  \otimes\hpp
\diamondsuit_i\big(\lambda/\mu,\hpp u,\hpp \lambda_{i+1}-\xi_{i+1}+v;M\oplus N\big) (W')  
 \hpp\otimes\hpp  (W_\Pi^{(\mathfrak s)} \coasterisk W_\Pi^{(\mathfrak t_2)}),
\end{align*}
and 
\begin{align}
\diamondsuit_i\big(\lambda/\mu,\hpp u,\hpp \lambda_{i+1}-\xi_{i+1}+v;M\oplus N\big) (W')  &=\nonumber\\[.18cm]
\sum_{\tilde{\bs b}^1, \tilde{\bs b}^2}
\text{sgn}_{\Zc_\pi}(\tilde{\bs b}^1, \tilde{\bs b}^2)\hpp & ( W_\Pi^{(\bs a)}\coasterisk W_\Pi^{(\tilde{\bs b}^1)}) 
 \otimes (W_\Pi^{(\tilde{\bs b}^2)} \coasterisk W_\Pi^{(\bs c)} ), \label{eq:sum1}
\end{align}
where the sum is taken over all pairs of standardized $\tilde{\bs b}^1 \in I(m+m'|n+n', \lambda_i-\mu_i-u)$ and  
$\tilde{\bs b}^2 \in I(m+m'|n+n',\xi_{i+1}-\mu_{i+1}-v)$ such that   $\text{st}(\tilde{\bs b}^1\vee \tilde{\bs b}^2)= \bs b\vee \mathfrak t^i$.

On the other hand, we notice that
\begin{align*}
\pm\hpp\diamondsuit_{\xi/\mu}(W) \coasterisk W_\Pi^{(\mathfrak t)}&=\\[.18cm]
 (W_\Pi^{(\mathfrak r)} &\coasterisk W_\Pi^{(\mathfrak t_1)}) \ \otimes\ 
(\diamondsuit_i(\xi/\mu,u,v;M)(W_\Pi^{(\bs a, \bs b, \bs c)}) \coasterisk W_\Pi^{(\bar{\mathfrak t})})\ 
 \otimes\ (W_\Pi^{(\mathfrak s)} \coasterisk W_\Pi^{(\mathfrak t_2)}),
 \end{align*}
and
\begin{align}
\pm\hpp \diamondsuit_i(\xi/\mu,u,v;M)(W_\Pi^{(\bs a, \bs b, \bs c)}) \coasterisk W_\Pi^{(\bar{\mathfrak t})} &=\nonumber\\[.18cm]
 \sum_{\bs b^1, \bs b^2} \text{sgn}_{\Zc_\pi}(\bs b^1, \bs b^2)\hpp & ( W_\Pi^{(\bs a)}\coasterisk W_\Pi^{(\bs b^1)} \coasterisk W_\Pi^{(\mathfrak t^i)} )
 \otimes (W_\Pi^{(\bs b^2)} \coasterisk W_\Pi^{(\bs c)}\coasterisk W_\Pi^{(\mathfrak t^{i+1})} ) \label{eq:sum2}
\end{align}
where the sum is taken over all pairs of  standardized  $\bs b^1 \in I(m|n,\xi_i-\mu_i-u)$, $\bs b^2 \in I(m|n,\xi_{i+1}-\mu_{i+1}-v)$
such that $\text{st}(\bs b^1\vee \bs b^2)= \bs b$.
\smallskip

Comparing the above sums, we see that every summand in (\ref{eq:sum2}), corresponding to 
some pair ($\bs b^1$, $\bs b^2$), also appears with the same sign in (\ref{eq:sum1}) as the summand corresponding
to the pair ($\tilde{\bs b}^1$, $\tilde{\bs b}^2$), with 
$\tilde{\bs b}^1 = \bs b^1 \vee \mathfrak t^i$ and $\tilde{\bs b}^2 = \bs b^2$. Furthermore, 
any other summand corresponding to 
($\tilde{\bs b}^1$, $\tilde{\bs b}^2$)
which is not of the above form  involves moving elements of $\mathfrak t^i$ from the $i$-th row
of $\lambda/\mu$ to the $(i+1)$-st row. Such a summand must  belong to $\dot{K}^\xi(M,N)$.
This completes the  proof that 
$\text{Im}(\diamondsuit_{\lambda/\xi}(M))\coasterisk \Gamma_\Pi^{\xi/\mu}(N)$ belongs 
to $J$. 
The proof that 
$\Gamma_\Pi^{\lambda/\xi}(M)\coasterisk \text{Im}(\diamondsuit_{\xi/\mu}(N)) \subset J$ 
is entirely symmetric.
\end{proof}

\begin{theorem}\label{thm:filtration}
For all pairs $(M,N)\in \mathsf\Gamma^d_{\M\times\M}$, the canonical map 
$$\psi_\xi :  \widehat{S}_{\xi/\mu}(M)\otimes \widehat{S}_{\lambda/\xi}(N)
\longrightarrow
L_\xi(M, N)\hp /\hp  
\dot{L}_\xi\big(M, N)$$ 
 is an isomorphism.  Hence, the bisuperfunctors 
$\{L_\xi:\ \mu\subset \xi\subset \lambda\}$
give a filtration of $(\widehat{S}_{\lambda/\mu})^\mathsf{bi}$ whose associated graded object is isomorphic
to 
$$\bigoplus_{\mu\subset \xi\subset \lambda} \widehat{S}_{\xi/\mu}\boxtimes 
\widehat{S}_{\lambda/\xi}.$$
\end{theorem}

\begin{proof}
First we notice that any
costandard $\mathfrak t \in \text{Tab}_{[r+s]}$ can be decomposed into a 
double tableau $\mathfrak{t} = \mathfrak t(1)|\mathfrak t(2)$, where 
  $\mathfrak t(1)$ is of shape $\kappa(\mathfrak t)/\mu$ with elements belonging to $[r]$
  and $\mathfrak t(2)$ is the 
tableau of shape $\lambda/\kappa(\mathfrak t)$ obtained by restricting to the entries from the set 
$[s']$.
It then
follows from Theorem \ref{thm:standard} that 
the set
$$\{\hat{\theta}_{\lambda/\mu}(W_\Pi^{(\mathfrak t)}) : \mathfrak t\in \text{Tab}_{[r+s]}(\lambda/\mu)
\text{ is costandard, and } \kappa(\mathfrak t) \preceq \xi \text{ (resp.\ } \kappa(\mathfrak t)\prec \xi)\}$$ 
gives a basis of $L_\xi(M,N)$ (resp.~$\dot{L}_\xi(M,N)$).
The theorem then follows from the fact that 
$W_\Pi^{(\mathfrak t)} = \pm\hpp  W_\Pi^{(\mathfrak t(1))} \coasterisk W_\Pi^{(\mathfrak t(2))}$.
\end{proof}

Notice that the functor $\text{Res}_{\C,d}: \Pol^\I_d \to \Pol^\II_d$ is an exact 
functor on the underlying even subcategories.
We thus have the following consequence of Theorem \ref{thm:filtration}.

\begin{corollary}\label{cor:filtration}
For each $\mu\subset \xi\subset \lambda$, let 
us define the bisuperfunctor $L^\II_\xi = \text{Res}_{\C,d}(L_\xi)$ in 
the category $\Pol^\II_d$.  Then  the set of all such 
$L^\II_\xi$ 
gives a filtration of $(\widehat{S}^\II_{\lambda/\mu})^\mathsf{bi}$ whose 
associated graded object is isomorphic to 
$$
\bigoplus_{\mu\subset \xi\subset \lambda} \widehat{S}^{\hp \II}_{\xi/\mu}\boxtimes 
\widehat{S}^{\hp \II}_{\lambda/\xi}.
$$
\end{corollary}

\section{Highest Weight Theory}\label{Sec:highest}

The main purpose of this section is to show that the Schur
superfunctors $\widehat{S}^{\I}_\lambda= \widehat{S}_\lambda$ 
are indecomposable objects of the category $\Pol^\I_d$. 
It will be convenient to work with the {Schur superalgebra} 
$S(m|n,d)$ for $m,n,d \in \Z_{\geq 0}$.
%
%
We also describe formal characters of 
$\widehat{S}^{\I}_\lambda$ and $\widehat{S}^{\II}_\lambda$.
For this, we will need to work with both Schur superalgebras 
$S(m|n,d)$ and $Q(n,d)$, respectively.

\subsection{The Schur superalgebra $S(m|n,d)$}

Suppose $m,n,d$ are nonnegative integers, and let $M=\k^{m|n}$.  
 We choose bases $\Xc= (X_1, \dots, X_m)$ of $M_0$
and $\Yc= (Y_1, \dots, Y_n)$ of $M_1$.  Let $\Zc = \Xc\sqcup \Yc$ be the 
$\Z_2$-graded set with $\Zc_0 =\Xc$, $\Zc_1=\Yc$ and with total order
$$X_1< \cdots < X_m < Y_1< \cdots< Y_n.$$ 
In this case, we we write 
$\text{Tab}_{m|n}(\lambda/\mu) = \text{Tab}_\Zc(\lambda/\mu)$ if 
$\mu \subset \lambda$ are any partitions. 
We also write
$Z_i = X_i$ and $Z_{m+j} = Y_j$ for $1\leq i\leq m$ and $1\leq j\leq n$.  

Let $E_{i,j} \in \End(M)$ denote the element such that $E_{i,j} Z_k = \delta_{j,k} Z_i$
for $k = 1,\dots, m+n$.  Then  $\{E_{i,j}:\ 1\leq i,j \leq m+n\}$ gives a basis  of $\End(M)$, which
we order lexicographically.
We define the the {\em Schur superalgebra} $S(m|n,d)$ to be the associative
superalgebra
$$ S(m|n,d) = \Gamma^d\End(M) \cong \End_{\k \Si_d}(M^{\otimes d}).$$
Denote the multiplication of $x,y\in S(m|n,d)$ by $x\circ y$.  
\smallskip

Recall that the symmetric group $\Si_d$ acts on  $I(m|n,d)$ by
composition.   Given $\bs i, \bs j\in I(m|n,d)$, let us write 
$\bs i \sim \bs j$ if $\bs j = \bs i.\sigma$, for some $\sigma \in \Si_d$.
We also let $\Si_d$ act on the set 
$I(m|n,d)\times I(m|n,d)$ via the diagonal action.
We again write $\sim$ denote the equivalence relation corresponding to the orbits of
this action. E.g., if $\bs i \sim \bs k$ and 
$\bs j \sim \bs l$, then $(\bs i, \bs j) \sim (\bs k, \bs l)$, but not vice versa. 
Let $\Omega(m|n,d)$  denote a set of $\Si_n$-orbit representatives in $I(m|n,d)\times I(m|n,d)$.
\smallskip

Following \cite{BrKu}, we say that  
$(\bs i, \bs j) \in I(m|n,d) \times I(m|n,d)$ 
is {\em strict} if: $(i_k,j_k) \neq (i_l,j_l)$ whenever both $k\neq l$ and $|i_k| + |j_k| = |i_l|+|j_l|=1$.
Let $I^2(m|n,d)$ denote the set of all such strict pairs. 
Then the diagonal action of $\Si_d$ restricts to an action on $I^2(m|n,d)$.
Notice that $|\bs i|_\Zc + |\bs j|_\Zc = |\bs i|_{\Zc_\pi} + |\bs j|_{\Zc_\pi}$,
for any $\bs i, \bs j \in I(m|n,d)$.  
\smallskip

We denote by $E^{(\bs i,\bs j)}\in S(m|n,d)=\Gamma^d \End(M)$ the element which is dual
to the monomial
$$\prod_{k,l=1}^{m+n} \check{E}_{i_k,j_l} \in S^d(\End(M)^*),$$
where $\{\check{E}_{i,j}\}$ denotes the dual basis of $\{E_{i,j}\}$.
%
Then the set
$$ \{E^{(\bs i,\bs j)} : (\bs i,\bs j) \in \Omega(m|n,d)\cap I^2(m|n,d)\}$$
gives a basis of $S(m|n,d)$.
\smallskip

Given any $\varepsilon, \delta \in (\Z_2)^d$, we define the {\em charge} :
$$\text{chr}(\varepsilon, \delta) = \prod_{1\leq s< t\leq d} \hpp (-1)^{\varepsilon_s \delta_t}.$$
Now 
suppose  given $(\bs i,\bs j), (\bs k,\bs l)\in I^2(m|n,d)$.  Then we define: 
$$ 
\text{chr}(\bs i, \bs j; \bs k) =  \text{chr}(|\bs i|_\Zc + |\bs j|_\Zc, 
|\bs k|_\Zc),\qquad 
\text{chr}_\pi(\bs i , \bs j; \bs k) =  \text{chr}(|\bs i|_\Zc + |\bs j|_\Zc, 
|\bs k|_{\Zc_\pi})$$
and
$$
\text{chr}(\bs i, \bs j; \bs k, \bs l) =  \text{chr}(|\bs i|_\Zc + |\bs j|_\Zc, 
|\bs k|_\Zc + |\bs l |_\Zc).$$
If $(\bs i,\bs j) \sim (\bs k,\bs l)$, then  we also define 
$$\text{sgn}(\bs i,\bs j; \bs k,\bs l) = \text{sgn}(|\bs i|_\Zc + |\bs j|_\Zc, \sigma)$$
where $\sigma \in \Si_d$ is any permutation such that 
$(\bs i, \bs j) = (\bs k.\sigma, \bs l.\sigma)$.
We may also consider the elements 
$E^{\bs i,\bs j}  = E^{i_1, j_1}\otimes E^{i_2, j_2} \otimes \cdots \otimes E^{i_d,j_d}
 \in \End(M)^{\otimes d} = \End(M^{\otimes d})$, where the action is given by the rule
 of signs.  Notice that we have
\begin{equation}\label{eq:tensor3}
E^{\bs i,\bs j} Z^{\bs k} = \text{chr}(\bs i,\bs j; \bs k ) \delta_{\bs j,\bs k} Z^{\bs i}
\quad \text{ and } \quad
E^{\bs i,\bs j} Z_\Pi^{\bs k} = \text{chr}_\pi(\bs i,\bs j; \bs k) \delta_{\bs j,\bs k} Z_\Pi^{\bs i}.
\end{equation}
Recall from Lemma \ref{lem:Ax} that there is an isomorphism 
$$\Gamma^d\End(M)=S(m|n,d) \stackrel{\sim}{\to} (\End(M)^{\otimes d})^{\Si_d}= \End_{\k\Si_d}(M^{\otimes d}),$$
which shows that the corresponding embedding $\Delta: S(m|n,d) \hookrightarrow \End(M^{\otimes d})$
is faithful.  
This  embedding is given explicitly by the equation
\begin{equation}\label{eq:tensor4}
\Delta(E^{(\bs i,\bs j)}) 
= \sum_{(\bs s,\bs t)\sim (\bs i,\bs j)} \text{sgn}(\bs i,\bs j; \bs s,\bs t) E^{\bs s,\bs t},
\end{equation}
which holds for any  $(\bs i, \bs j) \in I^2(m|n,d)$.

Now there is a canonical (even) isomorphism of superalgebras $\End(M) \cong \End(\Pi M)$.  It follows
that there is also a faithful embedding 
$$ S(m|n,d) \hookrightarrow \End(\Pi M)^{\otimes d} = \End((\Pi M)^{\otimes d}).$$
The superspaces $M^{\otimes d}$ and $(\Pi M)^{\otimes d}$ are thus naturally $S(m|n,d)$-supermodules.
The following lemma describes these actions (cf. \cite[Lem.~5.1]{BrKu}).

\begin{lemma}\label{lem:multiply2}
\begin{itemize}  
\item[(i)]
Suppose given $(\bs i,\bs j) \in I^2(m|n,d)$ and $\bs l \in I(m|n,d)$.  Then
the action of $S(m|n,d)$ on $M^{\otimes d}$ satisfies
the equation
$$E^{(\bs i,\bs j)}Z^{\bs l} = 
\sum_{(\bs k, \bs l) \sim (\bs i, \bs j)} \text{sgn}(\bs i,\bs j; \bs k,\bs l)\text{chr}(\bs i,\bs j; \bs k) Z^{\bs k},
$$
and the action on $(\Pi M)^{\otimes d}$ is similarly given by
$$
E^{(\bs i,\bs j)}Z_\Pi^{\bs l} = 
\sum_{(\bs k, \bs l) \sim (\bs i, \bs j)} \text{sgn}(\bs i,\bs j; \bs k,\bs l)\text{chr}_\pi(\bs i,\bs j; \bs k) Z_\Pi^{\bs k}.
$$
\item[(ii)]
Furthermore, the multiplication in $S(m|n,d)$ satisfies
$$
E^{(\bs i,\bs j)}\circ E^{(\bs k,\bs l)} = 
\sum_{(\bs s,\bs t)\in \Omega(m|n,d)} C_{\bs i,\bs j,\bs k,\bs l,\bs s,\bs t} E^{(\bs s,\bs t)},
$$
with the coefficients
$$C_{\bs i,\bs j,\bs k,\bs l,\bs s,\bs t} = \sum \text{sgn}(\bs i,\bs j;\bs s,\bs h) 
  \text{sgn}(\bs k,\bs l;\bs h,\bs t) %
\hpp \text{chr}(\bs h,\bs t; \bs s,\bs h),$$
the sum being over all $\bs h\in I(m|n,d)$ such that 
$(\bs s,\bs h) \sim (\bs i,\bs j)$ and $(\bs h,\bs t) \sim (\bs k,\bs l)$.
\end{itemize}
\end{lemma}

\begin{proof}
The first part of (i) follows from (\ref{eq:tensor3}) and (\ref{eq:tensor4}).  The
second equation is also not difficult to see since 
$|\bs i|_\Zc  + |\bs j|_\Zc = |\bs i|_{\Zc_\pi} + |\bs j|_{\Zc_\pi}$.  

The formula in (ii) can be verified using the first part of (i) together with the
fact that the embedding $S(m|n,d) \hookrightarrow \End(M^{\otimes d})$ is faithful.
\end{proof}

Let us write $\Lambda(m|n) = (\Z_{\geq 0})^{m+n}$ and 
$\Lambda(m|n,d) = \{\lambda\in \Lambda(m|n):\ |\lambda|= \sum \lambda_i=d\}$. 
We let $\Lambda^+(m|n,d)$ denote the subset consisting of all partitions $\lambda \in \Lambda(m|n,d)$.  
For any $\mu\in \Lambda(m|n,d)$, we define the element $E^{(\mu)} = 
E^{(\bs i^{\mu}, \bs i^{\mu})} \in S(m|n,d)$. 
The elements $E^{(\mu)}$ ($\mu\in  \Lambda(m|n,d)$)  are called  {\em weight idempotents}.  
Given $\bs i\in I(m|n,d)$, recall the definition of $wt(\bs i) \in\Lambda(m|n,d)$ 
from Section \ref{ss:divided}.
The following is a restatement of \cite[Lemma 5.3]{BrKu} in terms of our notation.

\begin{lemma}\label{lem:BrKu}
Suppose $(\bs i,\bs j)\in {I}^2(m|n,d)$ and $\mu \in  \Lambda(m|n,d)$.  Then 
$$
{E}^{(\mu)}\hm\circ E^{(\bs i,\bs j)} = \begin{cases}
E^{(\bs i,\bs j)}& \text{if } wt(\bs i) = \mu, \\
 0& \text{otherwise}, 
 \end{cases} 
\qquad
E^{(\bs i,\bs j)}\hm \circ E^{{(\mu)}} =\begin{cases}
E^{(\bs i,\bs j)} &\text{if } wt(\bs j) = \mu,\\
0& \text{otherwise}.
\end{cases}
$$
In particular, $\{E^{(\mu)}:\ \mu\in  \Lambda(m|n,d)\}$ gives a  set of mutually
orthogonal even idempotents whose sum, $e= \sum E^{(\mu)}$, 
is the identity of $S(m|n,d)$.
\end{lemma}

\begin{proof}
This follows from from part (ii) of  Lemma \ref{lem:multiply2}.
\end{proof}

Given a  finite dimensional $S(m|n,d)$-supermodule, $V$, 
and $\mu \in \Lambda(m|n,d)$,
we define the weight space $V_\mu= E^{(\mu)}V$.  
From the above lemma, we  have  
a weight space decomposition
\begin{equation}\label{eq:decomp1}
V = \bigoplus_{\mu\hp \in\hp \Lambda(m|n,d)} V_\mu.
\end{equation}
We call a homogeneous
vector $v\in V_{\mu}$ a {\em weight vector} of weight 
$\mu$, and we write $wt(v)=\mu$.
%

\subsection{Schur supermodules and Schur superfunctors}

Recall from \cite[Sec.~4]{Axtell}  that there is a functor $\Pol^\I_d \to S(m|n,d)\hpp \smod$, 
given by sending $T\mapsto T(\k^{m|n})$.  In particular $I^{\otimes d}\mapsto M^{\otimes d}$
and $\Pi^{\otimes d} \mapsto (\Pi M)^{\otimes d}$, where we continue to fix $M=\k^{m|n}$. 
We also have $\Gamma_\Pi^{\lambda} \mapsto \Gamma_\Pi^{\lambda}(M)$.  Since there
is an embedding $\Delta: \Gamma_\Pi^{\lambda} \hookrightarrow \Pi^{\otimes d}$,
we also have an embedding 
\begin{equation}\label{eq:embedPi2}
\Delta(M): \Gamma_\Pi^{\lambda}(M) \hookrightarrow (\Pi M)^{\otimes d}
\end{equation}
of $S(m|n,d)$-supermodules.  We wish to describe the action of $S(m|n,d)$ on $\Gamma_\Pi^{\lambda}M$
explicitly, by restricting the action on $(\Pi M)^{\otimes d}$ in Lemma \ref{lem:multiply2} considered above.
\smallskip

We next define an action of the symmetric groups $\Si_d$ on $\text{Tab}_{m|n}(\lambda)$.
The action is given by restricting the action of $\Si_d$ on $I(m|n,d)$ along the bijection, 
 $w(\ ): \text{Tab}_{m|n}(\lambda) \to I(m|n,d)$, 
sending a tableau to its reading word. 
I.e., if $\sigma \in \Si_d$ and $\mathfrak t \in \text{Tab}_{m|n}(\lambda)$, then 
 $\mathfrak t.\sigma$ is the tableau of 
 shape $\lambda$ such that $w(\mathfrak t.\sigma) = w(\mathfrak t).\sigma$.

Now consider the subgroup 
$\Si_{\lambda}^\leftrightarrow = \Si_{\lambda_1} \times \cdots \times \Si_{\lambda_q} 
\subset \Si_d$.  This subgroup acts on $\mathfrak t \in \text{Tab}_{m|n}(\lambda)$ by 
permuting the entries of each row.  
We say that $\mathfrak{s,t} \in \text{Tab}_{m|n}(\lambda)$ are {\em row equivalent} 
and write $\mathfrak s \approx_\text{R} \mathfrak t$, 
if there exists $\sigma \in \Si_{\lambda}^\leftrightarrow$ such that $\mathfrak s = \mathfrak t.\sigma$. 
\smallskip

We also define the subgroup  permuting the elements of each column in a tableaux.  
For this it is convenient to 
consider the tableaux $\mathfrak{col}_\lambda \in \text{Tab}_{m|n}(\lambda)$ given by 
$\mathfrak{col}_\lambda(i,j)=j$ for all $(i,j)\in \triangle_\lambda$.  
We then let 
$\Si_\lambda^\updownarrow = \{ \sigma \in \Si_d : w(\mathfrak{col}_\lambda).\sigma 
= w(\mathfrak{col}_\lambda)\}$.
Now if there exists a permutation $\sigma \in \Si_{\lambda}^\updownarrow$ 
such that $\mathfrak s = \mathfrak t.\sigma$,
then we say that $\mathfrak{s,t} \in \text{Tab}_{m|n}(\lambda)$ are {\em column equivalent} 
and write $\mathfrak s \approx_\text{C} \mathfrak t$.
\smallskip

Define a relation: $\mathfrak s \stackrel{\text{RC}}{\longrightarrow} \mathfrak t$, 
if there exists a tableau $\mathfrak r \in \text{Tab}_{m|n}(\lambda)$ such that 
$\mathfrak r \approx_\text{R} \mathfrak s$ and $\mathfrak r \approx_\text{C} \mathfrak t$.

\begin{lemma}\label{lem:precede}
Suppose $\mathfrak{s,t} \in \text{Tab}_{m|n}(\lambda)$ are both costandard.  
If $\mathfrak s \stackrel{\text{RC}}{\longrightarrow} \mathfrak t$, then 
$w(\mathfrak s) \succeq w(\mathfrak t)$.
\end{lemma}

\begin{proof}
We may assume that $\mathfrak s \neq \mathfrak t$. Then 
either $\mathfrak r\neq\mathfrak s$ or $\mathfrak r \neq \mathfrak t$. 
Notice that each equivalence class of $\approx_\text{C}$ (resp.~of $\approx_\text{R}$) 
contains at most one tableau which is costandard.
Hence, we have both $\mathfrak r\neq\mathfrak s$ and 
$\mathfrak r \neq \mathfrak t$. 
Let $i_0$ be the first row of $\lambda$ 
such that $\mathfrak r^{i_0} \neq \mathfrak t^{i_0}$.  
We must then have
$\mathfrak r(i_0,j) \geq \mathfrak t(i_0,j)$ for $1\leq j \leq \lambda_{i_0}$,
since $\mathfrak t(i_0,j) \leq \mathfrak t(i,j)$, for all $i>i_0$ such that 
$(i,j)\in \triangle_\lambda$. 

Notice that each row $\mathfrak s^i$ of $\mathfrak s$ is obtained from the $i$-th row
of $\mathfrak r$ by rearranging the entries of $\mathfrak r^i$ into nonincreasing order.  
Since $\mathfrak t$ is costandard and $\mathfrak r^i = \mathfrak t^i$ for $1\leq i< i_0$,
it follows that $\mathfrak s^i= \mathfrak t^i$ for $1\leq i<i_0$ as well.
Now since $\mathfrak r(i_0,j)\geq \mathfrak t(i_0,j)$, for $1\leq j\leq \lambda_{i_0}$, 
and $\mathfrak s^i$ is a non-increasing rearrangement of $\mathfrak r^i$, we must 
also have $\mathfrak s(i_0,j) \geq \mathfrak t(i_0,j)$, for $1\leq j\leq \lambda_{i_0}$. 
Finally, there must be at least one $j$ such that $\mathfrak s(i_0,j)> \mathfrak t(i_0,j)$
since $\mathfrak r^{i_0} \neq \mathfrak t^{i_0}$.
It follows that $w(\mathfrak s) \succ w(\mathfrak t)$.
\end{proof}

In the following, we will write $\text{chr}_\pi(\mathfrak{s, t};\mathfrak{t})$ instead 
of $\text{chr}_\pi(w(\mathfrak s),w(\mathfrak t);w(\mathfrak t))$, and similar notation
holds for $\text{sgn}$.

\begin{lemma}\label{lem:multiply3}
Suppose $\lambda$ is a partition. 
The  action of $S(m|n,d)$ on the supermodule $\Gamma^\lambda_\Pi M$
is given as follows.  
Suppose $\mathfrak t\in \text{Tab}_{m|n}(\lambda)$ is  row costandard.
Then
$$ E^{(\bs i,\bs j)} Z_\Pi^{(\mathfrak t)} =  
\sum_{w(\mathfrak s)\sim \bs i} C_{\bs i,\bs j; \mathfrak{s,t}} Z_\Pi^{(\mathfrak s)},$$
where the sum is over row costandard $\mathfrak s \in \text{Tab}_{m|n}(\lambda)$,
with the coefficients
$$
C_{\bs i,\bs j; \mathfrak{s,t}}= \sum  \text{sgn}(\bs i,\bs j; \mathfrak{s}, \mathfrak{r})\hpp 
\text{chr}_\pi(\mathfrak{s}, \mathfrak{r}; \mathfrak{s}),
$$
summing over all $\mathfrak r \in \text{Tab}_{m|n}(\lambda)$  such that 
$\mathfrak r \approx_\text{R} \mathfrak t$ and $(w(\mathfrak s), w(\mathfrak r)) \sim (\bs i,\bs j)$.
\end{lemma}

\begin{proof}
Notice that the embedding (\ref{eq:embedPi2}) can be given explicitly by the formula
$$ \Delta( Z_\Pi^{(\mathfrak t)})= 
 \sum_{\mathfrak s \approx_\text{R} t} \text{sgn}_{\Zc_\pi}(\mathfrak t; \mathfrak s)\hpp Z_\Pi^{\hp \mathfrak s}. 
$$
The lemma may then be verified by restricting the second formula in part (i) of Lemma \ref{lem:multiply2}
to the image of $\Gamma_\Pi^\lambda M$ under this embedding.
\end{proof}

Let us order the elements of $\bs i \in I(m|n,d)$ using the lexicographic order $\preceq$ 
on the sequences $(i_1, \dots, i_d)$. 
I.e., we have $\bs i \preceq \bs j$ if $\bs i = \bs j$ or if there exists $1\leq k< d$
such that $i_1=j_1, \dots, i_k=j_k$, and $i_{k+1} < j_{k+1}$.

Next we define the subsuperspace $N(m|n,d)\subset S(m|n,d)$ to be the 
linear span of all $E^{(\bs i,\bs j)}$ such
that $(\bs i, \bs j) \in I^2(m|n,d)$ and $\bs i \prec \bs j$.  Then it is clear that $N(m|n,d)$ is a (nonunital) subalgebra. We say that 
an element $v\in V$, of a finite dimensional $S(m|n,d)$-supermodule $V$, is {\em $N(m|n,d)$-invariant}
if $x.v = 0$ for all $x\in N(m|n,d)$.  An $N(m|n,d)$-invariant element which is also a nonzero weight vector is 
called a {\em highest weight vector} of $V$.
%

\begin{lemma}\label{lem:invariant}
Every nonzero supermodule $V\in S(m|n,d)\hpp \smod$ contains a
highest weight vector.
\end{lemma}

\begin{proof}
Suppose $v\in V_\mu$ is a weight vector, for some $\mu \in \Lambda(m|n,d)$.  
If $(\bs i, \bs j) \in I^2(m|n,d)$, then 
it is not difficult to see using Lemma \ref{lem:BrKu} that $E^{(\bs i, \bs j)} v$ is 
also a weight vector of weight $wt(\bs i)$, and  $E^{(\bs i, \bs j)} v = 0$
unless $wt(\bs j)= \mu$. 
It follows that $wt(E^{(\bs i, \bs j)} v) = wt(\bs i) \prec \mu$, whenever $\bs i \prec \bs j$. 
Hence, if we choose $\mu$ to be the earliest weight in the lexicographic order such 
that $V_\mu \neq 0$, then any nonzero $v\in V_\mu$ must be a highest
weight vector.
\end{proof}

Recall from \cite[Sec.~4]{Axtell} that there is an evaluation functor 
$\ev^{\hp\I}_{m|n}: \Pol^{\hp \I}_d \to S(m|n,d)\hpp \smod$, given 
by sending $T \mapsto T(\k^{m|n})$, which is an equivalence if $m,n\geq d$.

\begin{definition}
It follows from the preceding paragraph that 
the natural transformation 
$\hat{\theta}_\lambda: \Gamma^{\lambda'}_\Pi \to S^{\lambda}$ yields by
evaluation an even $S(m|n,d)$-supermodule homomorphism
 $$\hat{\theta}_\lambda(M): \Gamma^{\lambda'}_\Pi M \to S^{\lambda} M.$$ 
Hence, the image of this homomorphism, $\widehat{S}_\lambda(M)$, 
is naturally an $S(m|n,d)$-supermodule which we call the 
{\em Schur supermodule} of weight $\lambda$.
\end{definition}

Consider the canonical tableaux $\mathfrak c_\lambda$ which is the costandard
tableaux in $\text{Tab}_{m|n}(\lambda)$ defined by setting, for $(i,j) \in \triangle_\lambda$,
$$
\mathfrak c(i,j) = 
\begin{cases} j, & \text{if }1\leq j \leq m\\
m+i, & \text{if }j>m.
\end{cases}
$$
Then $\mathfrak c_\lambda \preceq \mathfrak t$, for any costandard
 $\mathfrak t \in \text{Tab}_{m|n}(\lambda)$.

\begin{example}\label{ex:Schur}
It follows from Lemma \ref{lem:invariant} and 
the Standard Basis Theorem \ref{thm:standard} that 
the element $\hat{\theta}_{\lambda'}(Z_\Pi^{(\mathfrak c_\lambda)})$ 
corresponding to the
canonical tableau $\mathfrak c_\lambda$ is a highest 
weight vector of the Schur supermodule $\widehat{S}_{\lambda'}(M)$.
\end{example}

From now on let us identify the Schur supermodule $\widehat{S}_{\lambda'}M$
as the quotient $\Gamma_\Pi^\lambda M/ \diamondsuit_\lambda(M)$.  For
convenience, we also write $x = x+ \diamondsuit_\lambda(M)$ to denote the
coset of any $x\in \Gamma_\Pi^\lambda M$.
The following may be viewed as a super analogue of the result \cite[Theorem 3.3]{DEP}
of De Concini, Eisenbud, and Procesi. 

\begin{theorem}\label{thm:indecomp}
Let $\lambda \in \Lambda^+(m|n,d)$ be a partition such that 
 $l(\lambda')\leq m$.
Then: 
\begin{itemize}
\item[(i)]  Every nonzero $N(m|n,d)$-submodule of $\widehat{S}_{\lambda'}(M)$
contains $Z_\Pi^{(\mathfrak c_\lambda)}$; \smallskip
\item[(ii)] the $N(m|n,d)$-invariants of $\widehat{S}_{\lambda'}(M)$ are
spanned by $Z_\Pi^{(\mathfrak c_\lambda)}$; \smallskip
\item[(iii)]  $\widehat{S}_{\lambda'}(M)$ is $N(m|n,d)$-indecomposable.
\end{itemize}
\end{theorem}

\begin{proof}
Proof of (i).  Suppose  $V\subset \widehat{S}_{\lambda'} M$ is a nonzero
$N(m|n,d)$-submodule, and let $v\in V$. Of course, we may assume that $v$ 
is not contained in the $\k$-span of $Z_\Pi^{(\mathfrak c_\lambda)}$.
Using Theorem \ref{thm:standard}, write $v$ as a linear combination 
$$ v =\sum_{\mathfrak t} a_{\mathfrak t}\hpp Z_\Pi^{(\mathfrak t)},$$
summing over all costandard $\mathfrak t \in \text{Tab}_{m|n}(\lambda)$, 
for some coefficients $a_{\mathfrak t} \in \k$.   

Consider any costandard $\mathfrak t_0$ such that $\mathfrak t_0 \neq \mathfrak c_\lambda$. 
Then $\mathfrak c_\lambda \prec \mathfrak t_0$, so that 
$E^{(\mathfrak c_\lambda, \mathfrak t_0)} \in N(m|n,d)$. 
Now since $l(\lambda') \leq m$, we have $\mathfrak c_\lambda(i,j) = j$ for all $(i,j) \in \triangle_\lambda$ 
(i.e. $\mathfrak c_\lambda= \mathfrak{col}_\lambda$). 
Hence, if $\mathfrak s$ is row costandard and $w(\mathfrak s) \sim w(\mathfrak c_\lambda)$, then we
must have $\mathfrak s = \mathfrak c_\lambda$.  
From Lemma \ref{lem:multiply3}, we thus have
$$ E^{(w(\mathfrak c_\lambda),w(\mathfrak t_0))} v\  =\ 
\sum_{\mathfrak t} \sum_{w(\mathfrak s)\sim w(\mathfrak c_\lambda)} a_{\mathfrak t}\hpp
C_{w(\mathfrak c_\lambda), w(\mathfrak t_0); \mathfrak s, \mathfrak t} Z_\Pi^{(\mathfrak s)}\ 
=\  
\sum_{\mathfrak t}  a_{\mathfrak t}\hpp
C_{w(\mathfrak c_\lambda), w(\mathfrak t_0); \mathfrak c_\lambda, \mathfrak t} Z_\Pi^{(\mathfrak c_\lambda)}.
$$
Now if $\mathfrak r \in \text{Tab}_{m|n}(\lambda)$, then notice that 
$(w(\mathfrak c_\lambda), w(\mathfrak r)) \sim (w(\mathfrak c_\lambda), w(\mathfrak t_0))$
if and only if $\mathfrak r \approx_\text{C} \mathfrak t_0$, where we again use the fact that
$l(\lambda')\leq m$.  Notice first that  $\mathfrak r \approx_\text{R} \mathfrak t_0$ and 
$\mathfrak r \approx_\text{C} \mathfrak t_0$ implies that $\mathfrak r = \mathfrak t_0$.  
Since $\text{sgn}(w(\mathfrak c_\lambda),w(\mathfrak t_0); \mathfrak c_\lambda, \mathfrak t_0)=1$,
we thus have
$$C_{w(\mathfrak c_\lambda), w(\mathfrak t_0); \mathfrak c_\lambda, \mathfrak t_0}
=   \text{chr}_\pi(\mathfrak c_\lambda, \mathfrak t_0; \mathfrak c_\lambda).$$
On the other hand, we must have 
$C_{w(\mathfrak c_\lambda), w(\mathfrak t_0); \mathfrak c_\lambda, \mathfrak t}
= 0$, unless 
$ \mathfrak t \stackrel{\text{RC}}{\longrightarrow} \mathfrak t_0$.  It thus follows from 
Lemma \ref{lem:precede} that 
$C_{w(\mathfrak c_\lambda), w(\mathfrak t_0); \mathfrak c_\lambda, \mathfrak t}= 0$
whenever $w(\mathfrak t) \prec w(\mathfrak t_0)$.
\smallskip

Let $\mathfrak t_1$ be the unique costandard tableau with the property: 
$a_{\mathfrak t_1}\neq 0$ and $w(\mathfrak t_0) \succeq w(\mathfrak s)$ 
for any costandard $\mathfrak s$ such that $a_{\mathfrak s}\neq 0$.  
It follows from the above arguments that
$$E^{(\mathfrak c_\lambda, \mathfrak t_1)} v = \pm a_{\mathfrak t_1} Z_\Pi^{(\mathfrak c_\lambda)}.$$
Since $a_{\mathfrak t_0}\neq 0$, it follows that $Z_\Pi^{(\mathfrak c_\lambda)}$ is in the 
$N(m|n,d)$-span of  $v$.
%

Proof of (ii).  The proof of (i) shows that $N(m|n,d)\hp v\neq 0$ for any $v\in \widehat{S}_{\lambda'}M$
which is not contained in the $\k$-span of $Z_\Pi^{(\mathfrak c_\lambda)}$.  So (ii) follows from Example
\ref{ex:Schur}.
%

Proof of (iii).  Suppose that we have a decomposition $\widehat{S}_\lambda M = V\oplus W$ into 
$N(m|n,d)$-submodules.  From part (i) it follows that $Z_\Pi^{(\mathfrak c_\lambda)}$ belongs to
both $V$ and $W$ simultaneously, which is impossible.
\end{proof}

\begin{proposition}\label{prop:indecomp}
Let $\lambda$ be any partition with $d=|\lambda|$.  
Then the Schur superfunctor 
$\widehat{S}_\lambda$ is an indecomposable object of the
category $\Pol^\I_d$.  I.e., if
$S,T\in \Pol^\I_d$ are such that $\widehat{S}_{\lambda}= S\oplus T$, 
then either $S=0$ or $S= \widehat{S}_\lambda$.
\end{proposition}

\begin{proof}
This follows from Theorem \ref{thm:indecomp}.(iii), since
by \cite[Thm.~4.2]{Axtell}, the evaluation functor is an equivalence of categories, 
$\Pol^\I_d \equi S(m|n,d)\hpp \smod$, for $m,n \gg 0$.
\end{proof}


\subsection{The Schur superalgebra $Q(n,d)$}
We next discuss  type $\II$ Schur superfunctors.
For this, we need to consider  the Schur superalgebra
$$Q(n,d) := \Gamma^d\hpp \End_{\C_1}(\U_r(1)^n).$$
Let us write $V_n = \U(1)^n$. 
Recall from \cite[Sec.~4]{Axtell} that there is an evaluation functor 
$$\ev^{\hp\II}_{n}\hp: \hpp \Pol^{\hp \II}_d \to Q(n,d)\hpp \smod, 
\qquad T \mapsto T(V_{n}),$$ 
which is an equivalence if $n\geq d$.
It follows that $\widehat{S}^\II_\lambda(V_n)$ has the structure of 
left $Q(n,d)$-supermodule. 
We note that 
this supermodule is usually not indecomposable.
Hence, $\widehat{S}^{\hp \II}_{\lambda}$, is 
usually not an indecomposable object of the category $\Pol^\II_d$.
\smallskip

Now as a superspace $V_n \simeq \k^{n|n}$. 
Let us identify $\End_{\C_1}(V_n)$ as a subset of $\End(V_n)$. 
Then there is a corresponding superalgebra embedding 
$Q(n,d) \hookrightarrow S(n|n,d)$.
For $1\leq i,j \leq n$, we may consider the following elements
of  $\End_{\C_1}(V_n)$: 
$$
E_{0;i,j} := E_{i,j} + E_{n+i,n+j} 
\qquad \text{and} \qquad
E_{1;i,j} := E_{i,n+j} + E_{n+i,j}. 
$$
It follows from \cite[Example 2.6]{Axtell} that the set  
$\{ E_{\varepsilon; i,j}, : \varepsilon\in \Z_2, 1\leq i,j \leq n \}$
is a basis of $\End_{\C_1}(V_n)$.  We order this basis lexicographically:
$E_{0;i,j} < E_{1; k,l}$, for all $1\leq i,j,k,l\leq n$, and 
$E_{\epsilon; i,j} \leq E_{\epsilon; k,l}$ if $(i,j) \preceq (k,l)$.
\smallskip

Let $I(n,d) = I(n|0,d)$.  Given $\varepsilon \in (\Z_2)^d$ and 
$\bs i,\bs j\in I(n,d)$, let $E^{(\varepsilon; \bs i,\bs j)}$ denote
the element of 
$Q(n,d) = \Gamma^d \End_{\C_1}(V_n)$ which is dual to the 
monomial 
$$\check{E}_{\varepsilon_1; i_1, j_1} \cdots 
\check{E}_{\varepsilon_d; i_d, j_d} \hpp \in\hpp 
S^d( \End_{\C_1}(V_n)^*),$$
where  $\{\check{E}_{\varepsilon; i, j}\}$ denotes the basis
of $\End_{\C_1}(V_n)^*$ dual to the basis $\{E_{\varepsilon; i, j}\}$. 
%
If $\varepsilon \in (\Z_2)^d$, then we say a pair $(\bs i,\bs j)
\in I(n,d) \times I(n,d)$ is {\em $\varepsilon$-strict} if 
$(i_k,j_k) \neq (i_l,j_l)$ whenever $\varepsilon_k = \varepsilon_l =1$.
Then $E^{(\varepsilon; \bs i,\bs j)} \neq 0$ if and only if $(\bs i,\bs j)$
are $\varepsilon$-strict.

Notice  there is a unique action of the Hyperoctahedral group, 
$H_d = \Z_2 \wr\Si_d$, on the set $(\Z_2)^d \times I(n,d) \times I(n,d)$,
such that:
$\Si_d$ acts on $I(n,d)\times I(n,d)$ via the diagonal action as usual, 
and $(\Z_2)^d \subset H_n$ acts on itself via point-wise addition.  
Let $\Omega(n,d)$ denote a set of $H_n$-orbit representatives in 
$(\Z_2)^d \times I(n,d) \times I(n,d)$. Then
it is not difficult to see that the set 
$$
\{ E^{(\varepsilon; \bs i,\bs j)} : (\varepsilon; \bs i, \bs j) \in \Omega(n,d)
\text{ and } (\bs i, \bs j)  \text{ is } \epsilon\text{-strict}\}$$
is a basis of $Q(n,d)$.

Next, we consider weight idempotents for $Q(n,d)$.  Let us write 
$\Lambda(n) = \Lambda(n|0)$ and $\Lambda(n,d) = \Lambda(n|0,d)$. Denote 
$E^{(0;\nu)} = E^{(0;\bs i^\nu, \bs i^\nu)}\in Q(n,d)$ for any $\nu \in \Lambda(n,d)$,
where $0 = (0, \dots, 0)\in (\Z_2)^d$.
We then have the following, which is proved in \cite[Lem.~6.1]{BKprojective} 
using different notation.

\begin{lemma}
Suppose $\bs i, \bs j \in I(n,d)$, $\varepsilon\in (\Z_2)^d$ and $\nu\in 
 \Lambda(n,d)$.  Then 
\begin{align*}
E^{(0;\nu)}\circ E^{(\varepsilon; \bs i, \bs j)} = 
\delta_{\nu, wt(\bs i)}  E^{(\varepsilon; \bs i, \bs j)},
\quad 
E^{(\varepsilon; \bs i, \bs j)} \circ E^{(0;\nu)} =
\delta_{wt(\bs j), \nu} E^{(\varepsilon; \bs i, \bs j)}.
\end{align*}
In particular, $\{E^{(0;\nu)}:\ \nu\in  \Lambda(n,d)\}$ gives a  set of mutually
orthogonal even idempotents whose sum is the identity of $Q(n,d)$.
\end{lemma}

Thus for any $V\in {Q(n,d)}\hpp \smod$, we again have a decomposition 
\begin{equation}\label{eq:decomp2}
 V= \bigoplus_{\nu\hp \in\hp \Lambda(n,d)} V_\nu,
 \end{equation}
where the weight space $V_\nu =  E^{(0;\nu)} V $.


\subsection{Formal characters}

Let $X(n)$ denote the free polynomial algebra $\Z[x_1, \dots, x_n]$.  
The ring $X(n)$ has the basis of monomials:
$x^\nu = x_1^{\nu_1}\dots x_n^{\nu_n}$, for $\nu \in \Lambda(n)$.
Furthermore, $$X(n)= \bigoplus_{d=0}^\infty X(n,d)$$ is a graded ring,
where the component $X(n,d)$ is spanned by $x^\nu$ such
that $\nu \in \Lambda(n,d)$.

Now let $Z(m|n)$ denote the ring $\Z[x_1,\dots, x_m; y_1 \dots, y_n]
= X(m) \otimes Y(n)$ . 
Then $Z(m|n)$ is also a graded ring 
$$Z(m|n) = \bigoplus_{d=0}^\infty Z(m|n,d),$$ 
and $Z(m|n,d)$ has a $\Z$-basis given by the monomials
$$ z^\mu\hpp =\hpp z_1^{\mu_1} \dots  z_{m+n}^{\mu_{m+n}}, \quad
\text{ where }\quad z_1= x_1,\ \dots\ z_m = x_m,\ z_{m+1}=y_1,\ \dots\ z_{m+n}=y_n, $$
such that $\mu \in \Lambda(m|n,d)$. 
%

\begin{definition}
Suppose that $V \in S(m|n,d)\hpp \smod$ and $W\in Q(n,d)\hpp \smod$.  
Recalling the weight space decompositions (\ref{eq:decomp1}) and 
(\ref{eq:decomp2}), 
we define the {\em formal characters} 
$$
\text{ch}(V) = \sum_{\mu\in  \Lambda(m|n,d)}\ \dim\hp (V_\mu)\hpp z^\mu
\quad
\text{and}\quad \text{ch}(W) = \sum_{\nu\in  \Lambda(n,d)}\ \dim\hp (W_\nu)\hpp x^\nu,
$$
which belong to $Z(m|n)$ and $X(n)$, respectively.
\end{definition}

We now consider the weight spaces of the $S(m|n,d)$-supermodule $\Gamma_\Pi^{\lambda'}(M)$,
with $M=\k^{m|n}$. 
Let us define the weight of a tableau $\mathfrak t$ to be the
weight of its associated word; i.e., $wt(\mathfrak t) = wt(w(\mathfrak t))$. 
Using Lemma \ref{lem:multiply3} it can be checked that 
$$ E^{(\mu)} Z_\Pi^{(\mathfrak t)} = \delta_{wt(\mathfrak t),\mu} Z_\Pi^{(\mathfrak t)}.$$
It follows that the weight space
$(\Gamma_\Pi^{\lambda'}M)_\mu$ has basis given by the set of all 
$Z_\Pi^{(\mathfrak t)}$ such that $\mathfrak t$ is row costandard and $wt(\mathfrak t) = \mu$.
The formal character as an $S(m|n,d)$-supermodule is then
equal to
$$
\text{ch}(\Gamma_\Pi^{\lambda'}M) 
= \sum_{\substack{ \text{row costandard}\\[.05cm] \mathfrak t\hp \in\hp \text{Tab}_{m|n}(\lambda') }} z^{wt(\mathfrak t)} 
\hpp =\hpp  \sum_{\substack{ \text{row standard}\\[.05cm]\mathfrak s\hp \in\hp \text{Tab}_{m|n}(\lambda)}} z^{wt(\mathfrak s)},
$$
since $\mathfrak t$ is costandard if and only if its conjugate, $\mathfrak s = \mathfrak t'$, is standard.  (Recall
that $\mathfrak t'(i,j) = \mathfrak t(j,i)$ for all $(i,j)\in\triangle_{\lambda'}$.) 
It then follows from Theorm \ref{thm:standard} that
\begin{equation}\label{eq:char}
\text{ch}(\widehat{S}_{\lambda}(\k^{m|n})) 
= \sum_{\substack{ \text{costandard}\\[.07cm] \mathfrak t\hp \in\hp \text{Tab}_{m|n}(\lambda')}} z^{wt(\mathfrak t)} 
\hpp =\hpp  \sum_{\substack{ \text{standard}\\[.07cm] \mathfrak s\hp \in\hp \text{Tab}_{m|n}(\lambda)}} z^{wt(\mathfrak s)}.
\end{equation}

Suppose given partitions $\mu \subset \lambda$.  
Then recall from \cite[ch.~I]{Mac} the definition of the {\em skew Schur function},
$s_{\lambda/\mu}(x_1, \dots, x_n)$, which belongs to the ring 
$\mathsf{\Lambda}_n(x) =X(n)^{\Si_n}$ of symmetric functions in $n$ variables. 
One may check using \cite[I.5.12]{Mac} that we have
\begin{equation}\label{eq:Schur}
s_{\lambda/\mu}(x_1, \dots, x_m)
\hpp =\hpp  \sum_{\substack{ \text{standard}\\[.07cm] \mathfrak t\hp \in\hp \text{Tab}_{m|0}(\lambda/\mu)}} z^{wt(\mathfrak t)},
\qquad
 s_{\lambda'/\mu'}(y_1, \dots, y_n)
\hpp =\hpp  \sum_{\substack{ \text{standard}\\[.07cm] \mathfrak s\hp \in\hp \text{Tab}_{0|n}(\lambda/\mu)}} z^{wt(\mathfrak s)}.
\end{equation}
The {\em hook Schur function} is defined in \cite[6.3]{BR} to be the sum
$$
hs_\lambda(x_1, \dots, x_m; y_1, \dots, y_n) = 
\sum_{\mu\subset \lambda} s_\mu(x_1, \dots, x_m) s_{\lambda'/\mu'}(y_1, \dots, y_n),
$$
which may be considered as an element of the ring
$\mathsf{\Lambda}_m(x) \otimes_\Z \mathsf{\Lambda}_n(y) \subset Z(m|n)$.  

\begin{proposition}\label{prop:char}
The character of the Schur supermodule $\widehat{S}_\lambda(\k^{m|n})$  
is equal to the hook Schur function
$hs_\lambda(x_1, \dots, x_m; y_1, \dots, y_n)$.
\end{proposition}

\begin{proof}
Any standard tableaux $\mathfrak t \in \text{Tab}_{m|n}(\lambda)$ may be decomposed
as a double tableau $\mathfrak t(1)| \mathfrak t(2)$ for standard $\mathfrak t(1) \in \text{Tab}_{m|0}(\mu)$
and $\mathfrak t(2) \in \text{Tab}_{0|n}(\lambda/\mu)$ for some $\mu\subset \lambda$.
Hence, the result follows from (\ref{eq:char}) and (\ref{eq:Schur}).
\end{proof}

Recall from \cite[Ch.~III]{Mac}, \cite{WanWang} the {\em Hall-Littlewood symmetric function}, 
$$S_\lambda(x_1, \dots, x_n) = hs_\lambda(x_1, \dots, x_n;x_1, \dots, x_n)
=\sum_{\mu\subset \lambda} s_\mu(x_1, \dots, x_m) s_{\lambda'/\mu'}(x_1, \dots, x_n),$$
defined for $\lambda \in \Lambda(n)$.
We then have the following.

\begin{corollary}\label{cor:char}
The character of the $Q(n,d)$-supermodule $\widehat{S}^{\hp \II}_\lambda(V_n)$ is equal
to the Hall-Littlewood symmetric function $S_\lambda(x_1, \dots, x_n)$.
\end{corollary}

\begin{proof}
Given any $\mu\in \Lambda(m|n)$,
let us write 
$\mu^+ = (\mu_1,\dots, \mu_m)\in\Lambda(m)$ and
  $\mu^-= (\mu_{m+1}, \dots, \mu_{m+n})\in\Lambda(n)$.  
Then note that the embedding $Q(n,d) \hookrightarrow S(n|n,d)$ mentioned above
sends 
$$E^{(0;\nu)} \mapsto 
\sum_{\substack{\mu\hp \in\hp \Lambda(n|n,d):\hp \mu^+ + \mu^- = \nu}} E^{(\mu)}.$$
Now any $V\in S(n|n,d)\hpp \smod$ may be considered as a $Q(n,d)$-supermodule
by restriction.  In particular, for any $\nu \in \Lambda(n,d)$ we have 
$$
\dim\hp ( \text{Res}^{S(n|n,d)}_{Q(n,d)} (V)_\nu)
= \sum_{\substack{\mu\hp \in\hp \Lambda(n|n,d):\hp \mu^+ + \mu^- = \nu}} 
\dim\hp (V_\mu).
$$
Then since $x^\nu = x^{\mu^+} x^{\mu^-}$ whenever
$\mu^+ + \mu^- = \nu$, it follows that  
$
\text{ch}(\text{Res}^{S(n|n,d)}_{Q(n,d)} (V)) 
= \text{ch}(V)|_{x=y}.
$
The result now follows from Proposition \ref{prop:char} and the definitions
since
$$\widehat{S}^{\hp \II}_\lambda(V_n) = \text{Res}^{S(n|n,d)}_{Q(n,d)} 
\widehat{S}_\lambda(\k^{n|n})$$
as a $Q(n,d)$-supermodule.
\end{proof}

Let us briefly describe formal characters of polynomial superfunctors.
Let $m'\geq m$, $n'\geq n$ be nonnegative integers. Then there are
 surjective ring homomorphisms
$$\rho^{n'}_n: X(n') \to X(n),\qquad \text{resp.}\quad 
\rho^{m'|n'}_{m|n}: Z(m|n') \to Z(m|n),$$
which send the variables $x_{n+1}, \dots, x_{n'}$, 
resp.~$x_{m+1}, \dots,x_{m'}, y_{n+1}, \dots, y_{n'}$,  to zero 
and which leave other variables fixed.
These  maps restrict to give surjective maps of abelian groups: 
$$\rho^{n'}_n(d): X(n',d)\to X(n,d), \qquad  \quad
\rho^{m'|n'}_{m|n}(d): Z(m'|n',d)\to Z(m|n,d).
$$
We may define inverse limits,
$$
X(\infty,d) = \varprojlim X(n,d) \quad \text{and} \quad 
Z(\infty|\infty,d) = \varprojlim Z(m|n,d),
$$
with respect to these maps.  
We then have corresponding graded rings,
$$X(\infty) = \bigoplus_{d=0}^\infty X(\infty,d), \quad \text{and} \quad 
Z(\infty|\infty) = \bigoplus_{d=0}^\infty Z(\infty|\infty,d),$$ 
with multiplication induced from the inverse limits.
\smallskip

Notice that  $X(\infty)$ and $Z(\infty|\infty)$ may both be considered 
as subrings of the ring 
$\Z [[ x;y ]] = \Z [[ x_1, x_2, \dots; y_1, y_2, \dots ]]$ 
of formal power series in infinitely many variables.
Explicitly, we may  identify an element 
$$(f_n)_{n\geq 0}\hpp \in\hpp X(\infty,d), \qquad 
\text{resp.}\quad  (g_{m|n})_{m,n\geq 0}\hpp \in\hpp Z(\infty|\infty,d),$$
with the
unique power series $f(x)$, resp.~$g(x;y)$, such that 
$f_n(x_1, \dots, x_n) = f(x_1, \dots, x_n, 0, \dots)$
and $g_{m|n}(x_1, \dots, x_n; y_1, \dots, y_n) = g(x_1, \dots, x_m, 0, \dots; y_1, \dots, y_n, 0, \dots)$
for all $m,n\geq 0$. 
In this case, we write 
$$f= \varprojlim f_n,\quad \qquad g= \varprojlim g_{m|n}.$$  
For example, we may consider inverse limits of the Hall-Littlewood and hook Schur functions: 
$$ 
S_\lambda(x) = \varprojlim S_\lambda(x_1, \dots, x_n),\qquad \quad 
hs_\lambda(x;y) = \varprojlim hs_\lambda(x_1, \dots, x_m; y_1, \dots, y_n).
$$ 
Note that we have: 
$S_\lambda(x;y) = hs_\lambda(x;x)$, and $hs_\lambda(x;y) = \sum_{\mu\subset \lambda}s_\mu(x) s_{\lambda'/\mu'}(y)$,
where $s_\lambda(x)$ is the Schur function
$s_\lambda(x) = \varprojlim s_\lambda(x_1, \dots, x_n) \in X(\infty)$.
\smallskip

We remark that it is possible to show that
$$\rho^{m'|n'}_{m|n}(\text{ch}(T(\k^{m'|n'})) = \text{ch}(T(\k^{m|n}))), \qquad
\text{and}\quad
\rho^{n'}_{n}(\text{ch}(T(V_{n'})) = \text{ch}(T(V_n)))$$
for all $m'\geq m$, $n'\geq n$.
Hence, one may make the following definition.
\begin{definition}
The {\em formal character} of a polynomial superfunctor $T\in \Pol^\I_d$, 
resp. $T\in \Pol^\II_d$, is defined to be the element
$$
\text{ch}^\I(T) = \varprojlim \text{ch}(T(\k^{m|n})), \qquad \text{resp.}\quad 
\text{ch}^\II(T) = \varprojlim \text{ch}(T(V_n)),
$$
of $X(\infty,d)$, resp.~$Z(\infty|\infty, d)$.
\end{definition}

The Schur superfunctors then have the following formal characters:
\begin{equation}\label{eq:Schur2}
\text{ch}^\I(\widehat{S}_\lambda^\I) = hs_\lambda(x;y) \qquad \text{and}\qquad
\text{ch}^\II(\widehat{S}_\lambda^\II) = S_\lambda(x),
\end{equation}
which follows from the definitions, Proposition \ref{prop:char} and Corollary \ref{cor:char}.
%


\end{document}